\documentclass{article}
\usepackage{amsmath, amsthm, amsfonts, mathabx, bbm, graphicx, subcaption, tikz-cd, mathtools, xcolor, float}

\usepackage[colorlinks, allcolors=blue]{hyperref}
\usepackage[noabbrev, nameinlink]{cleveref}

\usepackage[numbers,sort]{natbib}
\usepackage{thmtools} 
\usepackage{etoolbox}
\usepackage[margin=1in]{geometry}

\DeclareMathAlphabet{\mymathbb}{U}{BOONDOX-ds}{m}{n}



\makeatletter
\def\smalloverbrace#1{\mathop{\vbox{\m@th\ialign{##\crcr\noalign{\kern3\p@}%
\tiny\downbracefill\crcr\noalign{\kern3\p@\nointerlineskip}%
$\hfil\displaystyle{#1}\hfil$\crcr}}}\limits}
\makeatother

\newcommand{\downset}[2][ ]{%
\downarrow \! \{ #2 \}_{#1}
}

\newcommand{\upset}[2][ ]{%
\uparrow \! \{ #2 \}_{#1}
}

\DeclareSymbolFont{extraup}{U}{zavm}{m}{n}
\DeclareMathSymbol{\varspade}{\mathalpha}{extraup}{85}
\DeclareMathSymbol{\varheart}{\mathalpha}{extraup}{86}
\DeclareMathSymbol{\vardiamond}{\mathalpha}{extraup}{87}
\DeclareMathSymbol{\varclub}{\mathalpha}{extraup}{88}

\newcommand{\ffrac}[2]{\ensuremath{\frac{\displaystyle #1}{\displaystyle #2}}}
\newcommand{\unsim}{\mathord{\sim}}

\newcommand{\lav}{\left\vert}      
\newcommand{\rav}{\right\vert}     

\newcommand{\ev}[1]{( #1 )} 
\newcommand{\ew}[1]{\left( #1 \right)}

\newcommand{\Z}{\mathbb{Z}}
\newcommand{\C}{\mathbb{C}}

\newtheorem{theorem}{Theorem}[section]
\newtheorem{lemma}[theorem]{Lemma}
\newtheorem{corollary}[theorem]{Corollary}
\newtheorem{proposition}[theorem]{Proposition}

\numberwithin{equation}{section}

\theoremstyle{definition}
\newtheorem{remark}{Remark}
\newtheorem{definition}{Definition}
\newtheorem{example}{Example}

\title{Characteristic Polynomials and Hypergraph Generating Functions via Heaps of Pieces}

\usepackage{authblk}

\author[1]{Joshua Cooper}
\author[2]{Krystal Guo}
\author[3]{Utku Okur}
\affil[1]{Department of Mathematics, University of South Carolina, U.S.A. (\href{mailto:email}{cooper@math.sc.edu})}
\affil[2]{Korteweg-de Vries Institute for Mathematics, University of Amsterdam, The Netherlands. (\href{mailto:email}{k.guo@uva.nl})}
\affil[3]{Department of Mathematics, University of South Carolina, U.S.A. (\href{mailto:email}{uokur@email.sc.edu})}

\date{\today}

\let\underbrace\LaTeXunderbrace

\tikzset
{
graph_1/.pic=
{%
\node[fill = black, circle, inner sep=1pt, label = below:$1$ ] (c1) at (-2,0) {};     
\node[fill = black, circle, inner sep=1pt, label = below:$2$] (c2) at (0,0) {};  
\node[fill = black, circle, inner sep=1pt, label = below:$3$] (c3) at (2,0) {};          

\draw (c3) -- node[midway, yshift=-8pt	]{$f$} (c2);
\draw (c2) -- node[midway, yshift=-6pt	]{$e$} (c1);

}
}

\tikzset
{
graph_2/.pic=
{%
\node[fill = black, circle, inner sep=1pt, label = $1$ ] (c1) at (1.7,1) {};     
\node[fill = black, circle, inner sep=1pt, label = below: $2$] (c2) at (0,0) {};  
\node[fill = black, circle, inner sep=1pt, label = $3$] (c3) at (0,2) {};  
\node[fill = black, circle, inner sep=1pt, label = $4$] (c4) at (3.7,1) {};                

\draw (c3) -- node[midway, left=0.7pt	]{$f$} (c2);
\draw (c2) -- node[midway, xshift=4pt, yshift=-6pt	]{$e$} (c1);
\draw (c3) -- node[midway, xshift=4pt, yshift=6pt	]{$g$} (c1);
\draw (c1) -- node[midway, yshift=5pt	]{$a$} (c4);

}
}

\begin{document}
\maketitle

\begin{abstract}

It is a classical result due to Jacobi in algebraic combinatorics that the generating function of closed walks at a vertex $u$ in a graph $G$ is determined by the rational function 
\[
\frac{\phi_{G-u}(t)}{\phi_G(t)} 
\]
where $\phi_G(t)$ is the characteristic polynomial of $G$. In this paper, we show that the corresponding rational function for a hypergraph is also a generating function for some combinatorial objects in the hypergraph. 

We make use of the Heaps of Pieces framework, developed by Viennot, demonstrating its use on graphs, digraphs, and multigraphs before using it on hypergraphs. In the case of a graph $G$, the pieces are cycles and the concurrence relation is sharing a vertex. The pyramids with maximal piece containing a vertex $u \in V(G)$ are in one-to-one correspondence with closed walks at $u$. In the case of a hypergraph $\mathcal{H}$, connected ``infragraphs'' can be defined as the set of pieces, with the same concurrence relation: sharing a vertex.  Our main results are established by analyzing multivariate resultants of polynomial systems associated to adjacency hypermatrices.

\textit{Keywords: hypergraphs, characteristic polynomial, generating functions} 

\textit{Mathematics Subject Classifications 2020: Primary 05C65; Secondary 05C50, 05A15, 15A69.} 
\end{abstract}

\hypersetup{
linkcolor=black
} 

\tableofcontents

\hypersetup{
linkcolor=blue
} 

\section{Introduction}

Given a simple graph $G$, let $\phi_G(t)$ be the characteristic polynomial of its adjacency matrix $\mathbb{A}_G$. For a vertex $u\in V\ev{G}$, a well-known classical identity (\cite[Equation (1), p.~52]{godsil}) is,
\begin{equation}
\sum_{d\geq 0} w_d^{u} t^{d} = \frac{t^{-1}\phi_{G-u}(t^{-1})}{\phi_G(t^{-1})}
\label{eqn:origin}
\end{equation}
where the $d$-th coefficient $w_d^{u}$ is the number of closed walks of length $d$ at $u$, and $G-u$ is the \textit{vertex-deleted subgraph} of $G$, obtained by deleting the vertex $u$, together with all the edges incident with $u$. This result is attributed to Jacobi in 1833 (\cite{Jacobi1834}, see also \cite{godsil, Godsil_1992}). Note that the rational function on the right-hand side is equal to the generating function on the left-hand side, with non-negative integer coefficients. Thus, the function in question is an example of a broader class of functions of interest, the \textit{rational generating functions}, investigated in \cite{melou}.

In this paper, we explicate a proof of equation (\ref{eqn:origin}) (see \Cref{thm:walks_at_u_and_quotient} below), using the application of the Heaps of Pieces framework on the walks of a simple graph, and later extend this approach to the setting of a hypergraph. 

The theory of Heaps of Pieces is developed by Viennot (see \cite[p.~36]{viennot_french_paper}, \cite{viennot_talk, abdesselam_brydges, melou, krattenthaler}). In the case of a finite simple graph $G$, one takes cycles of $G$ as pieces, and sharing a vertex as the concurrence relation. In Section \ref{sec:bijections_pyramids_walks}, this approach is demonstrated and and laid out in detail, both for completeness and for reference. In particular, \Cref{cor:bij_walks_pyramids} establishes the bijection between closed walks of $G$ at a fixed vertex $u$, and the pyramids with maximal piece containing $u$. 

Afterwards, in Section \ref{sec:heap_hyper}, we turn to hypergraphs of rank $k\geq 3$ and inquire into how much of the above theory generalizes to hypergraphs. For a given uniform hypergraph $\mathcal{H}$ of order $n$ and rank $k\geq 3$ (i.e., there are $n$ vertices and each edge has $k$ vertices), the adjacency hypermatrix $\mathbb{A}_\mathcal{H}$ is the $n$-dimensional symmetric $k$-hypermatrix with entries 
$$
a_{i_1 \ldots i_k} =  \begin{cases}  \ffrac{ 1 }{(k-1)!} & \text{ if } \{i_1,\ldots,i_k\} \text{ is an edge} \\ 0 & \text{otherwise}
\end{cases} 
$$
The characteristic polynomial $\phi_\mathcal{H}\ev{t}$ of a hypergraph $\mathcal{H}$ is the multivariate resultant 
$$ 
\text{res}  \nabla  \ev{ t\cdot x_1^{k} + \ldots + t\cdot x_n^{k} - F_\mathcal{H}\ev{ x_1, \ldots, x_n }}
$$
where $\nabla$ is the gradient operator and $F_\mathcal{H}$ is the Lagrangian of $\mathcal{H}$ (cf. \Cref{def:char_poly}). In \cite{cc1}, coefficients of the characteristic polynomial of a hypergraph is expressed in terms of the counts of its ``infragraphs" (cf. \Cref{def:veblen_multi_hypergraph}). In this paper, we apply the Heaps of Pieces framework, by taking infragraphs as the set of pieces, and sharing a vertex as the concurrence relation. 

In particular, we express $t^{-N} \cdot \phi_\mathcal{H}\ev{t}$ as a sum over ``trivial" heaps, where $N:= \deg\ev{\phi_\mathcal{H}\ev{t}} = n\ev{k-1}^{n-1}$ is the degree of the characteristic polynomial. The infinite Laurent series $ - \log\ev{ t^{-N} \cdot \phi_\mathcal{H}\ev{t} }$ is expressed as a sum over ``pyramids", and the coefficients of the quotient 
$$
\dfrac{t^{-\deg\ev{ \phi_{\mathcal{H}-u}\ev{t} } } \cdot \phi_{\mathcal{H}-u}\ev{t} }{ t^{-\deg\ev{ \phi_{\mathcal{H}}\ev{t} } } \cdot \phi_{\mathcal{H}}\ev{t} } 
$$ 
are given by pyramids with maximal piece containing the vertex $u$. For this purpose, we note that each piece $\beta$ is allowed to have a weight $w\ev{\beta} \in R$, where $R$ is any commutative ring with unity, as long as the weight of a heap is the product of the weights of the pieces, i.e., the weight function is required to be multiplicative (see \Cref{thm:viennot_krattenthaler}). We define the weight function $w_n$ on each infragraph $X$ of a hypergraph $\mathcal{H}$ as follows (cf. Definition \ref{def:parallel_equiv} and Section \ref{subsec:kocay_lemma}):
$$ w_n\ev{X} = \sum_{ \mathbf{S} \in \widetilde{\mathcal{S}} \ev{ X }} \ev{ -\ev{k-1}^{n} }^{ c\ev{ \mathbf{S} } } \dfrac{1}{\alpha_{\mathbf{S}}} C_{\mathbf{S} } $$

When two infragraphs are disconnected, we show that the weight of the disjoint union is the product of the weights of the components, so the multiplicativity condition is satisfied. This setting generalizes the Heaps of Pieces framework outlined above for graphs of rank $2$ in the following sense: The weight of a connected infragraph $X$ of a graph $G$ of rank $2$ is non-zero only if $X$ is indecomposable, i.e., it is a cycle or an edge (see \cite{cc2}).

Although connected infragraphs can be taken as the pieces, the question remains open, whether there is a smaller set of pieces, from which we can obtain the weight of any infragraph as a product of those pieces. 

The Heaps of Pieces framework provides a geometric and intuitive way to think about some types of combinatorial objects and has proven useful in many enumeration problems. A non-exhaustive list of the applications of this theory includes: counting parallelogram polyominoes (\cite{krattenthaler,melou_polyomino}, \cite[A006958]{oeis}), counting directed animals of size $n$ (or directed $n$-ominoes in standard position) (\cite[p.~359]{goodman},\cite[p.~80]{flajolet},\cite{melou,viennot_beauchamps},\cite[A005773]{oeis}), the matching polynomial (\cite[Example 5.7, p.~336]{viennot}). The connections between walks of a simple digraph, heaps of pieces and words of a regular language have been established in \cite{melou}. Another interesting fact is the isomorphism between the heap monoid and the trace monoid (or a partially commutative monoid) (\cite{samy_juge,krattenthaler}). 

In the appendix, we collect some of the routine proofs for reference. 

\section{Heaps of Pieces and the Characteristic Polynomial of a Graph}
\subsection{Preliminaries}
\label{sec:preliminaries_for_graphs}

In this paper, we use $A\sqcup B$ to denote the union of two disjoint sets $A,B$. The \textit{Kleene star closure} of a set $A$ is defined as an infinite union of finite tuples:
$$
A^{*} = \bigsqcup_{d\geq 0} A^d 
$$
where $A^d = \underbrace{A\times \ldots \times A}_{d \text{ many}} = \{ \ev{ a_1,\ldots, a_d} : a_i \in A, \text{ for each } i = 1,\ldots,d \}$ for $d\geq 1 $ and $A^0 = \{ \emptyset \}$. 

Next, we define a \textit{multi-(hyper)graph of rank $k\geq 2$}, by a slight modification of the definition of a graph given in \cite[p.~2]{bondy}. A \textit{multi-hypergraph} $X$ of rank $k\geq 2$ is a triple $X = \ev{ V\ev{X}, E\ev{X}, \varphi }$, where $V\ev{X}, E\ev{X}$ are finite sets, and $\varphi: E\ev{X} \rightarrow \binom{V\ev{X}}{k}$ is a mapping from the edge-set into $k$-subsets of the vertex-set, such that the coordinates of $\varphi\ev{e}$ are pairwise distinct, for each $e\in E\ev{X}$ (there are no degenerate edges or loops.) The \textit{order} of $X$ is $|V\ev{X}|$. The multi-hypergraph $X$ is called a \textit{multi-graph} if it is of rank $k=2$. 
\begin{enumerate}
\item[\textit{i)}] If $X = \ev{ V\ev{X}, E\ev{X}, \varphi }$ is given such that $\varphi$ is injective, then $X$ is a \textit{simple} hypergraph. If $\varphi$ is surjective, then $X$ is a \textit{complete} hypergraph.
\item[\textit{ii)}] Two edges $e_1, e_2 \in E\ev{X}$ are \textit{parallel}, denoted $e_1 \sim e_2$, provided $\varphi\ev{e_1} = \varphi\ev{e_2}$. Given a multi-hypergraph $X$ and an edge $e\in E\ev{X}$, then the \textit{multiplicity} $m_X\ev{e}$ of $e$ is the cardinality $| [e]_{\unsim} | = | \{ f\in E\ev{X}: \varphi\ev{f} = \varphi\ev{e} \} | $ of the equivalence class of edges parallel to $e$. For each $v_1,\ldots,v_k \in V\ev{X}$, we write $m_X\ev{ \{ v_1, \ldots, v_k \} } = | \{ e\in E\ev{X} : \varphi\ev{e} = \{ v_1, \ldots, v_k \} \} | $.
\item[\textit{iii)}] We omit $\varphi$, whenever there is no confusion to arise. For example, given an edge $e$ and a vertex $u$, we write $u\in e$ instead of $u\in \varphi\ev{e}$ or $u\in \varphi\ev{[e]_{\unsim}}$. We define $\deg_X\ev{u} = |\{ e\in E\ev{X}: u\in e\} | $. 
\item[\textit{iv)}] The product of factorials of the multiplicities of the edges of $X$ is denoted as 
$$M_X := \prod_{ [e]_{\unsim} \in E\ev{X} / \unsim } m_X\ev{e} ! $$
\end{enumerate}

A \textit{multi-digraph}, or shortly, a \textit{digraph} $D=\ev{ V\ev{D}, E\ev{D}, \psi}$ (of rank $2$) is a triple, where $V\ev{D}, E\ev{D}$ are finite sets, and $\psi: E\ev{D} \rightarrow V\ev{D} \times V\ev{D}$ is a function, such that the coordinates of $\psi\ev{e}$ are distinct, for each edge $e\in E\ev{D}$. (There are no loops.) 
\begin{enumerate}
\item[\textit{i)}] Two edges $e_1, e_2 \in E\ev{D}$ are \textit{parallel}, denoted $e_1 \unsim e_2$, provided $\psi\ev{e_1} = \psi\ev{e_2}$. (We use the same notation for the equivalence relation of parallelism in a digraph.) 
\item[\textit{ii)}] Given a digraph $D$ and an edge $e\in E\ev{D}$, then the \textit{multiplicity} $m_D\ev{e}$ of $e$ is the cardinality of the set $[e]_{\unsim} = \{ f\in E\ev{D}: \psi\ev{f} = \psi\ev{e} \}$ of edges parallel to $e$. Also, we define $m_D\ev{u,v} = | \{ e\in E\ev{D} : \psi\ev{e} = \ev{u,v} \}|$, for each $u,v\in V\ev{D}$. 
\item[\textit{iii)}] Whenever there is no confusion to arise, we omit $\psi$. For example, if there is a unique edge $e\in E\ev{D}$ with $\psi\ev{ e } = \ev{ u,v }$, then we will shortly write $(u,v)$ instead of $e$. The in-degree of a vertex $u$ is defined as $\deg_D^{-}\ev{u} = \sum_{v\in V\ev{D}} m\ev{ v,u }$. 
\item[\textit{iv)}] Given an undirected multi-graph $X=(V(X),E(X),\varphi)$ (of rank $2$) and an digraph $O=(V(O),E(O),\psi)$, then $O$ is an \textit{orientation} of $X$, provided $V\ev{O} = V\ev{X}$ and $E\ev{O} = E\ev{X}$ and we have $\varphi\ev{ e } = \theta\ev{ \psi\ev{ e } }$ for each $e\in E\ev{X} = E\ev{O}$, where 
\begin{align*}
\theta: V\ev{X} \times V\ev{X} &\rightarrow \binom{V\ev{X}}{2} \\ 
\ev{ v_1, v_2} & \mapsto \{ v_1, v_2\}
\end{align*}
is the ``forgetful" mapping. The set of orientations of $X$ is denoted as $\mathcal{O}\ev{X}$. So, we have $|\mathcal{O}\ev{X}| = 2^{|E\ev{X}|}$. 
\end{enumerate}

Given two undirected multi-hypergraphs $X = \ev{ V, E, \varphi}$ and $X' = \ev{V', E', \varphi'}$, then an \textit{isomorphism from $X$ to $Y$} is a pair of bijections
$$ \ev{ \eta: V \rightarrow V', \xi: E \rightarrow E' } $$
such that $\varphi' \ev{ \xi \ev{ e } } = \eta \ev{ \varphi \ev{ e } }$ for each $e\in E$. We write $X \cong Y$, if $X$ and $Y$ are isomorphic, i.e., there is at least one isomorphism between them. 

For digraphs $D = \ev{ V, E, \psi}$ and $D' = \ev{V', E', \psi'} $, we define $D\cong D'$ in a similar way: $D \cong D'$ if and only if there are bijections $\ev{ \eta: V \rightarrow V', \xi: E \rightarrow E' } $ satisfying $\psi' \ev{ \xi \ev{ e } } = \eta \ev{ \psi \ev{ e } }$ for each $e\in E$.

\begin{definition}[Vertex-fixing isomorphisms of multi-graphs]
\label{def:approx}
Given two undirected multi-hypergraphs $X = \ev{ V, E, \varphi}$ and $X' = \ev{V', E', \varphi'}$, then we write $X \approx X'$ if $V = V'$ and $ m_X\ev{ \{  v_1,\ldots,v_k \} } = m_{X'}\ev{ \{ v_1,\ldots,v_k \} }$ for each $v_1,\ldots,v_k\in V = V'$. Note that $X \approx X'$ only if there is a vertex-fixing isomorphism $X\rightarrow X'$ that interchanges parallel edges and in particular, $X\cong X'$. 
\end{definition}

Given a multi-hypergraph $X$, then we may obtain a simple hypergraph $\underline{X}$, called the \textit{flattening} of $X$, by removing duplicate edges. More formally,

\begin{definition}[Flattening of a multi-hypergraph]
\label{def:flattening}
\phantom{a}

Given a multi-hypergraph $X = \ev{ V, E, \varphi: E\rightarrow \binom{V}{k} }$, then the \textit{flattening} $\underline{X}$ of $X$ is the simple hypergraph $\underline{X} = \ev{ V, E /\unsim, \widetilde{ \varphi } : E/ \unsim \rightarrow \binom{V}{k} }$, obtained by taking the quotient of the edge set, under the equivalence relation of parallelism, where $\widetilde{\varphi}([e]_{\unsim}) = \varphi(e)$ (which is well-defined, since $\varphi(e)=\varphi(f)$ for all $f \in [e]_{\unsim}$). 
\end{definition}

Given simple hypergraphs $X,Y$, then the number of isomorphic copies of $X$ in $Y$ is denoted as,
$$
\genfrac{\lbrack}{\rbrack}{0pt}{0}{Y}{X} = |\{ Z: Z\text{ is a subgraph of } Y \text{ and } Z \cong X \}|
$$

By Euler's Theorem (\cite[Theorem 3.5]{bondy}), a multi-graph $X$ is Eulerian (in the sense defined in Section \ref{sec:operations}) if and only if it has only even degrees. Now, we define Veblen multi-hypergraphs of rank $k\geq 2$, which specialize to Eulerian multi-graphs when the rank is $k=2$. In the next section, we will see that the Harary-Sachs Theorem for hypergraphs calculates the coefficients of the characteristic polynomial of a hypergraph in terms of Veblen multi-hypergraphs. 
\begin{definition}

\label{def:veblen_multi_hypergraph}

\phantom{a}
Let $k\geq 2$ be fixed. 

\begin{enumerate}
\item[\textit{i)}] A multi-hypergraph $X$ is \textit{$k$-valent} if the degree of each vertex of $X$ is divisible by $k$. 
\item[\textit{ii)}] A $k$-valent multi-hypergraph of rank $k$ is called a \textit{Veblen multi-hypergraph}. 
\end{enumerate}

For a Veblen multi-hypergraph $X$, let $c\ev{X}$ be the number of connected components of $X$. Let $\mathcal{V}^{m}$ denote the set of isomorphism classes of Veblen multi-hypergraphs with $m$ components, where $m\geq 1$. In other words, $\mathcal{V}^{m}$ is the set of equivalence classes $[X]_{ \cong }$ of Veblen multi-hypergraphs $X$ with $m$ components. By an abuse of notation, we will refer to the (unlabelled) elements $[X]_{\cong} \in \mathcal{V}^{m}$ by $X$, omitting the brackets. In particular, $\mathcal{V}^{1}$ is the set of isomorphism classes of connected Veblen hypergraphs. Let $\mathcal{V}^{\infty}$ be the set isomorphism classes of (possibly disconnected) Veblen multi-hypergraphs. In other words, we define
$$ \mathcal{V}^{\infty} = \bigsqcup_{ m = 1 }^{\infty} \mathcal{V}^{m} $$

A $k$-valent multi-hypergraph $X$ is an \textit{infragraph} of a simple hypergraph $\mathcal{H}$, provided $\underline{X}$ is a subgraph of $\mathcal{H}$. Let $ \textup{Inf}^{m}\ev{\mathcal{H}}$ denote the set of infragraphs of $\mathcal{H}$ with $m\geq 1$ components, up to equivalence under $\approx$ (see Definition \ref{def:approx}). In other words, $\textup{Inf}^{m}\ev{\mathcal{H}}$ is the set of equivalence classes $[X]_{ \approx }$ of infragraphs $X$ with $m$ components. By an abuse of notation, we will refer to the elements $[X]_{\approx} \in \textup{Inf}^{m}\ev{\mathcal{H}}$ by $X$, omitting the brackets. In particular, $\textup{Inf}^{1}\ev{\mathcal{H}}$ is the set of isomorphism classes of connected infragraphs of $\mathcal{H}$. We define,
$$
\textup{Inf}\ev{\mathcal{H}} = \bigsqcup_{ m = 1 }^{\infty} \textup{Inf}^{m}\ev{\mathcal{H}}
$$

We further distinguish Veblen multi-hypergraphs and infragraphs in terms of the sizes of their edge-sets: For $d\geq 0$, let $ \mathcal{V}_d^{\infty} = \{ X \in \mathcal{V}^{\infty} : |E\ev{X}| = d \}$ and $ \textup{Inf}_d\ev{\mathcal{H}} := \{ X \in \textup{Inf}\ev{\mathcal{H}}: |E\ev{X}| = d \}$. 
\end{definition}

\subsection{Operations on the Trails of an Eulerian Digraph}
\label{sec:operations}
In this section, we focus our attention on Eulerian multi-graphs and Eulerian digraphs. We have some preliminary definitions:

Let $X = \ev{ V\ev{X}, E\ev{X}, \varphi }$ be a multi-graph of rank $k=2$. A \textit{walk} of $X$ is an alternating sequence \[\mathbf{w}=\ev{v_0,e_1,v_1,e_2,\ldots, v_{d-1}, e_{d}, v_d}\] of vertices $\ev{ v_i }_{i=0}^{d}$ and edges $\ev{ e_i }_{i=1}^{d}$, such that consecutive pairs of elements are incident. 

The \textit{length} of a walk $\mathbf{w}$, denoted $|\mathbf{w}|$, is the number of edges in $\mathbf{w}$. The vertices $v_0,v_d$ are said to be the \textit{initial vertex} and the \textit{terminal vertex} of $\mathbf{w}$, respectively. The \textit{internal vertices} of $\mathbf{w}$ are $\{v_i\}_{0<i<d}$. A \textit{closed walk} at $v_0$ is a walk where $v_0 = v_d$. Given a vertex $v_0 \in V\ev{X}$, then the unique walk of length $0$ at $v_0$ is the walk $\mathbf{w} = \ev{ v_0 }$. 

Let $\mathcal{U} \ev{X}$ and $\mathcal{U}_d \ev{X}$ denote the set of closed walks and the set of closed walks of length $d\geq 0$, respectively. Let $w_d := |\mathcal{U}_d \ev{X}|$ be the number of closed walks of $X$ of length $d\geq 0$. Given a vertex $u\in V\ev{X}$, we denote by $\mathcal{U}^{u}\ev{X}$ and $\mathcal{U}^{u}_d \ev{X}$, the set of closed walks of $X$ at $u$ and the set of closed walks of $X$ at $u$, of length $d\geq 0$, respectively. Let $w_d^{u} := |\mathcal{U}^{u}\ev{X}|$ be the number of closed walks of $X$ of length $d \geq 0$ at $u$. Given an edge $e\in E\ev{X}$ with $\varphi\ev{e} = \{ z_1,z_2 \}$, the set of \textit{closed walks ending at $e$} is defined as,
\begin{align*}
\mathcal{U}^{e}\ev{X} &= \{ \mathbf{w} = \ev{ v_0, e_1, v_1,\ldots, v_{ d-2 }, e_{ d-1 }, z_1, e, z_2 } : \mathbf{w} \in \mathcal{U}\ev{X} \} \\
&\sqcup \{ \mathbf{w} = \ev{ v_0, e_1, v_1,\ldots, v_{ d-2 }, e_{ d-1 }, z_2, e, z_1 } : \mathbf{w} \in \mathcal{U}\ev{X} \}
\end{align*}
In particular, $\{ \mathbf{w} \in \mathcal{U}^{e}\ev{X}: |\mathbf{w}| = 0 \} = \{ \ev{ z_1 }, \ev{ z_2 } \}$ and $\mathcal{U}^{e}\ev{X} \subseteq \mathcal{U}^{z_1}\ev{X} \sqcup \mathcal{U}^{z_2}\ev{X}$.

We denote by $E\ev{ \mathbf{w} } = \{ e_1, \ldots, e_d \}$ the set of edges and by $V\ev{\mathbf{w}}=\{v_0,\ldots,v_d\}$ the set of vertices of a walk $\mathbf{w}$. A walk $\mathbf{w}$ is called a \textit{trail}, if its edges are distinct. A closed trail $\mathbf{w}$ of $X$ is an \textit{Eulerian trail} when $E\ev{\mathbf{w}} =  E\ev{X}$. A \textit{path} is a trail without repeated vertices. A closed trail without repeated internal vertices is said to be a \textit{simple} closed trail. We denote by $\mathcal{T}\ev{X}$ and $\mathcal{W}\ev{X}$, the sets of closed trails and Eulerian (closed) trails of $X$, respectively. Given a vertex $u\in V\ev{X}$, we write $\mathcal{T}^{u}\ev{X}$ for the set of \textit{closed trails of $X$ at $u$} and $\mathcal{W}^{u}\ev{X}$ for the set of \textit{Eulerian trails of $X$ at $u$}. Given an edge $e\in E\ev{X}$, we denote by $ \mathcal{T}^{e}\ev{X}$, the set of \textit{closed trails of $X$ ending at $e$} and by $\mathcal{W}^{e}\ev{X}$, the set of \textit{Eulerian trails of $X$ ending at $e$}.

A walk of a digraph is defined similarly: A walk $\mathbf{w}=\ev{v_0,e_1,v_1,e_2,\ldots, v_{d-1}, e_{d}, v_d}$ must satisfy $\psi\ev{ e_i } = \ev{ v_{i-1}, v_{i} }$, for each $i = 1,\ldots, d$. One fine point is that, for a digraph $D$ and an edge $e\in E\ev{D}$ with $\psi\ev{e} = \ev{z_1,z_2}$, we have $\{ \mathbf{w} \in \mathcal{U}^{e}\ev{X}: |\mathbf{w}| = 0 \} = \{ \ev{ z_2 } \}$ and $\mathcal{U}^{e}\ev{D} \subseteq \mathcal{U}^{z_2}\ev{X} $. For convenience of notation, we sometimes omit vertices of a sequence defining a walk of $D$. 

\begin{remark}
Equivalently, a \textit{walk} of a multi-graph $X$ is a pair of functions $\ev{ v: [0,d] \rightarrow V\ev{X}, e: [1,d] \rightarrow E\ev{ X } }$ such that $\varphi\ev{ e_i } = \{ v_{i-1}, v_i\}$ for each $i=1,\ldots,d$. More generally, we will use any strictly increasing sequence $\{ j_i \}_{i=0}^{d}$ of positive integers as the index set, provided $\varphi\ev{ e_{j_i} } = \{ v_{j_{i-1}}, v_{j_i}\}$ is satisfied, for each $i=1,\ldots,d$.
\end{remark}

Now, we can restate Euler's Theorem (\cite[Theorem 3.5]{bondy}), with the new notation: A multi-graph $X$ is Eulerian, i.e. $\mathcal{W}\ev{X} \neq \emptyset$ if and only if it is a Veblen multi-graph (cf. Definition \ref{def:veblen_multi_hypergraph}).

Let $X$ be a multi-graph of rank $k=2$. We define the set of \textit{circuits} of $X$ as the quotient $\mathcal{Z}\ev{X} := \mathcal{T}\ev{X} /\rho$ of closed trails under the equivalence relation $\rho$, which is defined by cyclic permutation of the edges. More formally, $\rho$ is the transitive closure of the relation $\rho'$, defined by:
$$ \mathbf{w}_1 = \ev{ v_0, e_1, v_1, e_2,\ldots, e_{d-1}, v_{d-1}, e_{d}, v_0 } \ \ \rho' \ \  \mathbf{w}_2 = \ev{ v_{d-1}, e_{d}, v_0, e_1, v_1,e_2,\ldots, e_{d-1} ,v_{d-1} } $$

The set of \textit{Eulerian circuits} is defined as $\mathfrak{C}\ev{ X } := \mathcal{W}\ev{ X } / \rho$, the quotient of the set of Eulerian trails of $X$, up to the equivalence relation of cyclic permutation of edges. 

Whenever a circuit $[w]_\rho \in \mathcal{Z}\ev{ X } $ has no repeated vertices, we call it a \textit{cycle}. Let $\mathcal{B}\ev{ X }$ denote the set of cycles of a multi-graph $X$. A cycle of length $d=2$ is called a \textit{digon}. 

The circuits and cycles of a digraph is defined similarly. Although $\mathcal{B}\ev{D}$ and $\mathcal{B}\ev{X}$ are sets of different type of objects, it is understood from context whether the input is a multi digraph or an undirected multi-graph, so no confusion arises.

\begin{remark}
\label{rmk:walks_circuits_bij}
For a digraph $D$, there are natural bijections, yielding,
\begin{enumerate}
\item $|\mathfrak{C}\ev{D}| = |\mathcal{W}^{e}\ev{D}|$, for each edge $e\in E\ev{D}$. 
\item $|\mathfrak{C}\ev{D}| \cdot \deg_D^{-}\ev{u} = |\mathcal{W}^{u}\ev{D}|$ for each vertex $u\in V\ev{D}$.
\end{enumerate}
\end{remark}
\begin{example}

\phantom{a}
\begin{figure}[ht]
\centering
\includegraphics[width=0.32\linewidth]{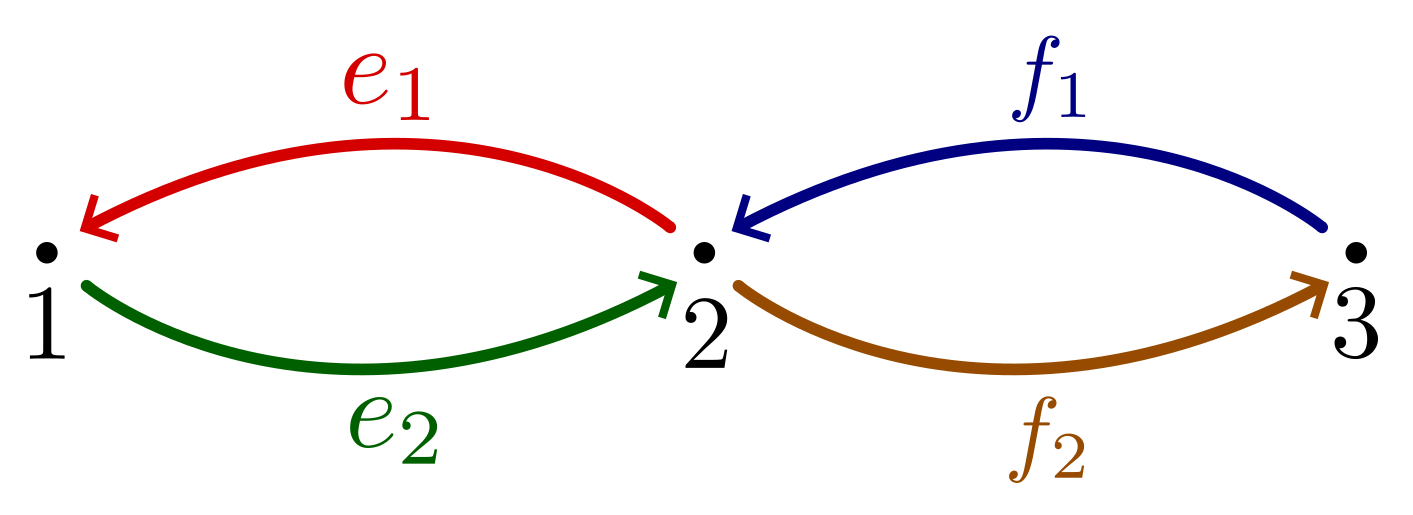}
\caption{A digraph $D$}
\label{fig:digraphexample_1}
\end{figure}
\begin{enumerate}
\item As an example, we consider the digraph in Figure \ref{fig:digraphexample_1}, and list its Eulerian trails, Eulerian circuits and cycles.
\begin{align*} 
& \mathcal{W}\ev{ D } = \{ ({\color{red!80!black} e_1 }, {\color{green!40!black} e_2 }, {\color{brown!50!black} f_2 }, {\color{blue!60!black} f_1 }) , ({\color{brown!50!black} f_2 }, {\color{blue!60!black} f_1 },{\color{red!80!black} e_1 }, {\color{green!40!black} e_2 }), ({\color{green!40!black} e_2 },{\color{brown!50!black} f_2 },{\color{blue!60!black} f_1 },{\color{red!80!black} e_1 }), ({\color{blue!60!black} f_1 },{\color{red!80!black} e_1 },{\color{green!40!black} e_2 },{\color{brown!50!black} f_2 }) \} \\ 
& \mathfrak{C}\ev{ D } = \{ \  [ ({\color{blue!60!black} f_1 } , {\color{red!80!black} e_1 }, {\color{green!40!black} e_2 },  {\color{brown!50!black} f_2 }) ]_\rho \ \} \\
& \mathcal{B}\ev{ D } = \{ \ [ \ev{  {\color{red!80!black} e_1 }, {\color{green!40!black} e_2 } } ]_\rho, [ \ev{ {\color{blue!60!black} f_1 }, {\color{brown!50!black} f_2 } } ]_\rho \ \}
\end{align*}
\item Consider the digraph $D$ in Figure \ref{fig:digraphexample_2}. The Eulerian trails at the vertex $2$, the Eulerian circuits and the cycles are listed below.
\begin{figure}[ht]
\centering
\includegraphics[width=0.19\linewidth]{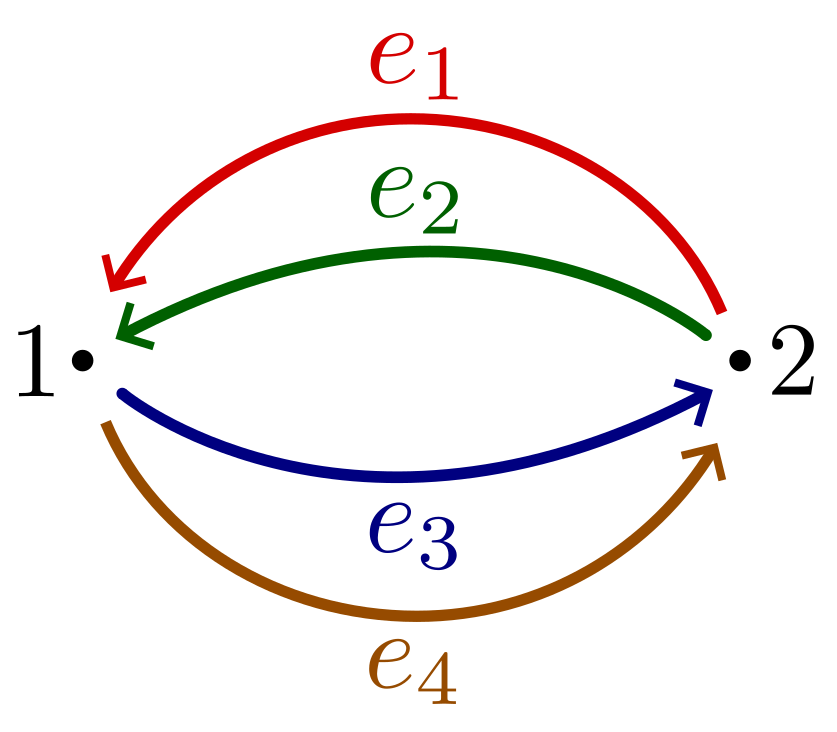}
\caption{A digraph $D$}
\label{fig:digraphexample_2}
\end{figure}
\begin{align*}
& \mathcal{W}^{2}\ev{ D } = \{ 
\ev{ {\color{red!80!black} e_1 }  , {\color{blue!60!black} e_3 } , {\color{green!40!black} e_2 },  {\color{brown!50!black} e_4 } }, 
\ev{ {\color{green!40!black} e_2 }, {\color{blue!60!black} e_3 } , {\color{red!80!black} e_1 }  ,  {\color{brown!50!black} e_4 } }, 
\ev{ {\color{red!80!black} e_1 }  , {\color{brown!50!black} e_4 }, {\color{green!40!black} e_2 },  {\color{blue!60!black} e_3 }  }, 
\ev{ {\color{green!40!black} e_2 }, {\color{brown!50!black} e_4 }, {\color{red!80!black} e_1 }  ,  {\color{blue!60!black} e_3 }  } 
\} \\ 
& \mathfrak{C}\ev{ D } = \{ \ [ \ev{ {\color{red!80!black} e_1 }  , {\color{blue!60!black} e_3 } , {\color{green!40!black} e_2 },  {\color{brown!50!black} e_4 } } ]_\rho, [ \ev{ {\color{red!80!black} e_1 }  , {\color{brown!50!black} e_4 }, {\color{green!40!black} e_2 },  {\color{blue!60!black} e_3 } } ]_\rho \ \} \\ 
& \mathcal{B}\ev{ D } = \{ \ [ \ev{  {\color{red!80!black} e_1 }  , {\color{blue!60!black} e_3 } } ]_\rho, [ \ev{  {\color{red!80!black} e_1 }  , {\color{brown!50!black} e_4 } } ]_\rho, [ \ev{  {\color{green!40!black} e_2 }, \ {\color{blue!60!black} e_3 } } ]_\rho, [ \ev{  {\color{green!40!black} e_2 }, \ {\color{brown!50!black} e_4 } } ]_\rho \ \}
\end{align*}
\end{enumerate}
\end{example}

\subsection{The Heaps of Pieces Framework}
To use the theory of Heaps of Pieces (\cite{krattenthaler,viennot}), we reformulate some definitions and theorems. Let $\ev{ \Omega, \leq }$ be a poset. 
\begin{enumerate}
\item Given two elements $x,y\in \Omega$, then $y$ \textit{covers} $x$, denoted $x\lessdot y$, provided $x < y$ and there is no $z\in \Omega$ such that $x<z<y$.	
\item The \textit{down-set} of $a$ is the set: $ \downset[\Omega]{a} := \{ b\in \Omega: b \leq_\Omega a \}$. Let $ \langle a \rangle_{\Omega} := \downset[\Omega]{a} \setminus \{a\} = \{ b\in \Omega: b <_\Omega a \}$ be the set of elements strictly below $a$. 
\item The \textit{up-set} of $a$ is the set: $ \upset[\Omega]{a} := \{ b\in \Omega: b \geq_\Omega a \}$.
\item Let $\mathcal{M}\ev{ \Omega }$ denote the set of maximal elements of $\Omega$. 
\end{enumerate}

The following is a restatement of \cite[Definition 2.1, p.~324]{viennot}:

\begin{definition}
\label{def:heap_viennot}
Let $\mathcal{B}$ be a set (of pieces) with a symmetric and reflexive binary (concurrence) relation $\mathcal{R}$. Then, a \textit{heap} is a triple
$$ H = \ev{ \Omega , \leq, \ell }$$
where $\ev{ \Omega , \leq }$ is a poset, $\ell: \Omega \rightarrow \mathcal{B}$ is a function ($\ell$ labels the elements of $\Omega$ by elements of $\mathcal{B}$) such that for all $x,y\in \Omega$,
\begin{enumerate}
\item If $\ell\ev{x} \mathcal{R} \ell\ev{y}$ then $x\leq y$ or $y \leq x$.
\item If $x \lessdot y$, then $\ell\ev{x} \mathcal{R} \ell\ev{y}$.
\end{enumerate} 
Equivalently (\cite[Definition 2.5]{krattenthaler}, \cite[p.~325]{viennot}), a \textit{heap} is a triple $ H = \ev{ \Omega , \leq, \ell }$, where $\ell: \Omega \rightarrow \mathcal{B}$ is a function and there exists a set of relations $\preceq$ such that, \textit{i)} For all $x,y\in \Omega$, $\ell\ev{x} \mathcal{R} \ell\ev{y}$ if and only if $x\preceq y$ or $y \preceq x$, and \textit{ii)} $\leq$ is the partial order obtained by taking the transitive closure of $\preceq$. 
\end{definition}

\begin{remark}
To see that the two definitions given above are equivalent, let $H = \ev{ \Omega , \leq, \ell }$ be a heap according to the first definition. Then, define $x\preceq y$ if and only if $x\leq y$ and $\ell\ev{x} \mathcal{R} \ell\ev{y}$, and note that $\leq $ is the transitive closure of $\preceq$, and so, $H$ is a heap according to the second definition. Conversely, if $H' = \ev{ \Omega' , \leq', \ell' }$ is a heap according to the second definition, then $\leq'$ satisfies the conditions of the first definition. 
\end{remark}

Let $\mathbbm{h}\ev{\mathcal{B},\mathcal{R}}$ be the set of heaps labeled by the set of pieces $\mathcal{B}$, with a given concurrence relation $\mathcal{R}$. Whenever the concurrence relation is clear from context, we suppress the argument $\mathcal{R}$ and shortly write $\mathbbm{h}\ev{\mathcal{B}}$ instead of $\mathbbm{h}\ev{\mathcal{B},\mathcal{R}}$. Whenever $\leq$ and $\ell$ are clear from context, we refer to a heap $H = \ev{ \Omega, \leq ,\ell}$ by the underlying ground set $\Omega$. 


We include a restatement of \cite[Definition 2.5, p.~4]{krattenthaler}, for the composition of heaps: Let $H_1 = (\Omega_1, \leq_1, \ell_1 : \Omega_1 \rightarrow \mathcal{B} ) , H_2 = (\Omega_2, \leq_2, \ell_2: \Omega_2 \rightarrow \mathcal{B} ) \in \mathbbm{h}\ev{ \mathcal{B}, \mathcal{R}}$ be heaps, such that $\Omega_1 \cap \Omega_2 = \emptyset$. Then
the composition of $H_1$ and $H_2$, $H_1 \circ H_2$, is the heap $(\Omega_1 \sqcup \Omega_2, \leq_3, \ell_3)$ with
\begin{enumerate}
\item[1)] $\ell_3\ev{ v } = \begin{cases} \ell_1\ev{ v } & \text{ if } v\in \Omega_1 \\  \ell_2\ev{ v } & \text{ if } v\in \Omega_2 \end{cases}$ 
\item[2)] The partial order $\ell_3$ on $ \Omega_1 \sqcup \Omega_2$ is the transitive closure of
\begin{enumerate}
\item[a)] $v_1 \leq_3 v_2$ if $v_1,v_2 \in \Omega_1$ and $v_1 \leq_1 v_2$,
\item[b)] $v_1 \leq_3 v_2$ if $v_1,v_2 \in \Omega_2$ and $v_1 \leq_2 v_2$,
\item[c)] $v_1 \leq_3 v_2$ if $v_1 \in \Omega_1$, $v_2 \in \Omega_2$ and $\ell_1(v_1) \mathcal{R} \ell_2(v_2)$.
\end{enumerate}
\end{enumerate}

\begin{remark}
Since there can be more than one element in a heap, labelled by the same piece, it is ambiguous to refer to elements of a heap by their labels. However, given a (not necessarily distinct) sequence of pieces $\beta_1,\ldots,\beta_m$, then the heap 
$$ \beta_1\circ \ldots \circ \beta_m $$
is well-defined. 
\end{remark}

For a heap $H = \ev{ \Omega, \leq, \ell}$, let $|H|$ be the number of elements in the poset $\Omega$ and let $\mathcal{M}\ev{H} := \mathcal{M}\ev{ \Omega }$ be the maximal elements in the poset.  
\begin{remark}

\label{rmk_2}

Given heaps $H_1$ and $H_2$, we have $\mathcal{M}\ev{H_1 \circ H_2 }\subseteq \mathcal{M}\ev{H_1}\cup \mathcal{M}\ev{H_2}$.

\end{remark}
Let $\mathbbm{t}\ev{\mathcal{B},\mathcal{R}}$ be the set of \textit{trivial heaps}, i.e., heaps where any pair of pieces have non-concurrent labels. Let $\mathbbm{p}\ev{ \mathcal{B},\mathcal{R}}$ be the set of \textit{pyramids}, i.e., heaps with a unique maximal piece. 

\begin{definition}[Heaps of Pieces on a Digraph]
\label{def:heaps_of_a_digraph}
Let $D$ be a multi-digraph. We define the set of pieces as the finite set $\mathcal{B}\ev{D}$, that is, the cycles of $D$. Two pieces $\beta_1$ and $\beta_2$ are \textit{concurrent} (denoted $\beta_1 \mathcal{R} \beta_2$) provided they share a vertex. Whenever it is clear from context, we shorten the notation and write $\mathbbm{h}\ev{ D }$ instead of $\mathbbm{h}\ev{ \mathcal{B}\ev{D}, \mathcal{R} }$, for the set of heaps of a multi-digraph $D$. Fix a vertex $u\in V\ev{D}$ and an edge $e\in E\ev{D}$.
\begin{enumerate}
\item Define the set of pieces containing $u$:
$$
\mathcal{B}^u = \{ \beta \in \mathcal{B}\ev{D}: u\in V\ev{\beta} \}
$$

\item The set of pieces containing the edge $e$:
$$
\mathcal{B}^e = \{ \beta\in \mathcal{B}\ev{D}: e\in E\ev{\beta} \}
$$

\item For a piece $\beta \in \mathcal{B}\ev{D}$, let $E\ev{\beta}$ be the edge-set of $\beta$. A pyramid $P \in \mathbbm{p}\ev{D}$ is called a \textit{decomposition pyramid}, provided $E\ev{D} =  \bigsqcup_{ \beta \in P } E\ev{\beta} $ is a partition of the edges of $D$ into cycles. Let $\mymathbb{dp}\ev{ D }$ be the set of decomposition pyramids of a connected digraph $D$. 
\item Define the set
$$
\mathbbm{p}^{e}\ev{ D } = \{ P\in \mathbbm{p}\ev{ D } : \mathcal{M}\ev{ P } = \{\beta\} \text{ and } \beta\in \mathcal{B}^e \}
$$
of pyramids with maximal piece containing $e$. Let $\mymathbb{dp}^{e}\ev{D} $ be the set of decomposition pyramids with maximal piece containing $e$. 
\item The pyramids $\mathbbm{p}^{u}\ev{D }$ with maximal piece containing $u$, and the decomposition pyramids $\mymathbb{dp}^{u}\ev{D }$ with maximal piece containing $u$ are defined similarly. 
\end{enumerate}
\end{definition}

\begin{remark}
By the definition of $\mathcal{B}^e$, we can not have any heap $H$ with $\mathcal{M}\ev{H}\subseteq \mathcal{B}^e$ and $|\mathcal{M}\ev{H}|\geq 2$. Hence, it follows that 
$$
\{H\in \mathbbm{h}\ev{ D } : \mathcal{M}\ev{H}\subseteq \mathcal{B}^e \} = \mathbbm{p}^{e}\ev{D }
$$
Similarly, we have 
$$
\{H\in \mathbbm{h}\ev{ D } : \mathcal{M}\ev{H}\subseteq \mathcal{B}^u \} = \mathbbm{p}^{u}\ev{D }
$$ 
\end{remark}

\begin{lemma}
\label{lem:push_down}
Given a heap $H = \ev{ \Omega,\leq, \ell}$ and an element $\omega\in \Omega$, then the down-set $\downset[H]{\omega}$ is a pyramid. Let $H_1:=\ev{ \downset[H]{\omega}, \leq_{H_1}, \ell_{H_1}}$ and $H_2=\ev{ H\setminus \downset[H]{\omega}, \leq_{H_2}, \ell_{H_2} }$ be elements of $\mymathbb{h}\ev{\mathcal{B},\mathcal{R}}$, where $\leq_{H_i}$ and $\ell_{H_i}$ are restrictions of $\leq$ and $\ell$ to $H_i$, for $i=1,2$. Then, we have $\mathcal{M}\ev{ H_2 } = \mathcal{M} \ev{ H } \setminus \{\omega\}$ and
$$H_1\circ H_2 = H$$
\end{lemma}
See \hyperlink{proof:downsetrecomposed}{proof} on p.~\pageref*{proof:downset_recomposed}. 

\subsection{Bijections between Pyramids and Trails/Walks}
\label{sec:bijections_pyramids_walks}
Let $D$ be an Eulerian multi-digraph. Let $u\in V\ev{D}$ be any vertex, and let $e\in E\ev{D}$ be any edge.

We will construct a bijection between Eulerian trails of $D$ at $u$ and decomposition pyramids of $\mathcal{B}\ev{D}$ with the maximal piece containing $u$. Likewise, Eulerian trails of $D$ ending at $e$ are in bijection with decomposition pyramids of $\mathcal{B}\ev{D}$ with maximal piece containing $e$. 

First, we have some preliminary definitions and results on closed trails.

Let $\mathbf{w}=\ev{v_0,e_1,v_1,e_2,\ldots, v_d, e_d, v_0}$ be a closed trail at $u$, where $v_0=u$. Let $X\subseteq V\ev{\mathbf{w}}=\{v_0,\ldots,v_d\}$ be a subset of the vertices of $\mathbf{w}$. Let $j$ be the minimal index such that $v_j\in X$. Then, we say that $v_j$ is the \textit{first vertex $\mathbf{w}$ visits in $X$.} 

Given two not-necessarily-closed trails $\mathbf{w}_1,\mathbf{w}_2$, assume that the terminal vertex of $\mathbf{w}_1$ is the same as the initial vertex of $\mathbf{w}_2$:
$$ \mathbf{w}_1 = \ev{ v_0, e_1,v_1,\ldots, v_{m-1},e_{m-1},v_m} \text{ and } \mathbf{w}_2 = \ev{ v_0', e_0', v_{1}', \ldots, v_{r-1}',e_{r-1}',v_r'}$$
where, $v_m = v_0'$ and $0\leq m, r$. Then, the \textit{concatenation} of $\mathbf{w}_1$ and $\mathbf{w}_2$, denoted $\mathbf{w}_1 \mathbf{w}_2$, is the trail:
$$ \ev{ v_0, e_1,v_1,\ldots, v_{m-1},e_{m-1},v_m, e_0', v_{1}', \ldots, v_{r-1}',e_{r-1}',v_r'} $$
Given a trail $\mathbf{w} = \ev{ v_0,e_1,\ldots,e_d,v_d }$, if $v_i = v_j$ for some $0\leq i < j \leq d$, then $\mathbf{w}_1 := v_i \mathbf{w} v_j$ is a closed \textit{subtrail} of $\mathbf{w}$. Then, we obtain another closed subtrail:
$$
\mathbf{w} \setminus \mathbf{w}_1 := \ev{ v_0,e_1, \ldots, e_i, v_i, e_{j+1},v_{j+1}, \ldots, e_d,v_d }
$$
If $\mathbf{w}_2$ is a closed subtrail of $\mathbf{w}\setminus \mathbf{w}_1$, then we use the notation $\mathbf{w}\setminus \ev{ \mathbf{w}_1 \cup \mathbf{w}_2 } := \ev{ \mathbf{w} \setminus \mathbf{w}_1 }\setminus \mathbf{w}_2 $. 

Let $\mathbf{w}_1, \mathbf{w}_2$ be closed trails at $v_0$ and $v_0'$, respectively:
$$ \mathbf{w}_1 = \ev{ v_0, e_1,v_1,\ldots, v_{m-1},e_{m-1},v_m} \text{ and } \mathbf{w}_2 = \ev{ v_0', e_0', v_{1}', \ldots, v_{r-1}',e_{r-1}',v_r'}$$
where $E\ev{ \mathbf{w}_1 } \cap E\ev{ \mathbf{w}_2 } = \emptyset$ and $v_0' = v_j \in V\ev{\mathbf{w}_1}$ for some $j\geq 0$, such that $V\ev{ \mathbf{w}_1 } \cap \{ v_i \}_{i=0}^{j-1} = \emptyset$ (in particular, $v_j$ is the first vertex $\mathbf{w}_2$ visits in $V\ev{\mathbf{w}_1} \cap V\ev{\mathbf{w}_2}$). Then, the \textit{insertion} of $\mathbf{w}_1$ into $\mathbf{w}_2$ is defined as the concatenation:
$$
\mathbf{w}_1\cdot \mathbf{w}_2 := (v_0 \mathbf{w}_2 v_j) \ev{ v_0' \mathbf{w}_1 v_r' }  \ev{ v_j \mathbf{w}_2 v_d }
$$ 
In particular, when $j=0$, i.e., both $\mathbf{w}_1$ and $\mathbf{w}_2$ are closed trails at the same vertex $v_0 = v_0'$, then $\mathbf{w}_1 \cdot \mathbf{w}_2 = \mathbf{w}_1 \mathbf{w}_2$ is just the concatenation of $\mathbf{w}_1$ and $\mathbf{w}_2$, as defined above.
\begin{example}
\label{ex:non_assoc}
Consider the digraph and the simple closed trails shown in Figure \ref{fig:nonassocexample}.
\begin{figure}[ht]
~ \qquad \qquad \qquad  \qquad  \qquad \qquad 
\begin{subfigure}{0.17\textwidth}%
\centering
\raisebox{-0.7cm}{ \includegraphics[width=\linewidth]{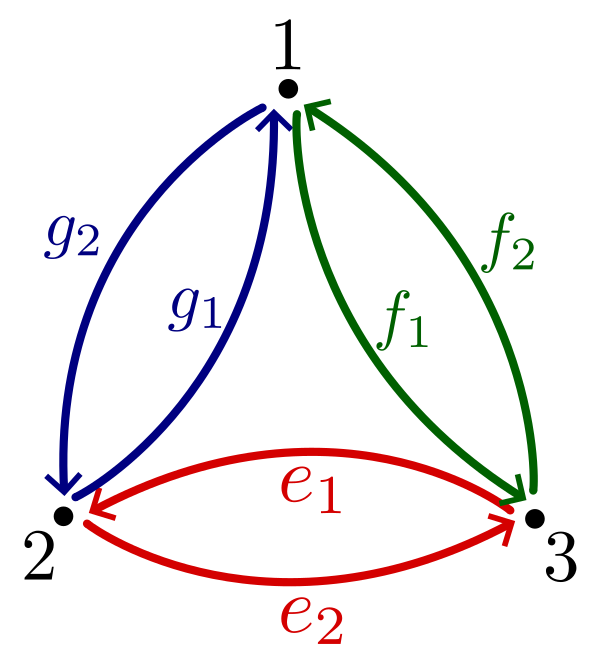} } 
\end{subfigure}
~ \quad
\begin{subfigure}{0.3\textwidth}%
\centering
\begin{align*}
& { \color{blue!60!black} \mathbf{w}_1 } = \ev{ 2, { \color{blue!60!black} g_1 }, 1, { \color{blue!60!black} g_2 }, 2} \\ 
& { \color{green!40!black} \mathbf{w}_2 }   = \ev{ 3, { \color{green!40!black} f_2 }, 1, { \color{green!40!black} f_1 }, 3} \\ 
& { \color{red!80!black} \mathbf{w}_3 } = \ev{ 2, { \color{red!80!black} e_2 }, 3, { \color{red!80!black} e_1 }, 2} 
\end{align*} 
\end{subfigure}
\caption{A digraph $D$ and some of its simple closed trails.}
\label{fig:nonassocexample}
\end{figure}

Then, we have 
$$
{ \color{blue!60!black} \mathbf{w}_1 } \cdot \ev{ { \color{green!40!black} \mathbf{w}_2 } \cdot {\color{red!80!black} \mathbf{w}_3 } } = { \color{blue!60!black} \mathbf{w}_1 } \cdot \ev{ 2,  { \color{red!80!black} e_2 } , 3, { \color{green!40!black} f_2 } , 1, { \color{green!40!black} f_1 }, 3,  { \color{red!80!black} e_1 }, 2} = \ev{ 2, { \color{blue!60!black} g_1 } , 1, { \color{blue!60!black} g_2 }, 2,  { \color{red!80!black} e_2 } , 3, { \color{green!40!black} f_2 } , 1, { \color{green!40!black} f_1 }, 3,  { \color{red!80!black} e_1 }, 2 }
$$ 
whereas $( \mathbf{w}_1 \cdot \mathbf{w}_2 ) \cdot \mathbf{w}_3$ is undefined, since $ \mathbf{w}_1 \cdot \mathbf{w}_2 $ is undefined, as the conditions of the insertion operation are not satisfied. 
\end{example}

\begin{remark}
As Example \ref{ex:non_assoc} demonstrates, the insertion operation is not associative. We will use the short notation $\mathbf{w}_1 \cdots \mathbf{w}_m$ to refer to the trail $\mathbf{w}_1 \cdot \ev{ \mathbf{w}_2 \cdot ( \cdots ( \mathbf{w}_{m-2} \cdot \ev{ \mathbf{w}_{m-1} \cdot \mathbf{w}_m } }$, for $m\geq 3$, whenever it is defined.
\end{remark}

Given a cycle $\beta\in \mathcal{B}\ev{D}$, and a vertex $v\in V\ev{ \beta }$, let $ \beta\ev{v} $ be the unique simple closed trail $\mathbf{w}$ at $v$ such that $[\mathbf{w}]_{\rho} = \beta$.

\begin{definition}[The insertion $\beta \cdot \mathbf{w}$ of a cycle $\beta$ into $\mathbf{w}$]
\label{def:insertion_beta}
Let $\mathbf{w} = \ev{ v_0, e_1,v_1,\ldots, v_{m-1},e_{m-1},v_m}$ be a closed trail at $v_0$ and let $\beta$ be a cycle such that $V\ev{\beta}\cap V\ev{\mathbf{w}} \neq \emptyset$ and $E\ev{ \beta } \cap E\ev{ \mathbf{w} } = \emptyset$. Let $j$ be minimal such that $v_j \in V\ev{ \beta}$. Then, the \textit{insertion} of $\beta$ into $\mathbf{w}$, is defined as the insertion of the trail $\beta \ev{ v_j } $ into $\mathbf{w}$:
$$\beta \cdot \mathbf{w} :=  \beta\ev{ v_j }  \cdot \mathbf{w}$$
\end{definition}

Let $D$ be a multi-digraph. From each closed trail $\mathbf{w} \in \mathcal{T}\ev{D}$, we will construct a unique sequence of cycles, as follows: First, we express $\mathbf{w}$ uniquely as a concatenation $\alpha \mathbf{w}_1 \gamma$ where,
\begin{itemize}
\item[i)] $\alpha$ is a path starting at $u$ (possibly of length zero); and 
\item[ii)] $\mathbf{w}_1$ is the simple closed trail, the internal vertices of which are distinct and disjoint from $\alpha$. 
\end{itemize}
We say that $\mathbf{w}_1$ is the \textsl{first simple closed subtrail} of $\mathbf{w}$ and $\alpha \gamma$ is the \textsl{remainder}. Note that the remainder $\alpha\gamma$ is again a closed trail at $u$, or a single vertex, i.e., a trail of length $0$. If $\alpha \gamma$ has length $>0$, we may again extract the first simple closed subtrail of $\alpha\gamma$, say $\mathbf{w}_2$. In this way, we iteratively decompose $\mathbf{w}$ into simple closed trails, $\{ \mathbf{w}_i \}_{i=1}^{m}$. Now, we provide a formal definition, to capture the idea exactly.
\begin{figure}[ht]
\centering
\includegraphics[width=0.6\linewidth]{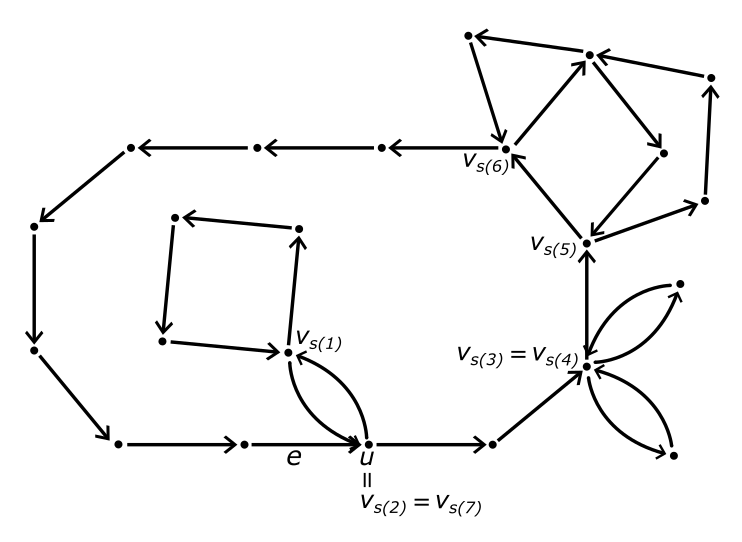}
\caption{\textit{An example of a closed trail.}}
\label{fig:walkspyramids}
\end{figure}
\begin{definition}[Cycle sequence of a closed trail]
\label{def:cycle_seq}
Let $\mathbf{w} =  \ev{ v_0, e_1,v_1,\ldots, v_{d-1},e_{d-1},v_d} $ be a closed trail. We define an increasing sequence $0\leq t\ev{1} < \ldots < t\ev{m} = d$ of indices as follows: Let $\widetilde{\mathbf{w}}_{1} = \mathbf{w}$. For $j \geq 1$, let $t\ev{j}$ be the minimum repeated index of $\widetilde{\mathbf{w}}_{j}$. Let $s\ev{j}$ be the unique index of $\widetilde{\mathbf{w}}_{j}$ such that $v_{s\ev{j}} = v_{t\ev{j}}$ (cf. Figure \ref{fig:walkspyramids}). Then, let $\mathbf{w}_{j} = v_{s\ev{j}} \widetilde{\mathbf{w}}_{j} v_{t\ev{j}}$ and $\widetilde{\mathbf{w}}_{j+1}:= \widetilde{\mathbf{w}}_{j} \setminus \mathbf{w}_j $. This gives a sequence of subtrails: 
\begin{align*}
\widetilde{\mathbf{w}}_{1} &:= \mathbf{w} 										&&	\hspace{-2.5cm}	&&		\hspace{-2.5cm}		= \alpha_1 \mathbf{w}_{1} \gamma_1 \\ 
\widetilde{\mathbf{w}}_{2} &:= \mathbf{w} \setminus \mathbf{w}_1 											&&\hspace{-2.5cm} = \alpha_1 \gamma_1  && \hspace{-2.5cm} = \alpha_2 \mathbf{w}_{2} \gamma_2 \\ 
& \vdots && \hspace{-2.5cm} \quad \vdots && \hspace{-2.5cm} \quad \vdots  \\
\widetilde{\mathbf{w}}_{m-1} &:= \mathbf{w} \setminus \ev{ \mathbf{w}_1 \cup \ldots \cup \mathbf{w}_{m-2} } &&\hspace{-2.5cm} =  \alpha_{m-2}\gamma_{m-2}  && \hspace{-2.5cm} = \alpha_{m-1} \mathbf{w}_{m-1} \gamma_{m-1} \\ 
\widetilde{\mathbf{w}}_{m} &:= \mathbf{w} \setminus \ev{ \mathbf{w}_1 \cup \ldots \cup \mathbf{w}_{m-1} } 	&&\hspace{-2.5cm} = \alpha_{m-1} \gamma_{m-1} && \hspace{-2.5cm} = \mathbf{w}_{m}  
\end{align*}
where:
\begin{equation}
\text{ $\alpha_i = v_0 \widetilde{\mathbf{w}}_{i}  v_{s\ev{ i }}$ is a path and $\mathbf{w}_i := v_{s\ev{i}} \widetilde{\mathbf{w}}_{i} v_{t\ev{i}} $ is a simple closed trail, for each $i = 1,\ldots,m $.}
\label{eqn:uniqueness_property}
\end{equation}
and in particular, $\alpha_m $ is the trail of length $0$ at $v_0$. Define $\beta_i := [\mathbf{w}_i]_\rho$, for $i=1,\ldots,m$. This yields a sequence of cycles $(\beta_1, \ldots, \beta_m )$, called the \textsl{cycle sequence} of $\mathbf{w}$. 
\end{definition}

We define a function $ \textup{cs}: \mathcal{T}\ev{D} \rightarrow \mathcal{B}\ev{D}^{*} $, where 
$$
\mathbf{w}\mapsto \begin{cases} \emptyset & \text{ if } |\mathbf{w}| = 0 \\ (\beta_1, \beta_2,\ldots, \beta_{m}) & \text{ otherwise}
\end{cases}
$$
mapping each closed trail to its cycle sequence.

\begin{remark}
\label{rmk:can_not_reach_back}
By definition of a cycle sequence, we satisfy the following: For each $r=2,\ldots,m$,
$$
\textup{cs}\ev{ \widetilde{\mathbf{w}}_{r} } = \ev{\beta_r,\ldots, \beta_m} \text{ and } \beta_1 \cdots \beta_{r-1} \cdot \widetilde{\mathbf{w}}_{r} = \mathbf{w} 
$$
Furthermore, the simple closed subtrails $\{ \mathbf{w}_{i} \}_{i=1}^{m}$ uniquely satisfy the property (\ref{eqn:uniqueness_property}), i.e., for each $i=1,\ldots,m$,
$$
\{ v_a \widetilde{\mathbf{w}}_{i} v_b: v_0 \widetilde{\mathbf{w}}_{i} v_a \text{ is a path and } v_a \widetilde{\mathbf{w}}_{i} v_b \text{ is simple closed and } b \geq a \geq 0 \} =  \{ \mathbf{w}_i \}
$$
\end{remark}

\begin{lemma}
\label{lem:cycle_seq}

\phantom{a}
Let $D$ be a multi-digraph. Let $e\in E\ev{D}$ be a fixed edge. Then, the mapping $\textup{cs}: \mathcal{T}^{e}\ev{D} \rightarrow \mathcal{B}\ev{D}^{*}$ is an injection, and moreover, $e\in E\ev{ \beta_m }$.
\end{lemma} 

\begin{proof}
We prove, by induction on $m$, that $\textup{cs}\ev{\mathbf{w}^1}=\textup{cs}\ev{\mathbf{w}^2}=\ev{\beta_1,\ldots,\beta_m}$ implies $\mathbf{w}^1=\mathbf{w}^2$, for any two closed trails $\mathbf{w}^1,\mathbf{w}^2$ ending at $e$. 

For the base case let $m=1$. If $\textup{cs}\ev{\mathbf{w}^1} = \beta = \textup{cs}\ev{\mathbf{w}^2}$, then $\mathbf{w}^1 = \mathbf{w}^2$, since there is a unique trail ending at $e$ with edge-set $E\ev{\beta}$. For the inductive hypothesis, assume that for each $m\geq 1$ and for all trails $\mathbf{w}^1,\mathbf{w}^2$ ending at $e$, if $\textup{cs}\ev{\mathbf{w}^1}=\textup{cs}\ev{\mathbf{w}^2} =\ev{ \beta_1,\ldots,\beta_{m-1}}$ then $\mathbf{w}^1=\mathbf{w}^2$. For the inductive step, let $m\geq 2$ and assume that $(\beta_1, \beta_2,\ldots, \beta_{m})$ is the cycle sequence of both $\mathbf{w}^1$ and $\mathbf{w}^2$. Then, by the definition of a cycle sequence, we have some closed subtrails $\mathbf{w}_1^{1}, \mathbf{w}_1^{2}$ such that $\beta_1 = [\mathbf{w}_1^{1}]_{\rho} = [\mathbf{w}_1^{2}]_{\rho} $ and 
$$\mathbf{w}^1=\alpha^{1} \mathbf{w}_1^{1} \gamma^{1}  \text{ and } \mathbf{w}^2=\alpha^{2} \mathbf{w}_1^{2} \gamma^{2}$$
where $\alpha^{i}$, for $i=1,2$, is a path (possibly of length zero), and $\widetilde{\mathbf{w}}_1^{i} = \alpha^{i} \gamma^{i}$ is a closed subtrail ending at $e$, with $\textup{cs}\ev{\widetilde{\mathbf{w}}_1^{i}} = (\beta_2, \beta_3,\ldots, \beta_{m})$. By the inductive hypothesis, $\widetilde{\mathbf{w}}_1^{1} = \widetilde{\mathbf{w}}_1^{2}$. For each $i=1,2$, let $v_i$ be the first vertex that $\mathbf{w}^i$ repeats, which is also the terminal vertex of $\alpha^{i}$. Suppose, for a contradiction, that $\alpha^1 \neq \alpha^2$. Assume, without loss of generality, $\alpha^1\subsetneq \alpha^2$. In particular, $v_1 \in V\ev{\alpha^2} \setminus V\ev{\alpha^1}$ and $v_1 \neq v_2$. But then, in the subtrail $\alpha^2 \mathbf{w}_1^{2}$ of $\mathbf{w}^2$, the second occurrence of $v_1$ in $\alpha^2 \mathbf{w}_1^{2} $ is before the second occurrence of $v_2$, which contradicts the definition of $v_2$. 

Hence, $\alpha^1=\alpha^2$, which also implies $\gamma^1=\gamma^2$. In particular, the terminal vertex of $\alpha^1 = \alpha^2$ is $v_1 = v_2$. On the cycle $\beta_1$, there is a unique edge $\widetilde{e}$ with terminal vertex $v_1$. Since $\mathbf{w}_1^{1},\mathbf{w}_1^{2}$ are both trails with cycle sequence $\beta_1$, ending at $\widetilde{e}$, the basis step implies $\mathbf{w}_1^{1} = \mathbf{w}_1^{2}$. Therefore, $\mathbf{w}_1=\alpha^1 \mathbf{w}_1^{1} \gamma^1 = \alpha^2 \mathbf{w}_1^{2} \gamma^2 =\mathbf{w}_2$. 
\end{proof}

Given a sequence $\vec{b} = \ev{ \beta_1,\ldots, \beta_m} \in \mathcal{B}\ev{D}^{*}$ and an edge $e\in E\ev{\beta_m}$, it is not necessarily the case that the sequence in question is the cycle sequence of a closed trail of $D$, ending at $e$. To determine if $\vec{b}$ is in the image $\mathrm{im}\ew{\textup{cs}\big\vert_{\mathcal{T}^{e}\ev{D}}}$, we propose a shortcut algorithm: Let $v$ be the terminal vertex of $e$ and let $ \widetilde{\mathbf{w}}_m :=  \beta_m\ev{ v }$. For each $i=m-1,\ldots,1$, construct $\widetilde{\mathbf{w}}_i = \beta_i \cdots \beta_{m-1} \cdot \beta_m\ev{ v }$ by the insertion operation. Then, test if $\textup{cs}\ev{  \widetilde{\mathbf{w}}_i   } = \ev{ \beta_i,\ldots,\beta_m}$. If not, stop and conclude that $\vec{b}$ does not correspond to a closed walk ending at $e$. Otherwise, the algorithm stops at $i=1$ and yields the unique closed trail $\mathbf{w}_1$ with the given cycle sequence.

\begin{remark}
\label{rmk:to_be_lemma}
Let $D$ be a digraph and let $\mathbf{w}_1$ and $\mathbf{w}_2 = \ev{v_0,e_1,v_1,\ldots,v_{d-1},e_d,v_d}$ be closed trails such that $V\ev{ \mathbf{w}_1 } \cap V\ev{ \mathbf{w}_2 } \neq \emptyset$, $E\ev{ \mathbf{w}_1 } \cap E\ev{ \mathbf{w}_2 } = \emptyset$ and 
\begin{align*}
& \textup{cs}\ev{ \mathbf{w}_1 } = (\beta_1, \beta_2,\ldots, \beta_{r})\\
& \textup{cs}\ev{ \mathbf{w}_2 } = (\widetilde{\beta}_1, \widetilde{\beta}_2,\ldots, \widetilde{\beta}_k) 
\end{align*}
Let $j$ be the minimal index such that $v_j\in V\ev{ \mathbf{w}_1 } \cap V\ev{ \mathbf{w}_2 }$. If $v_0 \mathbf{w}_2 v_j $ is a path, then the insertion $\mathbf{w}_1 \cdot \mathbf{w}_2$ of $\mathbf{w}_1$ into $\mathbf{w}_2$ satisfies:
$$ \textup{cs}\ev{ \mathbf{w}_1 \cdot \mathbf{w}_2 } = (\beta_1, \beta_2,\ldots, \beta_{r} , \widetilde{\beta}_1, \widetilde{\beta}_2,\ldots, \widetilde{\beta}_k) $$
\end{remark}

In the following, we refer to the elements of the pyramid-posets via their labels, which is unambiguous, since the labeling function $\ell$ in each pyramid is injective. 

\begin{theorem}
\label{thm:pyramids_containing_edge}
Let $D$ be an Eulerian (connected) multi-digraph. Define the mappings
\begin{align*}
\mathcal{T}\ev{D}	\xrightarrow{\textup{cs}}	\mathcal{B}\ev{D}^*  \xrightarrow{f}	\mathbbm{p}\ev{ D }
\end{align*}
where $\textup{cs}\ev{\mathbf{w}} = \ev{\beta_1,\ldots,\beta_m}$ is the cycle sequence of a closed trail $\mathbf{w}$, $\circ$ denotes composition of heaps, $f\ev{ \beta_1,\ldots,\beta_m } = \beta_1 \circ \ldots \circ \beta_m$ and $g = f\circ \textup{cs}$. Then, the following restrictions are bijections:
\begin{enumerate}
\item $g \big\vert_{\mathcal{W}^{e}\ev{D}} : \mathcal{W}^{e}\ev{D} \rightarrow \mymathbb{dp}^{e}\ev{ D }$ is a bijection between Eulerian trails of $D$ ending at $e$ and decomposition pyramids of $D$, with maximal piece containing $e$. (recall Definition \ref{def:heaps_of_a_digraph}.)
\item $g \big\vert_{\mathcal{W}^{u}\ev{D}} : \mathcal{W}^{u}\ev{D} \rightarrow \mymathbb{dp}^{u}\ev{ D }$ is a bijection between Eulerian trails of $D$ at $u$ and decomposition pyramids of $D$, with maximal piece containing $u$.
\end{enumerate}
\end{theorem}
Before the \hyperlink{proof:walkspyramidsfinal}{proof} of \Cref{thm:pyramids_containing_edge} (on p.~\pageref*{proof:walks_pyramids_final} below), we have some technical lemmas. 

\begin{lemma}
\label{lem:strict_down_set}
Consider the indices $ \mathcal{R}\ev{1} := \{ i \in [1,m]: s\ev{ i }\leq s\ev{ 1 } \}$ in \Cref{def:cycle_seq}. We can find a strictly increasing sequence $1 = r\ev{1} < \ldots < r\ev{ k } = m $ such that $ \mathcal{R}\ev{1} = \{ r\ev{ i } : 1\leq i \leq k \}$, and in particular, $ 0 = s\ev{ m } = s\ev{ r\ev{ k } } \leq \ldots \leq s\ev{ r\ev{1} } $. Let $y_j:= v_{ s\ev{r\ev{ j }} } $, for each $j=1,\ldots,k$. Then, for each $j=2,\ldots,k$:
\begin{enumerate}
\item If $a<r\ev{j}$, then $sr\ev{j-1} \leq s\ev{a}$.
\item $\alpha_{r\ev{j}} = v_0 \widetilde{\mathbf{w}}_{r\ev{j}} v_{sr\ev{j}} = v_0 \mathbf{w} v_{sr\ev{j}} $, and in particular, $\alpha_{r\ev{k}} \subseteq \ldots \subseteq \alpha_{r\ev{1}}$.
\item $y_{j-1} \in V\ev{ \mathbf{w}_{ r\ev{ j }} }$.
\item If $ i \leq r\ev{j}-1$ and $y_{j-1} \in V\ev{ \beta_i }$, then $i \leq r\ev{ j-1 }$.
\item The first vertex that the simple closed trail $\mathbf{w}_{r\ev{j}} = \beta_{r\ev{j}}\ev{ y_{j} }$ visits in $V\ev{ \mathbf{w}_{r\ev{j}} } \cap \bigcup_{1\leq i \leq r\ev{j}-1 } V\ev{ \mathbf{w}_i}$ is $y_{j-1}$. Furthermore, we have $\langle \beta_{r\ev{j}} \rangle_{P} = \{ \beta_i: 1\leq i \leq r\ev{j}-1 \}$.
\end{enumerate}
\end{lemma}
\begin{proof}
\phantom{a}

\begin{enumerate}
\item We need to consider two cases: If $a\in \mathcal{R}\ev{1}$, then $a=r\ev{j'}$, for some $j'$. We get $a = r\ev{j'} < r\ev{ j }$, which implies $j'\leq j-1$, $r\ev{j'} \leq r\ev{ j-1 }$ and $sr\ev{j-1} \leq sr\ev{j'} = s\ev{ a }$. If, on 
the other hand, $a\not\in  \mathcal{R}\ev{1}$, then $sr\ev{j-1} \leq s\ev{1} < s\ev{a}$. 
\item By Part 1, for any $a\leq r\ev{j}-1$, we have $sr\ev{j} \leq sr\ev{j-1} \leq s\ev{a}$, and so:
$$
\alpha_{r\ev{j}} = v_0 \widetilde{\mathbf{w}}_{r\ev{j}} v_{sr\ev{j}} = v_0 \mathbf{w}\setminus\ev{ \mathbf{w}_1 \cup \ldots \cup \mathbf{w}_{r\ev{j}-1} }  v_{sr\ev{j}} = v_0 \mathbf{w} v_{sr\ev{j}} 
$$
\item By Part 1, for any $a\leq r\ev{j}-1$, we have $sr\ev{j} \leq s\ev{a}$, which implies, 
$$
v_0 \widetilde{\mathbf{w}}_{r\ev{j}} v_{sr\ev{ j-1 }}  = v_0 \mathbf{w} \setminus \ev{ \mathbf{w}_1 \cup \ldots \cup \mathbf{w}_{r\ev{j}-1} } v_{sr\ev{ j-1 }} = v_0 \mathbf{w} v_{sr\ev{ j-1 }} =  \alpha_{r\ev{ j-1 }} 
$$
where the last step follows from Part 2. Since $v_{tr\ev{j}}$ is the first repeated vertex of $\widetilde{\mathbf{w}}_{r\ev{j}}$, it follows that $sr\ev{j-1} \leq tr\ev{j}$. Therefore, \\
$v_{sr\ev{j}} \widetilde{\mathbf{w}}_{r\ev{j}} v_{sr\ev{j-1}} $ is a subpath of the simple closed trail $\mathbf{w}_{r\ev{j}} = v_{sr\ev{j}} \widetilde{\mathbf{w}}_{r\ev{j}} v_{tr\ev{j}}$, and the claim follows. 

\item Let $i\in [1,r\ev{j}-1]$ be chosen, such that $y_{j-1}\in V\ev{ \mathbf{w}_i }$, where $\mathbf{w}_i =  v_{s\ev{i}} \widetilde{\mathbf{w}}_{i} v_{t\ev{i}} $ from Definition \ref{def:cycle_seq}. In particular, $ y_{j-1} = v_{sr\ev{ j-1 }} = v_b$ for some $s\ev{i} \leq b \leq t\ev{i}$. 

Then, consider the path $\alpha_i = v_0 \widetilde{\mathbf{w}}_{i} v_{ s\ev{ i } }$. 

\begin{enumerate}
\item The path $ v_0 \widetilde{\mathbf{w}}_{i} v_{ sr\ev{ j-1 } }$ is a subpath of $\alpha_i$. 
\item The paths $ v_{ sr\ev{ j-1 } } \widetilde{\mathbf{w}}_{i} v_{s\ev{ i } } $ and $ v_{s\ev{ i } } \widetilde{\mathbf{w}}_{i} v_b $ are subpaths of $\alpha_i$ and $\mathbf{w}_i$, respectively, which means that $v_{sr\ev{ j-1 }} \widetilde{\mathbf{w}}_{i} v_b$ is a simple closed trail. 
\end{enumerate}
Therefore, by the uniqueness property of $\mathbf{w}_{i}$ noted in \Cref{rmk:can_not_reach_back}, we have $v_{sr\ev{ j-1 }} \widetilde{\mathbf{w}}_{i} v_b = \mathbf{w}_{ i } = v_{s\ev{ i }} \widetilde{\mathbf{w}}_{i} v_{t\ev{i}} $ and $s\ev{ i } = sr\ev{ j-1 }$. Then, $s\ev{ i } \leq s\ev{ 1 }$ and $i \in \mathcal{R}\ev{1}$. Now, suppose, for a contradiction, that $r\ev{j-1} < i $. Then, $r\ev{j-1} < i < r\ev{j}$ implies $i \notin \mathcal{R}\ev{1}$, which is a contradiction. Hence, we obtain that $i\leq r\ev{ j-1 }$ and the claim is established. 

\item Assume $m\geq 2$. Define $\mathcal{R}\ev{2} = \{ i \in [1,m]: s\ev{ i }\leq s\ev{ 2 } \}$. We have two cases to consider. If $s\ev{2} < s\ev{1}$, then we have $r\ev{2} = 2$ and $\mathcal{R}\ev{2} = \mathcal{R}\ev{1}\setminus \{ 1 \}$. If, on the other hand, $s\ev{1} \leq s\ev{2}$, then $\mathcal{R}\ev{1} \subseteq \mathcal{R}\ev{2}$. In both cases, we can find a strictly increasing sequence $ 2= \tilde{r}\ev{ a } < \ldots <  \tilde{r}\ev{ 2 } < \ldots < \tilde{r}\ev{ k } = m $ such that $\mathcal{R}\ev{2} = \{ \tilde{r}\ev{ i } \}_{i=a}^{k} $, where $ \Z \ni a\leq 2$ and $\tilde{r}\ev{ j } = r\ev{j}$, for $j=2,\ldots,k$.

We first show that for each $i=2,\ldots, m$, 
\begin{equation}
v_a \in V\ev{ \beta_1 } \cap V\ev{ \beta_i } \text{ implies } s\ev{1} \leq a
\label{eqn:beta_intersecting}
\end{equation}

Let $i\in [2,m]$ be fixed. Suppose, for a contradiction, that there is some $a$ with $ a < s\ev{1}$ such that $ v_a \in V\ev{ \mathbf{w}_{i} } \cap V\ev{ \mathbf{w}_1 } $. Then, there is some index $b$ with $s\ev{1} \leq b \leq t\ev{1}$ such that $v_a = v_b$. Consider the path $\alpha_1 = v_0 \mathbf{w} v_{ s\ev{ 1 } }$. 
\begin{enumerate}
\item The path $ v_0 \mathbf{w} v_{ a }$ is a subpath of $\alpha_1$. 
\item The paths $ v_{ a } \mathbf{w} v_{s\ev{ 1 } } $ and $ v_{s\ev{ 1 } } \mathbf{w} v_b $ are subpaths of $\alpha_1$ and $\mathbf{w}_1$, respectively, which means that $v_{a} \mathbf{w} v_b$ is a simple closed trail. 
\end{enumerate}
Therefore, by the uniqueness property of $\mathbf{w}_{1}$ from \Cref{rmk:can_not_reach_back}, we have $v_{a} \mathbf{w} v_b = \mathbf{w}_{ 1 } = v_{s\ev{ 1 }} \mathbf{w} v_{t\ev{1}} $ and $a = s\ev{ 1 }$, which is a contradiction.

We apply induction on $m\geq 2$, to establish both claims. First, we show that a closed trail $\mathbf{w}$ with $\textup{cs}\ev{ \mathbf{w} } = \ev{ \beta_1,\ldots,\beta_m } $ satisfies the first claim. For the base case, assume $m=2$. Then, $y_1 = v_{s\ev{1}} \in V\ev{ \mathbf{w}_2 }$, by Part 3 above. The claim follows from property (\ref{eqn:beta_intersecting}). For the inductive step, assume $m\geq 3$. By the inductive hypothesis applied to $\widetilde{\mathbf{w}}_1 $, we know that the first vertex $\mathbf{w}_{ r\ev{ j }}$ visits in $V\ev{ \mathbf{w}_{r\ev{j}} } \cap \bigcup_{2\leq i \leq r\ev{j}-1 } V\ev{ \mathbf{w}_i}$ is $y_{j-1} = v_{sr\ev{j-1}}$. If $V\ev{\beta_1} \cap V\ev{ \mathbf{w}_{r\ev{j}} } = \emptyset$, then the claim follows immediately. Otherwise, let $v_a\in V\ev{\beta_1} \cap V\ev{ \mathbf{w}_{r\ev{j}} }$. By property (\ref{eqn:beta_intersecting}), we get $s\ev{ 1 }\leq a$. So, it follows that $sr\ev{ j-1 } = s\tilde{r} \ev{ j-1 } \leq \ldots \leq s\tilde{r} \ev{ 2 } = sr\ev{ 2 } \leq sr\ev{ 1 } =  s\ev{ 1 } \leq a $.

Finally, we apply induction on $m\geq 2$, to show that a closed trail $\mathbf{w}$ with $g\ev{ \mathbf{w} } = \beta_1 \circ \ldots \circ \beta_m = P$ has $\langle \beta_{r\ev{j}} \rangle_{P} = \{ \beta_i: 1\leq i \leq r\ev{j}-1 \}$. The base case of $m=2$ is clear. For the inductive step, let $m\geq 3$. Let $j\geq 2$ be fixed. We have $\mathbf{w} = \alpha_0 \mathbf{w}_{1} \gamma_0 $, where $R:= g\ev{ \widetilde{\mathbf{w}}_1 } = \beta_2 \circ \ldots \circ \beta_m $ and $\beta_1 \circ R = P$. First, note that for $z=2,\ldots,j$, the pieces $\beta_{r\ev{z-1}}$ and $\beta_{r\ev{z}}$ are concurrent, as they share the vertex $y_{z-1}$ and so, $ \beta_{r\ev{z-1}} \leq_{P}  \beta_{r\ev{z}}$. By transitivity, we have $\beta_1 \leq_{P}  \beta_{r\ev{j}}$. 
Therefore, 
\begin{align*}
& \langle \beta_{r\ev{j}} \rangle_{P} = \langle \beta_{\tilde{r}\ev{j}} \rangle_{R} \cup \{ \beta_1 \} \\ 
& = \{ \beta_i: 2\leq i \leq \tilde{r}\ev{j} - 1 \} \cup \{ \beta_1 \} & \text{ by the inductive hypothesis} \\ 
& = \{ \beta_i: 1\leq i \leq r\ev{j} - 1 \} 
\end{align*}
\end{enumerate}
\end{proof}

\begin{lemma}
\label{lem:terminal_subtrail}
For a multi-digraph $D$, let $\textup{cs}\ev{ \mathbf{w} } = \ev{\beta_1,\ldots,\beta_m}$ be the cycle sequence of a trail $\mathbf{w} \in \mathcal{T}\ev{D}$. Then, for each $i=1,\ldots, m-1$, the cycle $\beta_i$ is concurrent with at least one $\beta_j$ with $i < j \leq m$. In particular, $\beta_{m-1}$ and $\beta_m$ are concurrent. 
\end{lemma}	
\begin{proof}
Let $m\geq 1$. Given a trail $\mathbf{w}$ with cycle sequence $\textup{cs}\ev{\mathbf{w}}= \ev{ \beta_1,\ldots,\beta_m}$, then, by the definition of a cycle sequence, we have $\mathbf{w} = \mathbf{w}_1 \cdot \widetilde{\mathbf{w}}_2 = \beta_1 \cdot \widetilde{\mathbf{w}}_2$, where $\mathbf{w}_1$ is the unique closed subtrail with $[\mathbf{w}_1]_{\rho} = \beta_1$ and $\widetilde{\mathbf{w}}_2 $ is a closed subtrail with $\textup{cs} \ev{ \widetilde{\mathbf{w}}_2 } = \ev{\beta_2,\ldots, \beta_m} $.

We prove, by induction on $m$, that, for each $i=1,\ldots, m-1$, the cycle $\beta_i$ is concurrent with at least one cycle $\beta_j$ with $i < j \leq m$. In the base case $m=1$, the statement holds vacuously. For the inductive step, let $m\geq 2$. By the inductive hypothesis applied to $\widetilde{\mathbf{w}}_2$, we infer that for each $i=2,\ldots, m-1$, the cycle $\beta_i$ is concurrent with at least one cycle $\beta_j$ with $i < j \leq m$. Hence, it is enough to show the statement when $i=1$: By \Cref{lem:strict_down_set}, we have $y_1 \in V\ev{ \beta_1 } \cap V\ev{ \beta_{r\ev{2} } }$, and so, $\beta_1$ is concurrent with at least one $\beta_j$ with $1 < j \leq m$, and the inductive step is complete. 
\end{proof}

\hypertarget{proof:walkspyramidsfinal}{\phantom{a}}
\begin{proof}[Proof of \Cref{thm:pyramids_containing_edge}:]
\label{proof:walks_pyramids_final}

We only show that $g \big\vert_{\mathcal{W}^{e}\ev{D}}$ is bijective. The second part is proven in a similar way.

Let $u$ be the terminal vertex of $e$. Let $\mathbf{w} \in \mathcal{T}^{e}\ev{D}$ be a trail, ending at $e$. In particular, $\mathbf{w}$ is a trail at the vertex $u$, i.e., $\mathbf{w}\in \mathcal{T}^{u}\ev{D}$. Let $\ev{ \beta_1, \beta_2,\ldots, \beta_{m} }$ be the cycle sequence of $\mathbf{w}$. By \Cref{lem:cycle_seq}, we have $e\in E\ev{ \beta_m }$, the last cycle in the cycle sequence of $\mathbf{w}$. 

i) First, we show that $g\big\vert_{\mathcal{W}^{e}\ev{D}} \ev{\mathbf{w}} = \beta_1 \circ  \beta_2 \circ \ldots \circ \beta_{m} $ is a pyramid $P$ with maximal piece $\beta_m$ (which contains $e$). The base case, $m=1$ is clear. For the inductive step, let $m\geq 2$. By \Cref{rmk:can_not_reach_back}, the subtrail $\widetilde{\mathbf{w}}_{2}$ ending at $e$ has cycle sequence $(\beta_{2}, \ldots, \beta_m)$. By the inductive hypothesis, $g\ev{ \widetilde{\mathbf{w}}_{2} } = \beta_2 \circ \ldots \circ \beta_m $ is a pyramid with maximal piece $\beta_m \in \mathcal{B}^e$. By \Cref{lem:terminal_subtrail}, $\beta_1$ is concurrent with $\beta_j$, for some $2\leq j\leq m$. By the definition of composition of heaps, we have $\beta_1 \leq \beta_j$. Since $\beta_m $ is the maximal piece of the heap $\beta_2 \circ \ldots \circ \beta_m$, we also have $\beta_j \leq \beta_m$. Therefore, by transitivity, $\beta_1 \leq \beta_m$.

ii) $g\big\vert_{\mathcal{W}^{e}\ev{D}}$ is a surjection: 

By induction on the number of pieces, we prove the claim: For all pyramids $P$ with $m \geq 1$ pieces, where maximal piece $\delta$ contains $e$, there is some trail $\mathbf{w}$ with cycle sequence $\ev{\beta_1,\ldots,\beta_m = \delta }$ such that $g\ev{\mathbf{w}}=f\ev{\beta_1,\ldots,\beta_m}=P$. For the base case, assume that $m=1$ and consider the pyramid $P = \beta_1 $, a trivial heap with a single piece, containing $e$. Then, the trail $\mathbf{w} = \beta_1\ev{u}$ is in the preimage of $\beta_1$. For the inductive step, let $P$ be a pyramid on $m>1$ pieces, with maximal piece $\delta$ containing $e$. Let $\beta_{i_1},\ldots, \beta_{i_k}$, $k\geq 1$ be the pieces that $\delta$ covers. As $\delta$ covers $\beta_{i_1},\ldots,\beta_{i_k}$, they are not comparable to each other in the poset of $P$, and so, they are pairwise non-concurrent, by the definition of a heap. Let $v\in V\ev{ \delta } \cap \ew{  V\ev{\beta_{i_1}}\cup \ldots \cup V\ev{\beta_{i_k} } } $ be the first vertex that the trail $ \delta\ev{ u }$ visits. (It is possible that $ u = v $.) There is a unique $t$ with $1\leq t \leq k$ such that $v \in \beta_{i_t} =:\tilde{\delta} \lessdot \delta$. There is a unique edge $\tilde{e}$ of $\tilde{\delta}$, with terminal point $v$. Define $P_1:=\downset[P]{\tilde{\delta}} $, which is a pyramid with maximal piece $\tilde{\delta}$ containing the edge $\tilde{e}$. Also, $P_2:=P\setminus P_1$ is a pyramid with maximal piece $\delta$ and by \Cref{lem:push_down}, we have $P=P_1\circ P_2$. By the inductive hypothesis, there are some $a,b\geq 1$ with $a+b=m$ and some trails 
$$\mathbf{w}_1 \text{ ending at $\tilde{e}$ with cycle sequence } \ev{\widetilde{\beta}_1,\ldots,\widetilde{\beta}_{a-1}, \tilde{\delta} }$$ and $$\mathbf{w}_2 \text{ ending at $e$ with cycle sequence } \ev{\beta_1,\ldots,\beta_{b-1}, \delta} $$
such that $g\ev{\mathbf{w}_1}=P_1$ and $g\ev{\mathbf{w}_2}=P_2$. Since the first vertex $\delta$ visits in $V\ev{\delta} \cap V\ev{\tilde{\delta}}$ is $v$, it follows, by \Cref{rmk:to_be_lemma}, that the trail $\mathbf{w} = \mathbf{w}_1 \cdot \mathbf{w}_2 $ obtained by insertion has cycle sequence $$(\beta_1,\ldots,\beta_{r-1},\tilde{\delta},\gamma_1,\ldots,\gamma_{s-1},\delta)$$ and $$g(\mathbf{w}_1 \cdot \mathbf{w}_2) = P_1\circ P_2 = P$$

iii) $g\big\vert_{\mathcal{W}^{e}\ev{D}} $ is a injection: To show this claim, we define a left inverse for $g$. Given a pyramid $P \in \mymathbb{p}^{e}\ev{ D }$ with maximal piece $\delta$ containing $e$, we will recursively define a finite sequence of pieces $\ev{ a_i^{P} }_{i\geq 1}$ and root vertices $x_i^{P} \in V\ev{a_i^{P}}$, for $i\geq 1 $. 
Let $a_1 := \delta $ and $x_1  := u$, the terminal vertex of $e$. For each $i \geq 1$, if $a_i$ is minimal, stop. If, on the other hand, $a_i $ is not minimal in $P$, let $x_{i+1} $ be the first vertex that the simple closed trail $ a_i\ev{ x_i }$ visits in $V\ev{ a_i } \cap \bigcup_{ \beta \in \langle a_i  \rangle_{P}  } V\ev{ \beta }$. Let $a_{i+1} $ be the unique maximal piece $\beta$ in $\langle a_i  \rangle_{P}$ such that $x_{i+1} \in V\ev{ \beta }$. (The pieces $\{ \beta \in P: x_{i+1}\in V\ev{\beta}\}$ are pairwise concurrent and so, comparable in $P$, by the definition of a heap.) 

Next, we recursively define a function $h: \mymathbb{p}^{e}\ev{ D } \rightarrow \mathcal{W}^{e}\ev{D} $, where $h\ev{ \beta } = a_{k}^{P} \cdot h\ev{ P \setminus a_{k}^{P} }$. We claim that $h$ is a left-inverse of $g\big\vert_{\mathcal{W}^{e}\ev{D}}$. By induction on the number of pieces, we prove the claim: For all $\mathbf{w}\in \mathcal{W}^{e}\ev{D}$ such that $\text{cs}\ev{\mathbf{w}} = \ev{\beta_1,\ldots,\beta_m}$, we have $hg\ev{ \mathbf{w} } = \mathbf{w}$. The base case $m=1$ is clear. For the inductive step, assume $m\geq 2$. Let $g\ev{ \mathbf{w} } = \beta_1\circ\ldots \circ\beta_m =:P $ and consider the sequence of pieces $\ev{ a_i^{P} }_{i\geq 1}$. We apply reverse induction on $j$ and show that the sequence stops at $k$ and $a_{k+1-j}= \beta_{r\ev{j}}$ and $x_{k+1-j} = y_{j}$, for each $j=k,\ldots,1$. The base case, $j=k$ is clear. For the inductive hypothesis, assume that the statement holds for $ j\leq k $. By \Cref{lem:strict_down_set}, the elements less than $a_{k+1-j} = \beta_{r\ev{j}}$ are $\beta_1,\ldots,\beta_{r\ev{j}-1}$. The first vertex that the closed trail $ a_{k+1-j}\ev{ x_{k+1-j} } =  \beta_{r\ev{j}} \ev{ y_j }$ visits in $ V\ev{ a_i } \cap \bigcup_{ \beta \in \langle a_i  \rangle_{P}  } V\ev{ \beta } = V\ev{ \beta_{r\ev{j}} } \cap \bigcup_{ \beta \in \langle \beta_{r\ev{j}}  \rangle_{P}  } V\ev{ \beta }$ is $y_{j-1}$. Hence, we obtain $x_{k+1-(j-1)} = y_{j-1}$. Furthermore, $y_{j-1}\in \beta_{r\ev{j-1}}$ and any piece $\beta_i \in  \langle \beta_{r\ev{j}}  \rangle_{P} $ such that $y_{j-1} \in V\ev{\beta_i} $ satisfies $i\leq r\ev{j-1}$. For any $i_1 \leq i_2$, we have either $\beta_{i_1} \leq_P \beta_{i_2}$ or $\beta_{i_1}$ and $\beta_{i_2}$ are incomparable in $P$. Therefore, $a_{k-j} = \beta_{r\ev{j-1}}$ is the unique maximal piece in $ \langle \beta_{r\ev{j}}  \rangle_{P}  $ such that $x_{k-j} \in V\ev{ a_{k-j} }$, and the reverse induction is complete. Since $a_k = \beta_1$ is minimal in $P$, the sequence stops at $k$. Therefore,
\begin{align*}
&hg\ev{ \mathbf{w} } = h\ev{ \beta_1 \circ \ldots \circ \beta_m } \\
& = \beta_1\cdot h\ev{ \beta_2 \circ \ldots \circ \beta_m }  && \text{ by the definition of $h$, since $a_k^{P} = \beta_1$}\\ 
& = \beta_1 \cdot hg\ev{ \widetilde{\mathbf{w}}_{2} } \\ 
& = \beta_1 \cdot \widetilde{\mathbf{w}}_{2} && \text{ by the inductive hypothesis} \\ 
& = \textbf{ w }
\end{align*}

\end{proof}
As a consequence of \Cref{thm:pyramids_containing_edge} and \Cref{rmk:walks_circuits_bij}, note that the number of decomposition pyramids with maximal piece containing an edge $e\in E\ev{D}$ does not depend on $e$. Likewise, the number of decomposition pyramids with maximal piece containing a vertex $u\in V\ev{D}$ does not depend on $u$:
$$ |\mymathbb{dp}^{e}\ev{ D }| = | \mathcal{W}^{e}\ev{D} | = |\mathfrak{C}\ev{D}| \quad \text{ and } \quad  |\mymathbb{dp}^{u}\ev{ D }| = | \mathcal{W}^{u}\ev{D} | = \deg_D^{-}\ev{u} \cdot |\mathfrak{C}\ev{D}|$$

The bijection between decomposition pyramids of a digraph and its closed trails naturally translates into a bijection between decomposition pyramids of an undirected multi-graph and its closed trails (see \Cref{cor:bij_trails_pyramids}). We will further ``flatten" the edges of the pyramids and the closed trails of multi-graphs, to obtain a bijection involving pyramids and closed walks of a simple graph (cf. \Cref{cor:bij_walks_pyramids})

\begin{corollary}
\label{cor:bij_trails_pyramids}
Let $X$ be a connected multi-graph. Let $u\in V\ev{X}$ and $e\in E\ev{X}$. Define the mapping 
$$
h: \mathcal{T}\ev{X}\rightarrow \mathbbm{p}\ev{ X }
$$
defined as in \Cref{thm:pyramids_containing_edge}: $h\ev{\mathbf{w}} = \beta_1 \circ \ldots \circ \beta_m$ where $\ev{ \beta_1, \ldots, \beta_m}$ is the decomposition of $\mathbf{w}$ into (distinct) cycles. Then, the following restrictions are bijections: 
\begin{enumerate}
\item $h \big\vert_{\mathcal{W}^{e}\ev{X}} : \mathcal{W}^{e}\ev{X} \rightarrow \mymathbb{dp}^{e}\ev{ X }$ is a bijection between Eulerian trails of $X$ ending at $e$ and decomposition pyramids of $X$, with maximal piece containing $e$.
\item $h \big\vert_{\mathcal{W}^{u}\ev{X}} : \mathcal{W}^{u}\ev{X} \rightarrow \mymathbb{dp}^{u}\ev{ X }$ is a bijection between Eulerian trails of $X$ at $u$ and decomposition pyramids of $X$, with maximal piece containing $u$.
\end{enumerate}
\end{corollary}
See \hyperlink{proof:bijtrailspyramids}{proof} on p.~\pageref*{proof:bij_trails_pyramids_pf}. 

\begin{remark}
Let $\mymathbb{ip}^{e}\ev{ X }$ be the set of injective pyramids of $X$, i.e., the pyramids $P = \ev{ \Omega, \leq, \ell } \in \mathbbm{p}\ev{\mathcal{B}\ev{X},\mathcal{R}} $ with $\ell$ injective. \Cref{cor:bij_trails_pyramids} applied to sub-multi-hypergraphs of $X$ yields bijections:
\begin{enumerate}
\item $h \big\vert_{\mathcal{T}^{e}\ev{X}} : \mathcal{T}^{e}\ev{X} \rightarrow \mymathbb{ip}^{e}\ev{ X }$ between closed trails of $X$ ending at $e$ and injective pyramids of $X$ with maximal piece containing $e$. 
\item $h \big\vert_{\mathcal{T}^{u}\ev{X}} : \mathcal{T}^{u}\ev{X} \rightarrow \mymathbb{ip}^{u}\ev{ X }$ between closed trails of $X$ at $u$ and injective pyramids of $X$ with maximal piece containing $u$. 
\end{enumerate}
(Similar considerations apply to a digraph.)
\end{remark}

Trails of a multi-graph do not visit the same edge more than once. Given a simple graph $G$, we have a bijection (cf. \Cref{cor:bij_walks_pyramids} below.) featuring walks in which the same edge may be visited more than once and pyramids made of not necessarily distinct pieces. Given a sequence of vertices of a simple graph $G$, there is at most one walk visiting the sequence of vertices in the given order. This means that a walk is uniquely determined by its sequence of vertices.

Let $G$ be a connected simple graph of rank $2$. As defined in Section \ref{sec:preliminaries_for_graphs}, the set of cycles of $G$ is denoted $\mathcal{B}\ev{G}$. Since the graph $G$ is simple, it can not have a digon (a cycle of length $2$). However, walks of a graph of length $2$ up to cyclic permutation of edges, called \textit{edgegons}, are included in the set of pieces of $G$: Let $\mathcal{Q}\ev{G} := \{ [\ev{ v_0,e,v_1,e,v_0 }]_{\rho}: e = \{ v_0,v_1 \} \in E\ev{G} \} $. Write $\widetilde{B}\ev{ G } := \mathcal{B}\ev{ G } \sqcup \mathcal{Q}\ev{G} $ for the union of the set of cycles (of length $\geq 3$) and the set of edgegons of $G$. (We have $[\ev{ v_0,e,v_1,e,v_0 }]_{\rho} = [\ev{ v_1,e,v_0,e,v_1 }]_{\rho}$, for each edge $e = \{ v_0,v_1 \} \in E\ev{G}$ and necessarily, a graph $G$ has $|E\ev{G}|$ many edgegons.)
\begin{example}
\begin{enumerate}
\item As an example, we consider a simple graph, namely, the path of length $2$. 
\begin{figure}[H]
\hspace{6cm} \begin{tikzpicture}
\pic {graph_1};
\end{tikzpicture}
\end{figure}
\vspace{-0.5cm}
\begin{align*}
& \mathcal{U} \ev{G} = \{ \ \ev{ 1 \ e \ 2 \ f \ 3 \ f \ 2 \ e \ 1},\  \ev{ 2\ e\ 1\ e\ 2\ f\ 3\ f\ 2}, \ \ev{ 2\ f\ 3\ f\ 2\ e\ 1\ e\ 2}, \ \ev{ 3\ f\ 2\ e\ 1\ e\ 2\ f\ 3} \ \} \\ 
& \widetilde{\mathcal{B}}\ev{ G } = \{ \ [ \ev{ 1\ e\ 2\ e\ 1 } ]_\rho, [ \ev{ 2\ f\ 3\ f\ 2 } ]_\rho \ \}
\end{align*}

\item Consider the simple graph $G$, obtained by appending an edge to a triangle. We list the pieces of $G$ below. 
\begin{figure}[ht]
\hspace{6cm} \begin{tikzpicture}
\pic {graph_2};
\end{tikzpicture}
\end{figure}
\vspace{-0.2cm}
$$
\widetilde{\mathcal{B}}\ev{ G } = \{ \ [ \ev{ 1\ e\ 2\ e\ 1 } ]_\rho, [ \ev{ 2\ f\ 3\ f\ 2 } ]_\rho, [ \ev{ 1\ g\ 3\ g\ 1 } ]_\rho, [ \ev{ 1\ a\ 4\ a\ 1 } ]_\rho, [ \ev{ 1\ e\ 2\ f\ 3 \ g\ 1} ]_\rho, [ \ev{ 1\ g\ 3\ f\ 2 \ e\ 1 } ]_\rho \ \}
$$
\end{enumerate}
\end{example}
We note that, although non-trail walks of non-simple multi-graphs and digraphs are defined in theory, we will not need them in this paper. For simple graphs, the set of walks and the set of trails are used in different settings. 
\begin{remark}
By Veblen's Theorem (\cite[p. 87]{veblen}), a connected simple graph $G$ has an Eulerian trail, i.e., $\mathcal{W}\ev{G} \neq \emptyset $ if and only if it is obtained by an edge-disjoint union of cycles of length $\geq 3$. 
\end{remark}

Next, we derive a bijection involving walks of a simple graph $G$ and pyramids made of cycles and edgegons.

\begin{corollary}
\label{cor:bij_walks_pyramids}

Let $G$ be a simple connected graph, $u\in V\ev{G}$ and $e\in E\ev{ G}$. Define the mapping 
$$
q: \mathcal{U}\ev{G} \rightarrow \mathbbm{p}\ev{ G }
$$
defined as in \Cref{thm:pyramids_containing_edge}: $q\ev{\mathbf{w}} = \beta_1 \circ \ldots \circ \beta_m$ where $\ev{ \beta_1, \ldots, \beta_m}$ is the decomposition of $\mathbf{w}$ into (not necessarily distinct) cycles and edgegons. 
Then, the following restrictions are bijections: 
\begin{enumerate}
\item $q \big\vert_{\mathcal{U}^{e}\ev{G}} : \mathcal{U}^{e}\ev{G} \rightarrow \mathbbm{p}^{e}\ev{ \widetilde{\mathcal{B}}\ev{ G }  }$ is a bijection between closed walks of $G$ ending at $e$ and pyramids of $G$, with maximal piece containing $e$. 
\item $q \big\vert_{\mathcal{U}^{u}\ev{G}} : \mathcal{U}^{u}\ev{G} \rightarrow \mathbbm{p}^{u}\ev{ \widetilde{\mathcal{B}}\ev{ G }  }$ is a bijection between closed walks of $G$ at $u$ and pyramids of $G$, with maximal piece containing $u$. 
\end{enumerate}
\end{corollary}
See \hyperlink{proof:bijwalkspyramidspf}{proof} on p.~\pageref*{proof:bij_walks_pyramids_pf}. 

\subsection{The Characteristic Polynomial of a Graph as a sum over Trivial Heaps}
We copy below \cite[Corollary 4.5]{krattenthaler}:
\begin{theorem}[Viennot, 1986]
\label{thm:viennot_krattenthaler}
Let $\mathcal{B}$ be a set of pieces with concurrence relation $\mathcal{R}$. Let $\mathcal{B}_0 \subseteq \mathcal{B}$ be a subset. Let $w: \mathcal{B}\rightarrow R$ be a weight function, where $R$ is a ring with unity. Define the weight of a heap as the product of the weights of its pieces. Then,
\begin{enumerate}
\item 
$$
\sum_{\substack{H \in \mathbbm{h}(\mathcal{B} ) \\ \mathcal{M}\ev{H} \subseteq \mathcal{B}_0}} w(H)=\frac{\sum_{T \in \mathbbm{t}(\mathcal{B} \backslash \mathcal{B}_0 )}(-1)^{|T|} w(T)}{\sum_{T \in \mathbbm{t}(\mathcal{B} )}(-1)^{|T|} w(T)}
$$
\item In particular, substituting $\mathcal{B}_0 = \mathcal{B}$,
$$
\sum_{H \in \mathbbm{h}(\mathcal{B} )} w(H)=\frac{1}{\sum_{T \in \mathbbm{t}(\mathcal{B} )}(-1)^{|T|} w(T)}
$$
\item 
$$ 
\log \ew{ \sum_{T \in \mathbbm{t}(\mathcal{B} )}(-1)^{|T|} w(T) } = -\sum_{P \in \mathbbm{p}(\mathcal{B} )} \dfrac{1}{ |P| } w(P)
$$
\end{enumerate}

\end{theorem}
We intend to apply \Cref{thm:viennot_krattenthaler} to the characteristic polynomial of a simple graph $G$ of order $n$ and rank $2$, which is defined as the determinant $\text{det}\ev{t \cdot I_n - \mathbb{A}_G}$, where $I_n$ is the identity matrix of dimension $n$ and $\mathbb{A}_G$ is the adjacency matrix of $G$. We use the notation $\phi_G\ev{t}$ for the characteristic polynomial of $G$. 

In \cite{harary} by Harary and in \cite{sachs} by Sachs, the following theorem is proven, which calculates the coefficients of the characteristic polynomial $G$ in terms of the counts of the elementary subgraphs of $G$: 
\begin{theorem}[Harary-Sachs Theorem for Graphs of rank $2$]
\label{harary_sachs_ordinary}
Call a graph $X$ \textup{elementary} if it is a disjoint union of edges and cycles. Let $\mathcal{E}_d$ be the set of isomorphism classes of elementary graphs with $d\geq 0$ edges. 

Given a graph $G$ of order $n$, then we have the following formula for the co-degree $d$-th coefficient of $\phi_G \ev{ t }$:
$$[t^{n-d}] \phi_G \ev{ t } = \sum_{ X \in \mathcal{E}_d }  \ev{-1}^{c\ev{X}} 2^{z\ev{X}} \genfrac{\lbrack}{\rbrack}{0pt}{0}{G}{X} $$
where
\begin{enumerate}
\item $c\ev{X}$ is the number of components of the elementary graph $X$
\item $z\ev{X}$ is the number of cycles in $X$ of length $\geq 3$
\item $\genfrac{\lbrack}{\rbrack}{0pt}{0}{G}{X} $ is the number of isomorphic copies of $X$ in $G$
\end{enumerate} 
\end{theorem}
We divide the characteristic polynomial by its degree (for a simple graph, the degree is the order of the graph) and obtain the following function:
$$\widetilde{\phi}_\mathcal{H}\ev{t}:=t^{-\deg\ev{\phi_{\mathcal{H}}\ev{t}}}\phi_\mathcal{H}\ev{t}$$

Next, we apply \Cref{thm:viennot_krattenthaler} and obtain equations relating number of walks and sums of heaps. Let $G$ be a simple graph. For each cycle $\beta \in \widetilde{\mathcal{B}}\ev{G}$, we let $w\ev{\beta}:= t^{ - |E\ev{ \beta }| }$, where $|E\ev{ \beta }|$ is the length of $\beta$. Recall that the weight of a heap is defined as the product of the weights of its pieces.
\subsubsection*{Trivial Heaps}
Elementary subgraphs of $G$ are of the form $X = Y_1\sqcup \ldots \sqcup Y_r \sqcup Z_1 \sqcup \ldots \sqcup Z_s \in \mathcal{E}\ev{ G }$, where $\{ Y_i \}_{i=1}^{r}$ are edges and $\{ Z_j \}_{j=1}^{s}$ are cycles of length $\geq 3$. By orienting the cycles in two directions, we can make $2^s$ many trivial heaps from such an elementary subgraph. Therefore, the following equation holds:
\begin{align*}
\widetilde{\phi}_G\ev{t} = \sum_{T\in \mathbbm{t}\ev{G}} \ev{-1}^{c\ev{T}} w\ev{T} = \sum_{X \in \mathcal{E}\ev{ G }  } t^{-|E\ev{X}|}  \ev{-1}^{c\ev{X}} 2^{z\ev{X}} 
\end{align*}
where $X = \emptyset$ is an elementary graph with zero components, with $\ev{ -1 }^{c\ev{X}} 2^{z\ev{X}} = \ev{ -1 }^{0} 2^0 = 1$. 

\subsubsection*{All Pyramids}
Application of \Cref{thm:viennot_krattenthaler} yields:
\begin{corollary}
\label{cor:all_pyramids}
\begin{align*}
-\log\ev{ \widetilde{\phi}_G\ev{t} } & = -\log\ew{ \sum_{X\in \mathbbm{t}\ev{\mathcal{B} }}   \ev{-1}^{c\ev{X}}  w\ev{X}}  = \sum_{P\in \mathbbm{p}\ev{G }} \frac{1}{|P|} w\ev{P}  = \sum_{P\in \mathbbm{p}\ev{G }} \frac{1}{|P|} t^{-|E\ev{P}|} & 
\end{align*}
\end{corollary}

On the other hand, we have the following identity, for the logarithm of the characteristic polynomial:
\begin{theorem}
\label{thm:log_and_all_walks}
$$ -\log\ev{ \widetilde{\phi}_G\ev{t} } = \sum_{ d\geq 1 } \dfrac{1}{d} w_d t^{-d} = \sum_{\mathbf{w}\in \mathcal{U}\ev{G}} \frac{1}{|\mathbf{w}|} t^{-|\mathbf{w}|} $$
\end{theorem}
\begin{proof}
Recall that $M_X$ is the product of factorials of multiplicities of the edges of $X$ (see Section \ref{sec:preliminaries_for_graphs}). 
\begin{align*}
& -\log\ev{ \widetilde{\phi}_G\ev{t} }  =  \sum_{X\in \textup{Inf}^{1}\ev{ \mathcal{H} }}  t^{-|E\ev{X}|}  \dfrac{ |\mathfrak{C}\ev{ X }| }{ M_X } \\ 
& \hspace{3cm}\text{ by \cite[Lemma 13]{cc1} and \cite[Lemma 5]{cc2}}  \\
& = \sum_{X\in \textup{Inf}^{1}\ev{ G }}  t^{-|E\ev{X}|} \dfrac{1}{|E\ev{X}|} \dfrac{ |\mathcal{W}\ev{ X }| }{ M_X }  \\
& \hspace{3cm} \text{since each Eulerian circuit is represented by $|E\ev{X}|$ Eulerian trails}\\ 
& = \sum_{d\geq 0} t^{-d} \dfrac{1}{d} \sum_{X\in \textup{Inf}^{1}_{d}\ev{ G }}   \dfrac{ |\mathcal{W}\ev{ X }| }{ M_X } = \sum_{ \mathbf{w} \in \mathcal{U}\ev{G}} \frac{1}{|\mathbf{w}|} t^{-|\mathbf{w}|} 
\end{align*}
where the last equality follows from the bijective correspondence $\mathcal{U}\ev{G} \leftrightarrow \bigsqcup_{X\in \textup{Inf}^{1}\ev{G}} \mathcal{W}\ev{X} / \unsim_{\mathcal{T}\ev{X}}$, shown in the proof of \Cref{cor:bij_walks_pyramids}, where each equivalence class $[w]_{\unsim_{\mathcal{T}\ev{X}}} \in \bigsqcup_{\textup{Inf}^{1}\ev{G}} \mathcal{W}\ev{X} / \unsim_{\mathcal{T}\ev{X}}$ has size $M_X$. 
\end{proof}

\subsubsection*{All Heaps}
\begin{align*}
\frac{1}{\widetilde{\phi}_G\ev{t}} =\frac{1}{\sum_{T\in \mathbbm{t}\ev{G }}   \ev{-1}^{c\ev{T}} w\ev{T}} =\sum_{H\in \mathbbm{h}\ev{G }} w\ev{H} =\sum_{H\in \mathbbm{h}\ev{G }} t^{-|E\ev{H}|}  
\end{align*}

\subsection*{Pyramids with maximal piece containing a vertex $u$}
A classic theorem in algebraic graph theory measures the number of closed walks of a simple graph $G$ at a vertex $u$ via the characteristic polynomial. In particular, it is known that (\cite[p.~36]{viennot_french_paper} \cite{viennot_talk, abdesselam_brydges}) the Heaps of Pieces framework yields a proof of this theorem, stated below. (In the next section, we use the same approach with hypergraphs.) 
\begin{theorem}
\label{thm:walks_at_u_and_quotient}
Let $G$ be a simple graph of rank $k=2$. Let $u\in V\ev{G}$ be a vertex. Then,
$$\dfrac{\widetilde{\phi}_{G-u}\ev{t}}{\widetilde{\phi}_G\ev{t}} = \sum_{ d\geq 0  } w^{u}_d t^{-d}$$
In other words, the number of closed walks of $G$ at $u$ of length $d$ is equal to the coefficient of $t^{-d}$ in the fraction $\dfrac{\widetilde{\phi}_{G-u}\ev{t}}{\widetilde{\phi}_G\ev{t}}$. 
\end{theorem}

\begin{proof}
Let $\widetilde{\mathcal{B}}^u\ev{ G }$, or shortly $\widetilde{\mathcal{B}}^u$ be the union of cycles and edgegons of $G$, containing $u$. By \Cref{thm:viennot_krattenthaler}, we have:
\begin{align*}
& \dfrac{\widetilde{\phi}_{G-u}\ev{t}}{\widetilde{\phi}_G\ev{t}} = \dfrac{ \sum_{ T \in \mathbbm{t}\ev{ \widetilde{\mathcal{B}}  \setminus \widetilde{\mathcal{B}}^u  }}  \ev{-1}^{c\ev{T}} w\ev{T} }{ \sum_{T\in \mathbbm{t}\ev{ \widetilde{\mathcal{B}} }}  \ev{-1}^{c\ev{T}} w\ev{T} } = \sum_{ \substack{ H \in \mathbbm{h}\ev{ \widetilde{\mathcal{B}}  }  \\ \mathcal{M}\ev{ H }\subseteq \widetilde{\mathcal{B}}^u  } } w\ev{H} = \sum_{ \substack{ P \in \mathbbm{p}\ev{ \widetilde{\mathcal{B}}  }  \\ \mathcal{M}\ev{ P }\subseteq \widetilde{\mathcal{B}}^u  } } w\ev{P} \\
&  \hspace{3cm} \text{ since every pair of elements labeled by pieces from $\widetilde{\mathcal{B}}^u$ are concurrent} \\
& = \sum_{ \mathbf{w} \in \mathcal{U}^{u}\ev{G}   } t^{- | \mathbf{w} | }  = \sum_{ d\geq 0  } w^{u}_d t^{-d} & \hspace{-8cm} \text{ by \Cref{cor:bij_walks_pyramids} Part 2}  \hspace{8cm}
\end{align*}
\end{proof}

\subsection*{Pyramids with maximal piece containing an edge $e$}
Similarly, for a simple graph $G$ and an edge $e\in E\ev{G}$, the number of walks ending at $e$ can be measured via the characteristic polynomial. Given a simple graph $\mathcal{H} = \ev{ V\ev{\mathcal{H}}, E\ev{\mathcal{H}} }$, then a simple graph $G$ is a \textit{spanning subgraph} of $\mathcal{H}$ provided $V\ev{G} = V\ev{\mathcal{H}}$ and $ E\ev{ G } \subseteq E\ev{ \mathcal{H} }$ (\cite[p.~46]{bondy}) In particular, we write $\mathcal{H}-e$ for the subgraph obtained deleting $e$, without removing any vertex. 

Let $w_{d}^{e}$ be the number of closed walks of a simple graph, of length $d\geq 0$, ending at $e$. 

\begin{theorem}
\label{thm:walks_at_e_and_quotient}
Let $G$ be a simple graph of rank $k=2$. Let $e\in E\ev{G}$ be an edge. Let $G - e$ be the spanning subgraph of $G$, obtained by removing $e$. Then, 
$$\dfrac{\widetilde{\phi}_{G-e}\ev{t}}{\widetilde{\phi}_G\ev{t}} = \sum_{ d\geq 0  } w^{e}_d t^{-d}$$
In other words, the number of closed walks of $G$ ending at $e$ of length $d$ is equal to the coefficient of $t^{-d}$ in the fraction $\dfrac{\widetilde{\phi}_{G-e}\ev{t}}{\widetilde{\phi}_G\ev{t}}$. 
\end{theorem}

\begin{proof}
Let $\widetilde{\mathcal{B}}^e\ev{ G }$, or shortly $\widetilde{\mathcal{B}}^e$ be the union of cycles and edgegons of $G$, containing the edge $e$. (There is a single edgegon.) By \Cref{thm:viennot_krattenthaler}, we get:
\begin{align*}
& \dfrac{\widetilde{\phi}_{G-e}\ev{t}}{\widetilde{\phi}_G\ev{t}} = \dfrac{ \sum_{ T \in \mathbbm{t}\ev{ \widetilde{\mathcal{B}}  \setminus \widetilde{\mathcal{B}}^e  }}  \ev{-1}^{c\ev{T}} w\ev{T} }{ \sum_{T\in \mathbbm{t}\ev{ \widetilde{\mathcal{B}} }}  \ev{-1}^{c\ev{T}} w\ev{T} } = \sum_{ \substack{ H \in \mathbbm{h}\ev{ \widetilde{\mathcal{B}}  }  \\ \mathcal{M}\ev{ H }\subseteq \widetilde{\mathcal{B}}^e  } } w\ev{H} = \sum_{ \substack{ P \in \mathbbm{p}\ev{ \widetilde{\mathcal{B}}  }  \\ \mathcal{M}\ev{ P }\subseteq \widetilde{\mathcal{B}}^e  } } w\ev{P} \\
& \hspace{3cm}\text{ since every pair of elements labeled by pieces from $\widetilde{\mathcal{B}}^e$ are concurrent} \\
& = \sum_{ \mathbf{w} \in \mathcal{U}^{e}\ev{G}} t^{- | \mathbf{w} | } = \sum_{ d\geq 0  } w^{e}_d t^{-d}  && \hspace{-8cm}\text{ by \Cref{cor:bij_walks_pyramids} Part 1 }  \hspace{8cm}
\end{align*}
\end{proof}

By an abuse of notation, let us write $\mathcal{M}\ev{P}$ for the unique piece of a pyramid $P \in \mathbbm{p}\ev{G }$. By combining \Cref{thm:log_and_all_walks} and \Cref{cor:all_pyramids}, we obtain an interesting identity: 
\begin{corollary}
\label{cor:interesting_identity}
$$
\sum_{P\in \mathbbm{p}\ev{G }} \frac{1}{|P|} t^{-|E\ev{P}|} = \sum_{ P \in \mymathbb{p} \ev{ G }  } \frac{ |V\ev{\mathcal{M}\ev{P}}| }{|E\ev{P}|} t^{-|E\ev{P}|} 
$$
where $| E\ev{P} |$ is the number of edges and $|P|$ is the number of pieces in a pyramid. In particular, for each $d\geq 0$, 
$$
\sum_{P\in \mathbbm{p}_d\ev{G }} \frac{1}{|P|} = \frac{1}{d} \sum_{ P \in \mymathbb{p}_d \ev{ G }  }  |V\ev{\mathcal{M}\ev{P}}|
$$
where $\mymathbb{p}_{d}\ev{ G } $ is the set of pyramids with $d\geq 0$ edges. 
\end{corollary}
See \hyperlink{proof:interestingidentitypf}{proof} on p.~\pageref*{proof:interesting_identity_pf_page}. 

\subsection{An Example: Heaps of Pieces of a Simple Graph}

Let $G=K_4^{\ev{2}}$ be the complete graph on $4$ vertices. 

\begin{enumerate}
\item \textbf{Trivial Heaps.} Using Sagemath (\cite{sagemath}), we calculate: 
\begin{align*}
\widetilde{\phi}_G\ev{t} = 1-6t^{-1}-8t^{-2}-3t^{-4}
\end{align*}
The trivial heaps' contribution to the coefficient of $t^{-4}$ are,
$$
\sum_{T\in \mathbbm{t}\ev{G}} \ev{-1}^{c\ev{T}} w\ev{T} = \sum_{ X \in \mathcal{E}_d }  \ev{-1}^{c\ev{X}} 2^{z\ev{X}} \genfrac{\lbrack}{\rbrack}{0pt}{0}{G}{X} = \ldots + \ev{-1}^{1} \cdot 2 \cdot 3 t^{-4} + \ev{-1}^{2} 3 t^{-4} = -3 t^{-4}
$$
where the trivial heaps are formed from a cycle of length $4$ and a disjoint union of two edges, respectively. 
\item \textbf{All Pyramids.}
$$\log\ev{\widetilde{\phi}_G\ev{t}} = \log\ev{1-6t^{-1}-8t^{-2}-3t^{-4}} = -6 t^{-2}-8t^{-3}-21t^{-4}-\ldots $$
The contributions of walks and pyramids to the term $21t^{-4}$ are,
\begin{align*}
& - \sum_{\mathbf{w}\in \mathcal{U}\ev{G}} \frac{1}{|\mathbf{w}|} t^{-|\mathbf{w}|} = \ldots - t^{-4} \ew{ \frac{1}{4}  \cdot 4 \cdot 2 \cdot 3 + \frac{1}{4} \cdot 2 \cdot 1 \cdot 6 + \frac{1}{4} \cdot (1+2+1) \cdot 12 } - \ldots \\
& - \sum_{P\in \mathbbm{p}\ev{G }} \frac{1}{|P|} w\ev{P} = \ldots - t^{-4} \ew{ \dfrac{1}{1} \cdot 2 \cdot 3 + \dfrac{1}{2}\cdot 6 + \dfrac{1}{2} \cdot 2 \cdot 12 } - \ldots
\end{align*}
where pyramids are formed from a cycle of length $4$, a single edge and a path of length $2$, respectively. 
\item \textbf{All heaps.}
$$\frac{1}{\widetilde{\phi}_G\ev{t}} \underset{\text{Sage}}{=} 1+6 t^{-2}+8t^{-3}+39t^{-4}+\ldots $$
The contributions of all heaps to the term $39t^{-4}$ are,
$$
\sum_{H\in \mathbbm{h}\ev{\mathcal{B} }} t^{-|E\ev{H}|} = \ldots + t^{-4} \ev{ 6 + 6 + 2 \cdot 12 + 3 } + \ldots  = \ldots + 39t^{-4} + \ldots 
$$
where heaps are formed from a cycle of length $4$, a single edge, a path of length $2$ and a disjoint union of two edges, respectively. 
\end{enumerate}

\section{Heaps of Pieces on Hypergraphs}
\label{sec:heap_hyper}
\subsection{Preliminaries}
In \cite{cc1} and \cite{cc2}, a Harary-Sachs Theorem for hypergraphs of rank $k\geq 2$ was proven (\Cref{thm:harary_sachs_hyper} below), where the codegree $d$ coefficient of the characteristic polynomial $\phi_\mathcal{H}\ev{t}$ of $\mathcal{H}$ is calculated via infragraphs of $\mathcal{H}$ on $d$ edges and their associated coefficients (see Definition \ref{def:veblen_multi_hypergraph}). In this section, we give an alternative version of the Harary-Sachs Theorem for Hypergraphs. Hence, we include some definitions below and restate the version of the Harary-Sachs Theorem shown in \cite{cc1}. 

\subsubsection{Number of copies of an infragraph in a simple hypergraph}
Given a Veblen hypergraph $H\in \mathcal{V}^{m}$ with $m\geq 1$ components, let $H = \bigsqcup_{i=1}^{t} \bigsqcup_{j=1}^{\nu_i} G_{ij}$, where each $\{G_{i1},\ldots, G_{i \nu_i}\}$ is an isomorphism class of the components of $H$, for $i=1,\ldots,t$ and $t\leq m$. Define the number $\nu_H = \prod_{i=1}^{t} \nu_i !$. We note that the number of linear orderings of the components of $H$, up to isomorphism, is given by the number $\mu_H := \displaystyle\binom{m}{\nu_1,\dots,\nu_t} = \frac{m!}{\nu_H} $.

Now, we define $\ev{\# H\subseteq \mathcal{H}}$, which appears in the Harary-Sachs Theorem for hypergraphs (\Cref{thm:harary_sachs_hyper}), following the notation of \cite[p.~10, Lemma 12]{cc1}. The quantity $\ev{\# H\subseteq \mathcal{H}}$ measures the number of ways to ``stack'' a connected Veblen hypergraph $H$ on $\mathcal{H}$:

\begin{definition}
Let $\mathcal{H}$ be a simple hypergraph of rank $k$. Let $G\in \mathcal{V}^{1}$ be a connected Veblen $k$-graph. Then, 
$$ \ev{\# G \subseteq \mathcal{H}} := |\{ X\in \textup{Inf}\ev{\mathcal{H}}: X \cong G \}| =  \dfrac{ \lav \text{Aut}\ev{ \underline{G} } \rav }{  \lav  \text{Aut}\ev{G} \rav }  \cdot  \genfrac{\lbrack}{\rbrack}{0pt}{0}{\mathcal{H}}{\underline{G}}
$$
where $\underline{G}$ is the flattening of $G$, as in \Cref{def:flattening}. If $H = G_1 \sqcup \ldots G_m \in \mathcal{V}^{m}$ is disconnected, i.e., $m\geq 2$, then we define 
$$
\ev{ \# H \subseteq \mathcal{H} } := \dfrac{1}{\nu_H} \prod_{i=1}^{m} \ev{ \# G_i \subseteq \mathcal{H} } =  \dfrac{\mu_H}{m!} \prod_{i=1}^{m} \ev{ \# G_i \subseteq \mathcal{H} }
$$
\end{definition}
\begin{example}
Consider the complete hypergraph $\mathcal{H}$, of rank $k=3$ and order $n=4$, in Figure \ref{fig:pound_symbol_example}. The value $\ev{ \# G \subseteq \mathcal{H} }$ is calculated below, for an infragraph $G$.
\begin{figure}[ht]
~ \qquad \qquad \qquad \qquad
\begin{subfigure}[t]{0.3\textwidth}%
\centering
\includegraphics[width=1.1in]{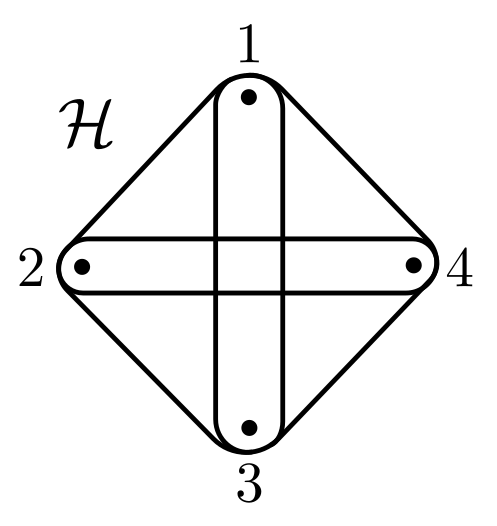} 
\end{subfigure}
~ \quad
\begin{subfigure}[t]{0.3\textwidth}%
\centering
\raisebox{4mm}{ \includegraphics[width=1.2in]{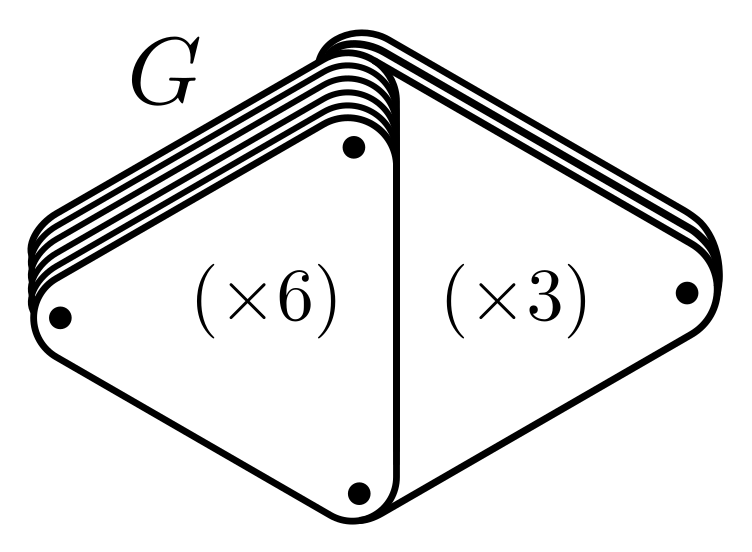}  }
\end{subfigure}
\caption{The host hypergraph $\mathcal{H}$ and $G \in \mathcal{V}^{1}_9$, a connected Veblen hypergraph on 9 edges.}
\label{fig:pound_symbol_example}
\end{figure}
$$
\ev{\# G \subseteq \mathcal{H}} =  \dfrac{ \lav \text{Aut}\ev{ \underline{G} } \rav }{  \lav  \text{Aut}\ev{G} \rav }  \cdot \genfrac{\lbrack}{\rbrack}{0pt}{0}{\mathcal{H}}{\underline{G}} = \dfrac{ 2 \cdot 2 }{ 2 } \cdot \binom{4}{2} = 12
$$
\end{example}
\subsubsection{Rooted stars and Rootings}
Let $X = \ev{ [n], E, \varphi } \in \mathcal{V}^{\infty}_d$ be a Veblen hypergraph with $d \geq 0 $ edges. Given an edge $e\in E\ev{X}$, for each $u\in e$, we define the \textit{$u$-rooted directed star of $e$}, denoted $S_u\ev{ e }$, as the digraph $T$ with $V\ev{T} = \varphi(e)$ and $E\ev{T} = \{ \ev{ u,v} : v\in \varphi(e) \text{ and } v\neq u\}$. A \textit{rooting} of $X$ is a $d$-tuple of stars $\ev{ S_{v_1}\ev{ e_1 },\ldots, S_{v_d}\ev{ e_d } }$ such that
\begin{itemize}
\item[\textit{i)}] $(\forall i \in [d-1]) (v_i\leq v_{i+1})$ 
\item[\textit{ii)}] $e_1,\ldots,e_d$ are pairwise distinct and $E = \{ e_1,\ldots,e_d\}$. 
\item[\textit{iii)}] $[n]=\{v_1,\ldots,v_d\}$
\item[\textit{iv)}] $(\forall i \in [d]) (v_i \in e_i)$
\end{itemize}

From each rooting $R = \ev{ S_{v_1}\ev{ e_1 },\ldots, S_{v_d}\ev{ e_d } }$ of Veblen hypergraph $X$, we obtain a digraph by taking the union of the edges of stars, adding multiplicities when two edges are parallel. Formally, let $D_R$ be the digraph with $V(D_R) = V(X)$, $E(D_R) = \{(v,w,j): v=v_j, w \in e_j \setminus \{v\}\}$, and $\psi(v,w,j)=(v,w)$ for each $v,w \in [n]$.  Note that $ m_{D_R}\ev{u,v} = | \{ j \in [d] : u = v_j, v\in e_j \setminus \{u\}  \}|$, for each $u,v \in [n]$, i.e., the number of stars $S_u(e_j)$ in $R$ so that $v \in e_j \setminus \{u\}$. We say that $R$ is an \textit{Eulerian rooting} if $D_R$ is Eulerian. Let $ \mathfrak{R}\ev{X} $ be the set of Eulerian rootings of $X$. For a rooting $R\in \mathfrak{R}\ev{X} $, let $i$ be chosen as the root of $r_i$-many stars,
$$ 
S_{i}\ev{e_{i1}},\ldots,S_{i}\ev{e_{ir_i}} 
$$
Then, the rooting $R$ has the form 
$$
\ev{S_{1}\ev{e_{11}},\ldots,S_{1}\ev{e_{1r_1}},\ldots, S_{n}\ev{e_{n1}},\ldots,S_{n}\ev{e_{nr_n}}} 
$$
By the definition of a rooting, it follows that for each pair of parallel edges $e,e'$ and each vertex $i\in [n]$, we have $S_{i}\ev{e} = S_{i}\ev{e'}$. For example, if $X $ is the infragraph with $V\ev{X} = [2]$ and $E\ev{X} = \{e_1,e_2,e_3,e_4\}$ with $\varphi(e_j) = \{1,2\}$ for each $j$, then $\mathfrak{R}\ev{X} = \{R\}$, where 
$$ 
R = \ev{ S_1\ev{e_1}, S_1\ev{e_2}, S_2\ev{e_3}, S_2\ev{e_4}}
$$
since $S_i\ev{e_1} = S_i\ev{e_2} = S_i\ev{e_3} = S_i\ev{e_4}$, for $i=1,2$. Two rootings are considered equivalent, if one is obtained from the other by interchanging two stars with the same root (and possibly non-parallel edges). Formally, we define an equivalence relation $\equiv$ on $\mathfrak{R}\ev{ X }$, as the transitive closure of the relations,
\begin{align*}
\ev{S_{v_1}\ev{e_1},\ldots, S_{v_i}\ev{e_i}, S_{v_{i+1}}\ev{e_{i+1}} \ldots ,S_{v_d}\ev{e_d}} & \equiv \ev{S_{v_1}\ev{e_1},\ldots, S_{v_{i+1}}\ev{e_{i+1}},  S_{v_i}\ev{e_i}, \ldots ,S_{v_d}\ev{e_d}} 
\end{align*}
where $v_i = v_{i+1}$ and $1 \leq i \leq d-1$. Let $\mathfrak{R}\ev{X} / \!\equiv$ be the set of rootings of $X$, up to the equivalence relation $\equiv$. Instead of $[R]_\equiv\in \mathfrak{R}\ev{X}/ \!\equiv$, we succinctly write $\mathbf{R}\in \overline{\mathfrak{R}}\ev{X}$. Note that $R\equiv R'$ implies $D_R \approx D_{R'}$. Hence, we the following are well-defined, independent of the choice of representative:
$$
D_\mathbf{R} := [D_R]_{\approx} \text{ and } \mathfrak{C}\ev{ D_\mathbf{R} } := \mathfrak{C}\ev{ D_R } \text{ for any $R \in \mathbf{R}$.}
$$
Let $X \in \mathcal{V}^1$ be a connected Veblen hypergraph. Let $E\ev{X} / \unsim = \{ \mathbf{e}_1, \ldots, \mathbf{e}_s \}$ be the edges of $X$, up to the equivalence relation $\sim$ of parallelism. Let $R\in \mathfrak{R}\ev{X}$ be a rooting. For each vertex $u \in V(X)$ and $j\in [s]$, let $\mathfrak{s}^{R}_u\ev{\mathbf{e}_j}$ be the number of times $u$ is chosen as the root of $\mathbf{e}_j$. We use the notation
$$ 
\Gamma^{R}_u := \prod_{j=1}^{s} \mathfrak{s}^{R}_u\ev{ \mathbf{e}_j }! 
$$

\begin{remark}
\label{rmk:previous_work}
By \cite{cc1}, the number of stars $S_u\ev{e}$ in $R$, rooted at $u$ is given by 
$$ 
\mathfrak{s}_u\ev{X} := \sum_{j=1}^{s} \mathfrak{s}^{R}_u\ev{ \mathbf{e}_j } = \dfrac{\deg^{-}_{ D_R }\ev{u}}{k-1} = \dfrac{ \deg_X\ev{u} }{ k }
$$ 
which does not depend on the rooting $R \in \mathfrak{R}\ev{X}$. Then, the size of the equivalence class $\mathbf{R} \in \overline{\mathfrak{R}}\ev{X} $ can be expressed as
$$ | \mathbf{R} |  = \prod_{u\in V\ev{X}} \dfrac{ \mathfrak{s}_u\ev{X}! }{ \Gamma^{ \mathbf{R} }_u} $$
where $ \Gamma^{\mathbf{R}}_u := \Gamma^{R}_u$ is well-defined, independent of the choice of representative $R\in \mathbf{R}$.
\end{remark}

\subsubsection{The Associated Coefficient}
\begin{definition}
Given a connected Veblen hypergraph $G$, we define the \textit{associated coefficient} of $G$ (\cite{cc1,cc2}):
$$C_G=\sum_{R\in \mathfrak{R}\ev{G}}\dfrac{\tau_{D_R}}{\prod_{v\in V\ev{D_R}} \text{deg}^{-}_{D_R} \ev{v}}$$
where $\tau_D$ is the number of arborescences of a digraph $D$ rooted at a vertex $u\in V\ev{D}$, which is independent of the choice of the root, by the B.E.S.T. Theorem (\cite[Theorem 6, p.~213]{best}). Again, by the B.E.S.T. Theorem, we have an alternative formula,
$$
C_G = \sum_{R \in \mathfrak{R}\ev{G}}\dfrac{ | \mathfrak{C}\ev{D_R} | }{\prod_{v\in V\ev{D_R}} \text{deg}^{-} \ev{v} ! } = \sum_{\mathbf{R} \in  \overline{\mathfrak{R}}\ev{G} } | \mathbf{R} | \cdot \dfrac{ | \mathfrak{C}\ev{D_\mathbf{R}} | }{\prod_{v\in V\ev{D_\mathbf{R}}} \text{deg}^{-} \ev{v} ! }
$$
If a Veblen hypergraph $H = G_1 \sqcup \ldots \sqcup G_m$ is disconnected, then the associated coefficient is defined multiplicatively as,
$$C_H = \prod_{i=1}^{m} C_{G_i}$$
\end{definition}

\subsubsection{The Harary-Sachs Theorem for Hypergraphs}
Let $\mathcal{H}$ be a simple hypergraph of order $n$ and rank $k\geq 2$. 
\begin{definition}[The Characteristic Polynomial of a Hypergraph]
\label{def:char_poly}
The (normalized) \textit{adjacency hypermatrix} $\mathbb{A}_\mathcal{H} = \ev{a_{i_1 \ldots i_k}}$ of $\mathcal{H}$ is the $n$-dimensional symmetric $k$-hypermatrix with entries
$$ 
a_{i_1 \ldots i_k} = \begin{cases}  \ffrac{ 1 }{(k-1)!} & \text{ if } \{i_1,\ldots,i_k\} \in E\ev{\mathcal{H}} \\ 0 & \text{otherwise}
\end{cases}
$$
The Lagrangian of $\mathcal{H}$ is the polynomial
$$
F_\mathcal{H}\ev{ \mathbf{x} } = \sum_{i_1 \ldots i_k = 1 }^{n} a_{i_1 \ldots i_k} x_{i_1} \cdots x_{i_k}
$$
Letting $\nabla$ denote the $n$-dimensional gradient $\C[x_1,\ldots,x_n] \rightarrow \prod_{j=1}^{n} \C[x_1,\ldots,x_n]$, we define the characteristic polynomial as the $n$-dimensional resultant: 
$$
\textup{res} \nabla \ev{ t\cdot F_\mathcal{I} - F_\mathcal{H} } 
$$
where $\mathcal{I} = \ev{ \delta_{i_1,\ldots,i_k} }$ is the $n$-dimensional identity hypermatrix, with entries $\delta_{i_1,\ldots,i_k} = \begin{cases} 1 & \text{ if } i_1 = \ldots = i_k \\ 0 & \text{ otherwise} \end{cases}$. 
\end{definition}

As shown in \cite[Proposition 1, p.~1306]{qi}, the degree of the characteristic polynomial is 
$$
N := \deg\ev{\phi_{\mathcal{H}}\ev{t}} = n\ev{k-1}^{n-1}
$$

By \cite[p.~95, Theorem 3.1]{usingalg}, we have the following property of the $n$-dimensional resultant: Given homogenous polynomials $F_1,\ldots,F_n$ of degrees $d_1,\ldots, d_n$, respectively,
$$  \text{res} \ev{ t \cdot F_1, \ldots, t\cdot F_n} =t^{\sum_{i=1}^{n} d_1 \cdots d_{i-1} \cdot d_{i+1} \cdots d_n } \cdot \text{res}\ev{F_1, \ldots, F_n} $$
Putting $d_i = k-1$ for each $i=1,\ldots,n$, we have:
$$
\phi_\mathcal{H}\ev{t} = \text{res} \nabla \ev{  t\cdot F_\mathcal{I} - F_\mathcal{H} } = \text{res} \nabla \ev{  t\cdot \ev{ F_\mathcal{I} - t^{-1}\cdot F_\mathcal{H}} } = t^{n \ev{ k-1}^{n-1} } \cdot \text{res}\nabla \ev{ F_\mathcal{I} - t^{-1}\cdot F_\mathcal{H}}
$$
In particular, we have $\widetilde{\phi}_\mathcal{H}\ev{t} = t^{-N} \cdot \phi_\mathcal{H}\ev{t} = \text{res}\nabla \ev{ F_\mathcal{I} - t^{-1}\cdot F_\mathcal{H}}$. 

We state the Harary-Sachs Theorem from \cite{cc1}:
\begin{theorem}[Harary-Sachs Theorem for Hypergraphs]
\label{thm:harary_sachs_hyper}
Let $\mathcal{H}$ be a hypergraph. The codegree $d$ coefficient of the characteristic polynomial, i.e., the coefficient of $t^{N-d}$, where $N$ is the degree of the characteristic polynomial, is calculated as
$$[t^{N-d}] \ev{\phi_{\mathcal{H}}\ev{t}}=\sum_{H\in \mathcal{V}_{d}^{\infty}}  \ev{-\ev{k-1}^{n}}^{c\ev{H}} C_H  \ev{\# H\subseteq \mathcal{H}}$$
In particular, 
$$ \widetilde{\phi}_{\mathcal{H}}\ev{t} = \sum_{H\in \mathcal{V}^{\infty}} t^{- | E\ev{H} |} \ev{-\ev{k-1}^{n}}^{c\ev{H}} C_H  \ev{\# H\subseteq \mathcal{H}} $$
\end{theorem}

\subsection{Kocay's Lemma and Alternative Expression of the Characteristic Polynomial}
\label{subsec:kocay_lemma}
For some Veblen hypergraphs $X$, it is possible to decompose $X$ into smaller Veblen hypergraphs. The set of decompositions of $X$ into connected Veblen hypergraphs is defined as follows: 
\begin{definition}
\label{def:parallel_equiv}
Given a Veblen multi-hypergraph $X = \ev{ V, E, \varphi} \in \mathcal{V} $, define
$$\mathcal{S}\ev{X} := \left\{ \{  A_1, \ldots, A_m \} : \emptyset \neq A_i \subseteq E \text{ and } A_i \in \mathcal{V}^{1} \text{ and } E = \bigsqcup_{i=1}^{m} A_i \text{ is a disjoint union}  \right\} $$
\end{definition}

Given the equivalence relation $\unsim$ on $E$ denoting parallelism of edges, we will define further equivalence relations on $\mathcal{P}(E)$ and $\mathcal{P}(\mathcal{P}(E))$ which, by an abuse of notation, will also be denoted by the symbol $\unsim$.

\begin{definition}

\phantom{a}

\begin{enumerate}
\item[\textit{i)}] Let $\text{Sym}_{\unsim} \ev{ E } $ be the subgroup of $\text{Sym}\ev{E}$ defined by
$$ \text{Sym}_{\unsim} \ev{ E }  := \{ \sigma\in \text{Sym}\ev{E}: \sigma\ev{ e } \unsim e \text{ for each } e\in E \}$$
\item[\textit{ii)}] Denote by $\unsim$ an equivalence relation on the power set $\mathcal{P}\ev{E}$, as follows: Given $A,B \subseteq E$, 
$$ 
A \ \unsim \ B \text{ if and only if } \exists \sigma \in \text{Sym}_{\unsim} \ev{ E } \text{ such that } \sigma\ev{A} = B
$$
In particular, for each $S = \{ A_1,\ldots, A_m \} \in \mathcal{S}\ev{X}$, we have $S\subseteq \mathcal{P}\ev{E}$ and so, the following cardinality is well-defined:
$$ 
\alpha_{ S } := \prod_{j=1}^{r} | [B_j]_{\unsim} | 
$$
where $S/\unsim = \{ [B_{1}]_{\unsim}, \ldots,  [B_{r}]_{\unsim} \}$ and $ \{ B_1 ,\ldots, B_r\} $ is any choice of representatives of $S/\unsim$ and $1\leq r\leq m$. 
\item[\textit{iii)}] Define an equivalence relation $\unsim$ on $\mathcal{S}\ev{X} \subseteq \mathcal{P}(\mathcal{P}(E))$, as follows: Given $S, T \in \mathcal{S}\ev{X}$, 
$$ 
S \unsim T \text{ if and only if } \exists \sigma \in \text{Sym}_{\unsim} \ev{ E } \text{ so that } \sigma(S) = T
$$
\end{enumerate}
\end{definition}
We use the notation $\mathbf{S}\in \widetilde{\mathcal{S}}\ev{X}$ instead of $[S]_{\unsim} \in \mathcal{S}\ev{X} / \unsim$.
\begin{remark}
The following numbers are well-defined, independent of the choice of representative $S \in \mathbf{S}$ for each equivalence class $\mathbf{S}\in \widetilde{\mathcal{S}}\ev{X}$:
\begin{enumerate}
\item $\alpha_{ \mathbf{S} } := \alpha_{ S }$ for any $S\in \mathbf{S}$.
\item $\nu_{ \mathbf{S} } := \nu_{ S }$ for any $S\in \mathbf{S}$.
\end{enumerate}

\end{remark}

\begin{example}
Consider the simple ``host" hypergraph and its infragraph in Figure \ref{fig:basic_infra_host}. 
\begin{figure}[h!]
~ \qquad \qquad \qquad  \qquad
\begin{subfigure}[t]{0.3\textwidth}%
\centering
\includegraphics[width=1in]{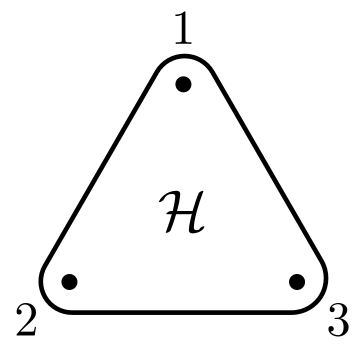} 
\end{subfigure}
~ \quad
\begin{subfigure}[t]{0.3\textwidth}%
\centering
\includegraphics[width=1.1in]{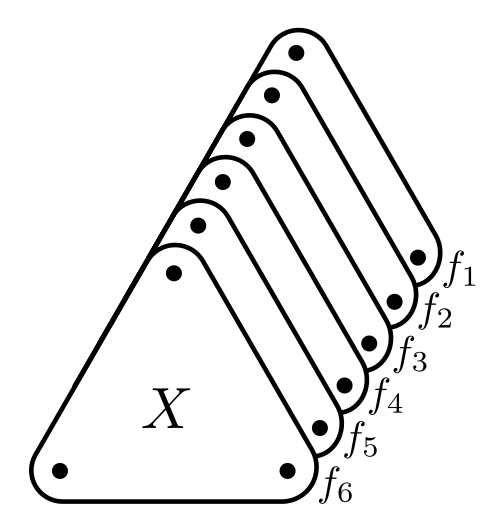} 
\end{subfigure}
\caption{The simple host hypergraph $\mathcal{H}$ and its infragraph $X \in \text{Inf}\ev{\mathcal{H}}$ }
\label{fig:basic_infra_host}
\end{figure}

We have $E\ev{X} = \{f_i\}_{i=1}^{6}$ and $\varphi\ev{f_i} = \{1,2,3\}$, for $i=1,\ldots,6$. Then, there are $ S, T \in \mathcal{S}\ev{X} $ such that $S\unsim T$, where:
$$
S = \{ \{ f_1, f_2, f_3\} , \{ f_4, f_5, f_6\} \} \text{ and } T = \{ \{ f_1, f_2, f_4\} , \{ f_3, f_5, f_6\} \} 
$$
We can also see that $ | \textbf{S} | = \binom{6}{3}$, $\alpha_{\textbf{S}} = 2!$ and $ \widetilde{\mathcal{S}}\ev{X}  = \{ [ \{ \{ f_1, f_2, f_3\} , \{ f_4, f_5, f_6\} \} ]_{\unsim}, [ \{ \{ f_1, f_2, f_3 , f_4, f_5, f_6\} \} ]_{\unsim} \} $.
\end{example}

The definition below is a reformulation of \cite[Definition 2.10]{thatte}, in which tuples are used instead of sets: Let $1\leq r\leq m$ be fixed. Let $X\in \mathcal{V}_d^{r}$ and $P = P_1\sqcup \ldots \sqcup P_m \in \mathcal{V}_d^{m}$ be given, where $P_i$ is connected for each $i = 1,\ldots,m$.
$$\mathcal{S}\ev{P,X} := \{ S = \{  A_1, \ldots, A_m \} \in \mathcal{S} \ev{X} : P_i \cong A_i \} $$
We write $P\vdash X$ if and only if $\mathcal{S}\ev{P,X}\neq \emptyset$. The relation $\vdash$ forms a ranked poset on $\mathcal{V}^{\infty}$, with the rank function $r\ev{ P } := -c\ev{P}$. Note that 
$$ 
\mathcal{S} \ev{ X } = \bigsqcup_{P\vdash X} \mathcal{S}\ev{P,X} 
$$
is a disjoint union, and the equivalence relation $\unsim$ restricts to $\mathcal{S}\ev{P,X}/\unsim := \widetilde{\mathcal{S}}\ev{ P, X }$. In particular, $\widetilde{\mathcal{S}}\ev{X, X}$ is a singleton.

\begin{example}
Consider the simple host graph $\mathcal{H}$ and its infragraph $X$ (see Figure \ref{fig:infragraph_example_2_12}).

\begin{figure}[ht]
~ \qquad \qquad \qquad
\begin{subfigure}[t]{0.3\textwidth}%
\centering
\includegraphics[width=2in]{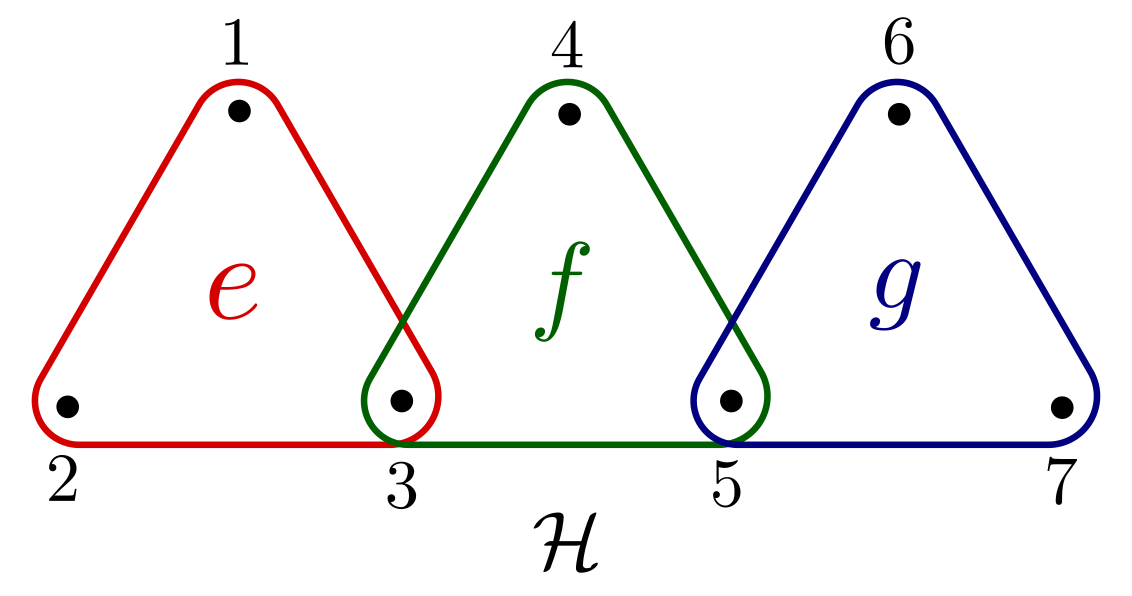} 
\end{subfigure}
~ \qquad\qquad
\begin{subfigure}[t]{0.3\textwidth}%
\centering
\includegraphics[width=2in]{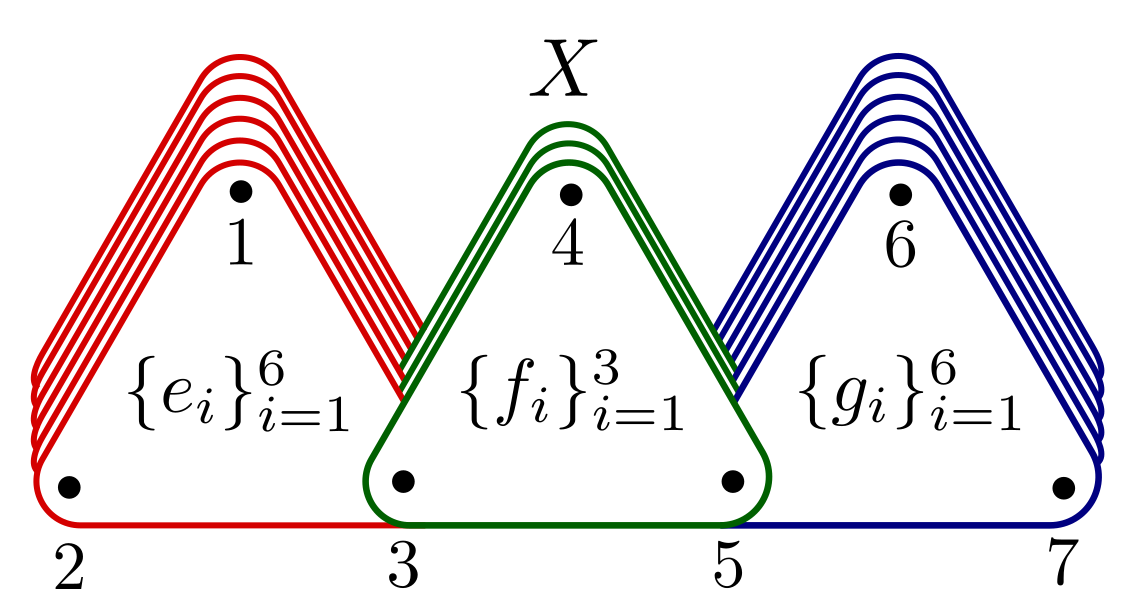} 
\end{subfigure}
\caption{The simple host hypergraph $\mathcal{H}$, its infragraph $X \in \text{Inf}\ev{\mathcal{H}}$.}
\label{fig:infragraph_example_2_12}
\end{figure}

Consider the Veblen hypergraph $P$, which partitions $X$, i.e., $P\vdash X$ (see Figure \ref{fig:infragraph_example_2_3}). 
\begin{figure}[ht]
\centering
\includegraphics[width=3.3in]{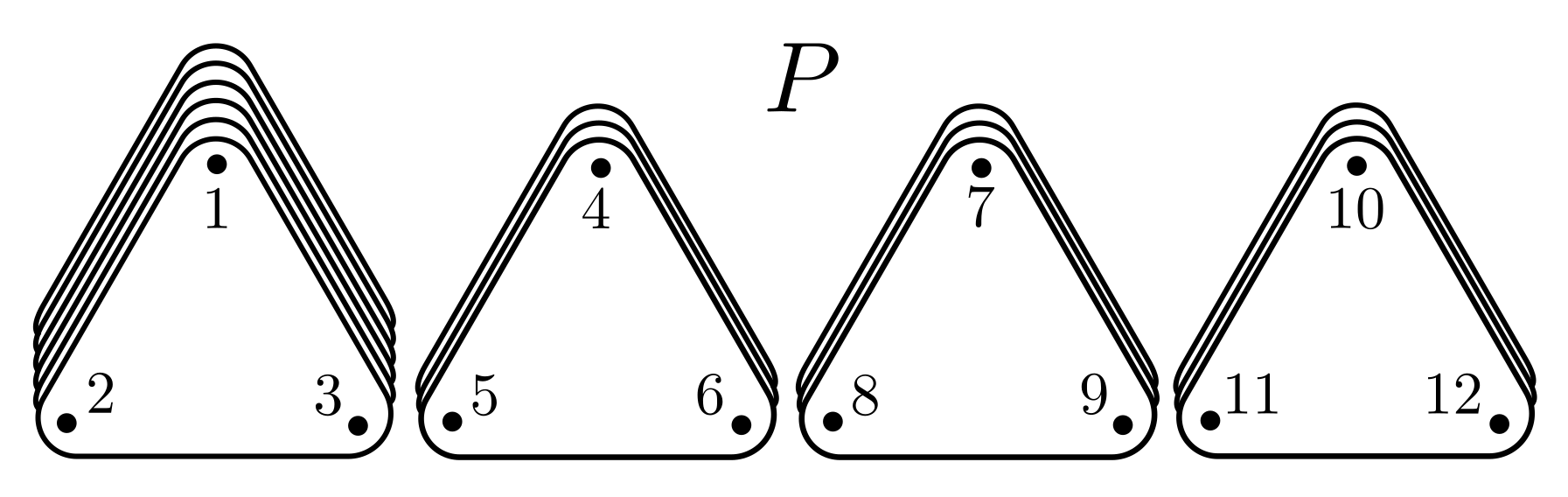} 
\caption{The Veblen hypergraph $P$, with edge-multiplicities $6,3,3,3$.}
\label{fig:infragraph_example_2_3}
\end{figure}

Then, we have $\widetilde{\mathcal{S}}\ev{P, X} =  \{ [S]_{\unsim}, [T]_{\unsim}\}$, where:
\begin{align*}
&  S = \{ { \color{red!80!black}{ \{e_i\}  } }_{i=1}^{6}, { \color{green!50!black}{ \{f_i\}  } }_{i=1}^{3} ,  { \color{blue!60!black}{ \{g_i\}  } }_{i=1}^{3}  , { \color{blue!60!black}{ \{g_i\}  } }_{i=4}^{6}  \} \text{ and } T = \{ { \color{blue!60!black}{ \{g_i\}  } }_{i=1}^{6}, { \color{green!50!black}{ \{f_i\}  } }_{i=1}^{3} ,  { \color{red!80!black}{ \{e_i\}  } }_{i=1}^{3}, { \color{red!80!black}{ \{e_i\}  } }_{i=4}^{6} \}, \text{ and,} \\
& \alpha_{  [S]_{\unsim} } = \alpha_{  [T]_{\unsim} } = 2!
\end{align*}

\end{example}

The following is a variation of Kocay's Lemma (cf. \cite{kocay_recons_spanning} and \cite[Lemma 1.3]{thatte}). 

\begin{lemma}[Kocay's Lemma]
\label{kocay}
If $H = G_1\sqcup \ldots \sqcup G_m  \in \mathcal{V}^m$ is a Veblen hypergraph (where $\{G_i\}_{i=1}^{m}$ are connected), then
\begin{align*}
\nu_H \ev{\# H\subseteq \mathcal{H}}  & =\prod_{ i = 1}^{m}\ev{\# G_i\subseteq \mathcal{H}} = \nu_H  \sum_{X\in \textup{Inf}\ev{\mathcal{H}}} \sum_{ \mathbf{S} \in \widetilde{\mathcal{S}} \ev{P, X}} \dfrac{1}{\alpha_\mathbf{S} }
\end{align*}
\end{lemma}
\begin{proof}
We define a mapping, 
\begin{align*}
\varepsilon: \prod_{i=1}^{m} \{ X\in \textup{Inf}^{1}\ev{\mathcal{H}} : X\cong G_i \} &\rightarrow \bigsqcup_{Y \in \textup{Inf}\ev{\mathcal{H}} } \widetilde{\mathcal{S}} \ev{P, Y} 
\\
\ev{ X_1,\ldots,X_m } & \mapsto [ \{ X_1, \ldots, X_m \} ]_{\unsim} \in \widetilde{\mathcal{S}} \ev{P, X_1\cup \ldots \cup X_m} 
\end{align*}
The mapping $\varepsilon$ is a surjection, where, for each $Y\in \textup{Inf}\ev{\mathcal{H}}$ and for each $\mathbf{S} = [ \{ X_1, \ldots, X_m \} ]_{\unsim} \in \widetilde{\mathcal{S}} \ev{P, Y}$, we have $|\varepsilon^{-1} \ev{ \mathbf{S} }| = \dfrac{ \nu_{ \mathbf{S} }  }{ \alpha_{  \mathbf{S} }}$. Therefore, 
$$ \prod_{i=1}^{m} \ev{\# G_i \subseteq \mathcal{H}}  = \sum_{Y \in \textup{Inf}\ev{\mathcal{H}} } \sum_{  \mathbf{S} \in \widetilde{\mathcal{S}} \ev{P, Y} } \dfrac{ \nu_{ \mathbf{S} }  }{ \alpha_{  \mathbf{S} }} = \nu_{ P } \sum_{Y \in \textup{Inf}\ev{\mathcal{H}} } \sum_{  \mathbf{S} \in \widetilde{\mathcal{S}} \ev{P, Y} } \dfrac{ 1  }{ \alpha_{  \mathbf{S} } } $$
since $\nu_{ \mathbf{S} } = \nu_P$ for all $\mathbf{S}\in \widetilde{\mathcal{S}} \ev{P, Y} $ and $Y\in \textup{Inf}\ev{\mathcal{H}}$. 
\end{proof}

For a Veblen hypergraph $X$ on $d\geq 0$ edges and a non-negative integer $n\geq 0$, let 
$$\Delta\ev{n,X} : = \sum_{ \mathbf{S} \in \widetilde{\mathcal{S}} \ev{X}}  \ev{ -\ev{k-1}^{n}}^{c\ev{\mathbf{S}}} \dfrac{C_\mathbf{S}}{\alpha_\mathbf{S} } $$
Note that,
$$\Delta\ev{0,X} = \sum_{ \mathbf{S} \in \widetilde{\mathcal{S}} \ev{X}}  \ev{ -1}^{c\ev{\mathbf{S}}} \dfrac{C_\mathbf{S}}{\alpha_\mathbf{S} } $$
which does not depend on $n$. 

Kocay's Lemma and this short notation allow an alternative expression of $\widetilde{\phi}_{\mathcal{H}} \ev{ t }$: 
\begin{proposition}
\label{Kocay_chi}
$$\widetilde{\phi}_{\mathcal{H}} \ev{ t } = \sum_{ X\in \textup{Inf}\ev{\mathcal{H}} } t^{-|E\ev{X}| } \Delta\ev{ | V\ev{ \mathcal{H} } | , X} $$ 
\end{proposition}

\begin{proof}
\begin{align*}
&\phi_{\mathcal{H}}\ev{ t } = \sum_{H \in \mathcal{V}^{\infty} } t^{-| E\ev{H} | } \ev{ -\ev{k-1}^{n}}^{c\ev{H}} C_H \ev{\# H\subseteq \mathcal{H}}   && \text{ by \Cref{thm:harary_sachs_hyper}}\\ 
& = \sum_{H \in \mathcal{V}^{\infty} } t^{-| E\ev{H} | } \ev{ -\ev{k-1}^{n}}^{c\ev{H}} C_H  \sum_{ \substack{X\in \textup{Inf}\ev{\mathcal{H}} \\ H \vdash X } } \sum_{ \mathbf{S} \in \widetilde{\mathcal{S}} \ev{H, X}} \dfrac{1}{\alpha_\mathbf{S} } && \text{ by Kocay's Lemma (\Cref{kocay})}\\
& = \sum_{X\in \textup{Inf}\ev{\mathcal{H}}} t^{-|E\ev{X}| } \sum_{ \substack{ H \in \mathcal{V}^{\infty} \\ H\vdash X } } \ev{ -\ev{k-1}^{n}}^{c\ev{H}} \sum_{ \mathbf{S} \in \widetilde{\mathcal{S}} \ev{H, X}} \dfrac{C_H}{\alpha_\mathbf{S} } && \\
& = \sum_{X\in \textup{Inf}\ev{\mathcal{H}}}  t^{-|E\ev{X}| } \sum_{ \mathbf{S} \in \widetilde{\mathcal{S}} \ev{X}}  \ev{ -\ev{k-1}^{n}}^{c\ev{\mathbf{S}}} \dfrac{C_\mathbf{S}}{\alpha_\mathbf{S} } 
\end{align*}
since $\widetilde{\mathcal{S}} \ev{X} = \bigsqcup_{H \in \mathcal{V}^{\infty}} \widetilde{\mathcal{S}} \ev{H, X} $ and $C_\mathbf{S} = C_H $ for each $\mathbf{S} \in \widetilde{\mathcal{S}} \ev{H, X}$

\end{proof}

\subsection{Heaps of Pieces on the Characteristic Polynomial of a Simple Hypergraph}
The weight of a heap is defined as the product of the weights of its pieces. To show that infragraphs can be taken as pieces, we need to show that the weight function we intend to use is multiplicative over disjoint components. First, we show that for a fixed simple host graph $\mathcal{H}$, the value of $\Delta\ev{n,X}$ is multiplicative, where $n\geq 0$ and $X\in \textup{Inf}\ev{\mathcal{H}}$. 

\begin{proposition}
\label{prop:cover_mult}
Let $X:= X_1 \sqcup X_2$ be an infragraph of a host hypergraph $\mathcal{H}$, where $X_1$ and $X_2$ are disjoint. Then, we have 
$$ \Delta\ev{n,X} = \Delta\ev{n,X_1} \cdot \Delta\ev{n,X_2}$$
where $n\geq 0$ is a non-negative integer. 
\end{proposition}
\begin{proof}
We have a bijection, 
\begin{align*}
\widetilde{\mathcal{S}} \ev{ X_1 } \times \widetilde{\mathcal{S}} \ev{ X_2 } &\rightarrow \widetilde{\mathcal{S}} \ev{ X_1\sqcup X_2 } \\
\ev{ \mathbf{S}_1 ,  \mathbf{S}_2}  & \mapsto \mathbf{S}_1\sqcup \mathbf{S}_2:=  [ \{ A_1,\ldots, A_m, B_1, \ldots, B_{r}\} ]_{\unsim}
\end{align*}
where $\mathbf{S}_1 = [\{ A_1,\ldots, A_m \} ]_{\unsim}$ and $\mathbf{S}_2 =[  \{ B_1,\ldots, B_{r} \} ]_{\unsim} $. Therefore,

\begin{align*}
&\Delta\ev{n,X} = \sum_{\mathbf{S} \in \widetilde{\mathcal{S}} \ev{ X_1 \sqcup X_2 } }  \ev{-\ev{k-1}^{n}}^{ c\ev{ \mathbf{S} } }  \dfrac{1}{\alpha_{\mathbf{S} }}  C_{ \mathbf{S} 	}  = \sum_{ \substack{ \mathbf{S}_1 \in \widetilde{\mathcal{S}} \ev{ X_1 }  \\ \mathbf{S}_2 \in \widetilde{\mathcal{S}} \ev{ X_2 } } }  \ev{-\ev{k-1}^{n}}^{c\ev{ \mathbf{S}_1 \sqcup \mathbf{S}_2  } }  \dfrac{1}{\alpha_{ \mathbf{S}_1 \sqcup \mathbf{S}_2} } C_{ \mathbf{S}_1 \sqcup \mathbf{S}_2 }  \\
& = \sum_{ \substack{ \mathbf{S}_1 \in \widetilde{\mathcal{S}} \ev{ X_1 }  \\ \mathbf{S}_2 \in \widetilde{\mathcal{S}} \ev{ X_2 } } }  \ev{-\ev{k-1}^{n}}^{c\ev{ \mathbf{S}_1  } + c\ev{ \mathbf{S}_2 }} \dfrac{1}{\alpha_{ \mathbf{S}_1 }}  \dfrac{1}{\alpha_{ \mathbf{S}_2 }} C_{ \mathbf{S}_1 }C_{ \mathbf{S}_2 } \\
& = \ew{ \sum_{ \mathbf{S}_1 \in \widetilde{\mathcal{S}} \ev{ X_1 } } \ev{-\ev{k-1}^{n}}^{c\ev{ \mathbf{S}_1 } } \dfrac{1}{ \alpha_{  \mathbf{S}_1  } } C_{ \mathbf{S}_1 } } \cdot \ew{ \sum_{ \mathbf{S}_2 \in \widetilde{\mathcal{S}} \ev{ X_2 } } \ev{-\ev{k-1}^{n}}^{c\ev{ \mathbf{S}_2 } } \dfrac{1}{ \alpha_{  \mathbf{S}_2  } } C_{ \mathbf{S}_2 } }  = \Delta\ev{n,X_1} \cdot \Delta\ev{n,X_2}
\end{align*}
where we used the multiplicativity of the associated coefficient. 
\end{proof}

We are ready to apply the Heaps of Pieces framework on hypergraphs. Let $\mathcal{H}$ be a simple hypergraph of rank $k\geq 3$, with $|V\ev{ \mathcal{H} }| = n$ vertices. 

We take connected infragraphs as the set of pieces:
$$
\mathcal{B}\ev{ \mathcal{H} } := \textup{Inf}^{1} \ev{\mathcal{H} }
$$
Two pieces $X_1$ and $X_2$ are \textit{concurrent} (denoted $X_1 \mathcal{R} X_2$) if and only if they share a vertex. 

\begin{theorem}
\label{thm:heaps_hyper}
Given a simple hypergraph $\mathcal{H}$ with $n$ vertices, then $\widetilde{\phi}_\mathcal{H} \ev{ t }$ can be expressed as a sum over trivial heaps:
$$ \widetilde{\phi}_\mathcal{H} \ev{ t } = \sum_{T\in \mathbbm{t}\ev{\mathcal{B}\ev{\mathcal{H}},\mathcal{R}}} \ev{-1}^{|T|}  w_n\ev{T} $$ 
where $w_n\ev{X} = -\Delta\ev{n,X} t^{- |E\ev{X}|  } $ is the weight function, for a given piece $X \in \mathcal{B}\ev{ \mathcal{H} } $.
\end{theorem}
\begin{proof}
The weight of a piece $X \in \mathcal{B}\ev{ \mathcal{H} }$ is given by,
$$ w_n\ev{X} = -\Delta\ev{n,X} t^{- |E\ev{X}|  } = - t^{- |E\ev{X}|  }  \sum_{ \mathbf{S} \in \widetilde{\mathcal{S}} \ev{ X }} \ev{ -\ev{k-1}^{n} }^{ c\ev{ \mathbf{S} } } \dfrac{1}{\alpha_{\mathbf{S}}} C_{\mathbf{S} } $$
We define the weight of a trivial heap with $r\geq 1$ pieces, $Y = X_1\sqcup \ldots \sqcup X_r$ (a disjoint union of $r$ connected Veblen hypergraphs), as the product of the weights of its pieces. Hence,
\begin{align*}
& w_n\ev{Y} := \prod_{i=1}^{r} w_n\ev{X_i} = \prod_{i=1}^{r} -\Delta\ev{n,X_i} t^{- | E\ev{X_i} |  }  = \ev{-1}^{c\ev{Y}} t^{- | E\ev{Y} | } \prod_{i=1}^{r} \Delta\ev{n,X_i}  && \\
&= \ev{-1}^{c\ev{Y}} t^{- | E\ev{Y} |  } \Delta\ev{n,X_1 \sqcup \ldots \sqcup X_r}  = \ev{-1}^{c\ev{Y}} t^{- | E\ev{Y} |  } \Delta\ev{n,Y} && \text{by \Cref{prop:cover_mult}}  
\end{align*}
In particular, we have $ t^{- | E\ev{Y} |  } \Delta\ev{n,Y}  = \ev{-1}^{c\ev{Y}}  w_n\ev{Y}$. Therefore,
\begin{align*}
&\widetilde{\phi}_\mathcal{H} \ev{ t } = \sum_{Y\in \textup{Inf}\ev{\mathcal{H}}}   t^{- | E\ev{Y} |  } \Delta\ev{n,Y} = \sum_{T\in \mathbbm{t}\ev{\mathcal{B}\ev{ \mathcal{H} }}} \ev{-1}^{|T|}  w_n\ev{T}  & \text{ by \Cref{Kocay_chi}} 
\end{align*}
\end{proof}
As we will see in the next section, there are only finitely many non-zero contributions to this sum, since a connected infragraph with more edges than the degree of the characteristic polynomial has zero weight (cf. \Cref{rmk:finitely_many_pieces}.)

\subsection{Heaps of Pieces on the Characteristic Polynomial with edge-variables}

\subsection*{Characteristic Polynomial with Edge-Variables}
We consider a variation of \Cref{thm:harary_sachs_hyper}, using indeterminate edge-variables. First, we need some preliminaries:

\begin{definition}[Adjacency Hypermatrix with edge-variables]
\label{def:adjacency_hypermatrix_edge_vars}
Let $\mathcal{H}$ be a simple hypergraph on $n$ vertices, with rank $k\geq 2$. Let $e_{\{j_1,\ldots,j_k\}}$ be formal variables for each $k$-set $\{j_1,\ldots,j_k\}$, where $j_1,\ldots,j_k\in [n]$. The (normalized) \textit{adjacency hypermatrix} $\widehat{\mathbb{A}}_\mathcal{H} = \ev{ a_{j_1,\ldots,j_k} }$ of $\mathcal{H}$ \textit{with edge-variables} is the $n$-dimensional symmetric $k$-hypermatrix with entries
$$ a_{j_1,\ldots,j_k} := \begin{cases}  \ffrac{ e_{\{j_1,\ldots,j_k\}} }{(k-1)!} & \text{ if } \{j_1,\ldots,j_k\} \in E\ev{ \mathcal{H} } \\ 0 & \text{otherwise}
\end{cases}$$
If the edges are enumerated as $E\ev{\mathcal{H}} = \{e_1,\ldots,e_r\}$, then we define the entries as 
$$ a_{j_1,\ldots,j_k} := \begin{cases}  \ffrac{ e_i }{(k-1)!} & \text{ if } \{j_1,\ldots,j_k\} = e_i \in E\ev{ \mathcal{H} } \text{ for some $i$} \\ 0 & \text{otherwise}
\end{cases}$$
\end{definition}

\begin{definition}[Characteristic Polynomial with edge-variables]
Let $\mathcal{H}$ be a simple hypergraph and let $t$ be an indeterminate. The \textit{characteristic polynomial} $\widehat{\phi}_{\mathcal{H}}\ev{t}$ \textit{with edge-variables} is the resultant
$$ \textup{res}  \nabla \ev{ t\cdot \widehat{F}_\mathcal{I} - \widehat{F}_\mathcal{H} }  $$
where $\widehat{F}_\mathcal{H}$ is the Lagrangian polynomial (Definition \ref{def:char_poly}) of the adjacency hypermatrix $\widehat{\mathbb{A}}_\mathcal{H}$ (Definition \ref{def:adjacency_hypermatrix_edge_vars}).
\end{definition}

Let $E\ev{\mathcal{H}} = \{e_1,\ldots, e_r\}$ be the edge-set of a simple hypergraph $\mathcal{H}$. For ease of notation, we use
$$ \mathbf{e}^Y = e_1^{m_Y\ev{e_1}} \cdots e_r^{m_Y\ev{e_r}} $$
for each infragraph $Y \in \textup{Inf}\ev{\mathcal{H}}$. 

\begin{theorem}[Harary-Sachs Theorem with edge-variables]
\label{chi_edge_var}
Given a simple hypergraph $\mathcal{H}$, with edge-set $E\ev{\mathcal{H}} = \{e_1,\ldots, e_r\}$ and order $|V\ev{ \mathcal{H}}| = n$, let $N = n\ev{k-1}^{n-1}$ be the degree of the characteristic polynomial $\widehat{\phi}_{\mathcal{H}}\ev{t}$. Then, the codegree $d$ coefficient of $\widehat{\phi}_{\mathcal{H}}\ev{t}$ is:
\begin{align*}
& [t^{N-d}] \widehat{\phi}_{\mathcal{H}}\ev{t} = \sum_{X\in \textup{Inf}_d\ev{\mathcal{H}} } \mathbf{e}^X \ew{ \sum_{ \mathbf{S} \in \widetilde{\mathcal{S}} \ev{ X }} \ev{ -\ev{k-1}^{n} }^{ c\ev{ \mathbf{S} } } \dfrac{1}{\alpha_{\mathbf{S}}} C_{\mathbf{S} } } = \sum_{X\in \textup{Inf}_d\ev{\mathcal{H}} }  \mathbf{e}^X \Delta\ev{n,X}
\end{align*}
In other words,
$$ \widehat{\phi}_{\mathcal{H}}\ev{t} := t^{-N} \cdot  \widehat{\phi}_{\mathcal{H}}\ev{t} = \sum_{X\in \textup{Inf}\ev{\mathcal{H}} }  \mathbf{e}^X t^{- |E\ev{X}| }  \Delta\ev{n,X} $$

\end{theorem}

\begin{proof}
Let $\widehat{\mathbb{A}}_{\mathcal{H}}  = \ev{ a_{j_1,\ldots,j_k} }$ be the edge-variable adjacency hypermatrix of $\mathcal{H}$. Let $f = \nabla \widehat{F}_\mathcal{H} = [  f_1 \cdots  f_n ]^{T}$ be the vector with entries $f_i = \sum_{j_2, j_3, \ldots, j_k=1}^n a_{i j_2 j_3 \ldots j_k} x_{j_2} \ldots x_{j_k}$ for $i=1,\ldots,n$.

Let $A$ be an auxiliary matrix with distinct variables $A_{ij}$ as entries. For each $i=1,\ldots,n$, we define a differential operator $\partial {f}_i$ by replacing $x_j$ in $f_i$ with $\partial / \partial A_{ij} $, for $j=1,\ldots,n$ and thereby obtain $\partial {f} =  [\partial f_1 \cdots \partial f_n ]^{T}$ (cf. the example below).

By \cite{morozov} and \cite[Lemma 10]{cc1}, we have a formula for the generalized trace: 
$$ \operatorname{Tr}_d(\mathcal{H})=(k-1)^{n-1} \sum_{d_1+\cdots+d_n=d}\left(\prod_{i=1}^n \frac{\partial {f}_i^{d_i}}{\left(d_i(k-1)\right) !} \operatorname{tr}\left(A^{d(k-1)}\right)\right) $$

The rest of the proof is almost identical to that of the characteristic polynomial without edge-variables: For any differential operator $\hat{g} $, we have $\hat{g} \operatorname{tr} A^d(k-1) \neq 0 $  if and only if $\hat{g} = \partial {f}_{D_R}$, the operator of a rooted digraph $D_R$, for some rooting $R\in \mathfrak{R}\ev{\mathcal{H}}$, of a unique connected infragraph $X$ of $\mathcal{H}$, in which case, $\hat{g} \operatorname{tr} A^d(k-1) = |\mathfrak{C}\ev{D_R}| \cdot |E\ev{D_R}|$. So, the equation turns into: 
$$ \operatorname{Tr}_d(\mathcal{H})=d(k-1)^{n} \sum_{X\in \textup{Inf}^{1}_d\ev{\mathcal{H}}}   \mathbf{e}^X  \sum_{R \in \mathfrak{R}\ev{X}} \ffrac{ |\mathfrak{C}\ev{D_R}|}{\prod_{v\in D_R } \deg^{-}\ev{v}! } = d(k-1)^{n} \sum_{X\in \textup{Inf}^{1}_d\ev{\mathcal{H}}}   \mathbf{e}^X  C_X $$
We exponentiate and obtain,
\begin{align*}
& [t^{N-d}]\ev{\phi_\mathcal{H}\ev{t}} = \sum_{m=1}^d \sum_{d_1+\cdots+d_m=d} \frac{1}{m !} \prod_{i=1}^m \frac{-\operatorname{Tr}_{d_i}(\mathcal{H})}{d_i}  \\
& = \sum_{X\in \textup{Inf}_d\ev{\mathcal{H}} }  \mathbf{e}^X  \sum_{ \mathbf{S} \in \widetilde{\mathcal{S}} \ev{X}}  \ev{ -\ev{k-1}^{n}}^{c\ev{\mathbf{S}}} \dfrac{C_\mathbf{S}}{\alpha_\mathbf{S} }  && \text{ by \cite[Theorem 14]{cc1} and Kocay's Lemma (\Cref{kocay}) } \\
&= \sum_{X\in \textup{Inf}_d\ev{\mathcal{H}} }  \mathbf{e}^X  \Delta\ev{n,X} 
\end{align*}
\end{proof}
\begin{example}
As an example, consider:
$\mathcal{H} = \ev{ V\ev{\mathcal{H}} , E\ev{ \mathcal{H} } } : = \ev{ [4] , \{e_1 = (123), e_2 = (124)\}}$ 
In this example, we have 
$$\widehat{F}_\mathcal{H} = 3 e_1 x_1 x_2 x_3 + 3 e_2 x_1 x_2 x_4 \qquad  f = \nabla \widehat{F}_\mathcal{H} = \begin{bmatrix} 3 e_1 x_2 x_3 + 3 e_2 x_2 x_4 \\  3e_1 x_1 x_3 + 3e_2 x_1 x_4 \\  3e_1 x_1 x_2 \\  3e_2 x_1 x_2 \end{bmatrix} $$
Introduce the notation, $\partial {S}_{i}\ev{\{j_1,\ldots,j_k\}} = \ffrac{\partial}{\partial A_{ij_1}} \cdots \ffrac{\partial}{\partial A_{ij_k}}$. Then, the vector $\partial {f}$ in our example is the following:
$$ \partial {f} = \begin{bmatrix} e_1 \ffrac{\partial}{\partial A_{12}} \ffrac{\partial}{\partial A_{13}} + e_2 \ffrac{\partial}{\partial A_{12}} \ffrac{\partial}{\partial A_{14}} \\  e_1 \ffrac{\partial}{\partial A_{21}} \ffrac{\partial}{\partial A_{23}} + e_2 \ffrac{\partial}{\partial A_{21}} \ffrac{\partial}{\partial A_{24}} \\  e_1 \ffrac{\partial}{\partial A_{31}} \ffrac{\partial}{\partial A_{32}} \\  e_2 \ffrac{\partial}{\partial A_{41}} \ffrac{\partial}{\partial A_{42}} \end{bmatrix} = \begin{bmatrix} e_1 \partial {S}_{1}\ev{e_1}  + e_2 \partial {S}_{1}\ev{e_2} \\ e_1 \partial {S}_{2}\ev{e_1}  + e_2 \partial {S}_{2}\ev{e_2} \\   e_1 \partial {S}_{3}\ev{e_1}  \\   e_2 \partial {S}_{4}\ev{e_2}  \end{bmatrix}  $$

\end{example}

\subsection*{Spanning Subgraphs and the Characteristic Polynomial with Edge-Variables}
\begin{proposition}
Let a simple hypergraph $$\mathcal{H} = \ev{V\ev{\mathcal{H}}, E\ev{\mathcal{H}} } = \ev{ [n],  \{e_1,\ldots,e_t\} } $$
and a spanning subgraph $\mathcal{G}$ with the same vertex set, 
$$\mathcal{G} = \ev{V\ev{\mathcal{G}}, E\ev{\mathcal{G}} } = \ev{ [n],  \{e_1,\ldots,e_r\} }  \text{ where } r\leq t$$
be given. Then:
$$\widehat{\phi}_\mathcal{G}\ev{t}\ev{ e_1,\ldots,e_r } = \widehat{\phi}_\mathcal{H}\ev{t}\ev{ e_1,\ldots,e_r, 0, \ldots, 0 }$$
\end{proposition}
\begin{proof}
By the Harary Sachs Theorem with edge-variables, we have
$$\widehat{\phi}_\mathcal{H} \ev{t} \ev{ e_1, \ldots, e_t } = \sum_{X\in \textup{Inf}\ev{\mathcal{H}} } t^{- |E\ev{X}| } \Delta\ev{|V\ev{\mathcal{H}}|,X}  e_1^{m_X\ev{e_1}} \cdots e_t^{m_X\ev{e_t}}  $$
Therefore,
\begin{align*}
& \widehat{\phi}_\mathcal{H} \ev{t} \ev{ e_1,\ldots,e_r, 0, \ldots, 0 } =  \sum_{X\in \textup{Inf}\ev{\mathcal{H}} }  t^{- |E\ev{X}| } \Delta\ev{|V\ev{\mathcal{H}}|,X}  e_1^{m_X\ev{e_1}} \cdots e_r^{m_X\ev{e_r}} 0^{ m_X\ev{e_{r+1}} } \cdots 0^{m_X\ev{e_t}} \\
& =  \sum_{\substack{X\in \textup{Inf}\ev{\mathcal{H}} \\ m_X\ev{e_i} = 0,\ \forall i\geq r+1 } }   t^{- |E\ev{X}| } \Delta\ev{|V\ev{\mathcal{H}}|,X}  e_1^{m_X\ev{e_1}} \cdots e_r^{m_X\ev{e_r}} e_{r+1}^{ 0 } \cdots e_{t}^{ 0 }  \\
& = \sum_{X\in \textup{Inf}\ev{\mathcal{G}} } t^{- |E\ev{X}| } \Delta\ev{|V\ev{\mathcal{H}}|,X} e_1^{m_X\ev{e_1}} \cdots e_r^{m_X\ev{e_r}}  =   \widehat{\phi}_\mathcal{G}\ev{t}\ev{ e_1,\ldots,e_r } 
\end{align*}
\end{proof}

\begin{remark}
\label{rmk:finitely_many_pieces}
As the characteristic polynomial with edge-variables is a polynomial in the variables $t,e_1,\ldots,e_m$, it follows from \Cref{chi_edge_var} that for each infragraph $X$ with $|E\ev{ X }|> \deg\ev{\widehat{\phi}_{\mathcal{H}}\ev{t}} =  n\ev{k-1}^{n-1}$, we have $\Delta\ev{|V\ev{ \mathcal{H} }|,X} = 0$. 
\end{remark}

Due to this observation, we use the finite set $\widehat{\mathcal{B}}\ev{ \mathcal{H} } = \{ X\in \textup{Inf}^{1}\ev{ \mathcal{H} }: |E\ev{X}| \leq n\ev{k-1}^{n-1} \}$ as the set of pieces (with the concurrence relation $\mathcal{R}$ of sharing a vertex), to express the characteristic polynomial with edge-variables as a sum over trivial heaps. The proof is almost identical to that of \Cref{thm:heaps_hyper}, and omitted here.
\begin{theorem}
Given a simple hypergraph $\mathcal{H}$ with $n$ vertices, then we can express $\widehat{\phi}_{\mathcal{H}} \ev{ t }$ as a sum over trivial heaps:
$$\widehat{\phi}_{\mathcal{H}} \ev{ t } = \sum_{T\in \mathbbm{t}\ev{\widehat{\mathcal{B}}\ev{ \mathcal{H} }}} \ev{-1}^{ |T| }  w_n\ev{T}  $$ 
where $w_n\ev{X} = -\Delta\ev{n,X} t^{- |E\ev{X}|  } \mathbf{e}^X$ is the weight function, for a given piece $X$.
\end{theorem}

\subsection{Heaps of Pieces on the Root of the Characteristic Polynomial}
Let $|V\ev{ \mathcal{H} }|=n$ be the order of a simple hypergraph $\mathcal{H}$. We define,
$$ \overline{\phi}_{\mathcal{H}}\ev{t}:=\widetilde{\phi}_\mathcal{H}\ev{t}^{1/\ev{k-1}^{n}} = \ew{ \dfrac{ \phi_\mathcal{H}\ev{t} }{ t^{n\ev{ k-1 }^{n-1} }} }^{1/\ev{k-1}^{n}} $$

\begin{proposition}
\label{prop:chi_expression}
The coefficient of the term $t^{-d}$ in $ \overline{\phi}_{\mathcal{H}}\ev{t} $ is given by 
$$ \sum_{H \in \mathcal{V}_d^{\infty} } \ev{-1}^{c\ev{H}} C_H \ev{ \# H\subseteq \mathcal{H}} $$
which does not depend on the order $|V\ev{\mathcal{H}}|$ of $\mathcal{H}$.
\end{proposition}
\begin{proof}
We start with the following equation from \cite[Lemma 13, p.~11]{cc1}:
$$\log\ev{ \widetilde{\phi}_\mathcal{H}\ev{t} } = -\ev{k-1}^{n} \sum_{d\geq 0 }\sum_{X\in \textup{Inf}^{1}_d\ev{\mathcal{H}}} t^{-d} C_X$$
Dividing by $\ev{k-1}^{n}$,
$$ \log\ev{  \overline{\phi}\ev{t} }  =  \log\ew{ \ev{ \widetilde{\phi}_\mathcal{H}\ev{t} }^{1/\ev{k-1}^{n}} }  =  \dfrac{1}{\ev{k-1}^{n}}\log\ev{ \widetilde{\phi}_\mathcal{H}\ev{t} }  = -\sum_{d\geq 0 }\sum_{X\in \textup{Inf}_d\ev{H}} t^{-d} C_X$$
Taking exponentials,
$$  \overline{\phi}\ev{t} = \exp\ew{-\sum_{d\geq 0 }\sum_{X\in \textup{Inf}_d\ev{H}} t^{-d} C_X }$$
By Theorem \cite[Theorem 14]{cc1}, this implies,
$$  \overline{\phi}_{\mathcal{H}}\ev{t} =  \sum_{d\geq 0 } t^{-d} \sum_{H \in \mathcal{V}_d^{\infty} } \ev{-1}^{c\ev{H}} C_H \ev{ \# H\subseteq \mathcal{H}} = \sum_{H \in \mathcal{V}^{\infty} } t^{- |E\ev{H}| } \ev{-1}^{c\ev{H}} C_H \ev{ \# H\subseteq \mathcal{H}} $$
\end{proof}

We apply Kocay's Lemma (\Cref{kocay}) and get an alternative formula for the coefficient of $t^{-d}$ in $\overline{\phi}_\mathcal{\mathcal{H}}\ev{t}$: 
\begin{corollary}
The coefficient of the term $t^{-d}$ in $ \overline{\phi}\ev{t}$ is given by 
$$ \sum_{X\in \mathrm{Inf}_d\ev{\mathcal{H}}} \sum_{ \mathbf{S} \in \widetilde{\mathcal{S}} \ev{ X }} \ev{ -1 }^{ c\ev{ \mathbf{S} } } \dfrac{1}{\alpha_{\mathbf{S}}} C_{\mathbf{S} } = \sum_{X\in \textup{Inf}_d\ev{H}}  \Delta\ev{0,X}$$
Therefore, $\overline{\phi}_\mathcal{\mathcal{H}}\ev{t} $ is expressed as,
$$  \overline{\phi}_\mathcal{\mathcal{H}}\ev{t} = \sum_{X\in \textup{Inf}\ev{\mathcal{H}}}  t^{-|E\ev{X}| } \Delta\ev{0,X}  $$
\end{corollary}

The Heaps of Pieces framework applies to the function $\overline{\phi}_\mathcal{\mathcal{H}}\ev{t}$. Again, the proof is almost identical to that of \Cref{thm:heaps_hyper}, and omitted here. We take the set of pieces as the infinite set of connected infragraphs: $ \overline{\mathcal{B}}\ev{ \mathcal{H} }:= \textup{Inf}^{1}\ev{\mathcal{H}} $. As before, two pieces $X_1$ and $X_2$ are \textit{concurrent} (denoted $X_1 \mathcal{R} X_2$) if and only if they share a vertex.

\begin{theorem}
\label{root_heaps_pieces}
Given a simple hypergraph $\mathcal{H}$, the power series $\overline{\phi}_\mathcal{H}\ev{t}$ can be expressed as a formal sum over trivial heaps:
$$\overline{\phi}_\mathcal{H}\ev{t} =  \sum_{T\in \mathbbm{t}\ev{ \overline{\mathcal{B}}\ev{ \mathcal{H} }} } \ev{-1}^{ |T| } w_0\ev{T} $$ 
where $w_0\ev{X} = -\Delta\ev{0,X} t^{- |E\ev{X}|  }$ is the weight function, for a given piece $X \in \overline{\mathcal{B}}\ev{ \mathcal{H} }$.
\end{theorem}
%
%
The weight function $w_0$ for $\overline{\phi}_\mathcal{H} \ev{ t }$ is uniform among different hypergraphs $\mathcal{H}$, as it does not depend on the order $n= |V\ev{\mathcal{H}}|$. However, there are infinitely many pieces with non-zero weight. 

\section{Calculation of the Weight of an Infragraph in a Special Case}
\label{sec:special_cases}
In this section, we include a specific result on the calculation of the weight of a Veblen hypergraph $X$, provided it possesses a certain structure described below. First, we have a result about the associated coefficient. 
\begin{lemma}
\label{lem:non_multiplicative}
Let $X = X_1 \cup X_2$ be a Veblen hypergraph, where $X_1$ and $X_2$ are connected Veblen hypergraphs with $V\ev{X_1} \cap V\ev{X_2} = \{u\}$.
Then, we have,
$$ C_X = \ev{k-1} \dfrac{ \mathfrak{s}_u\ev{X_1} \mathfrak{s}_u\ev{X_2}  }{ \mathfrak{s}_u\ev{X} }\binom{   \mathfrak{s}_u\ev{X}  }{  \mathfrak{s}_u\ev{X_1}  , \mathfrak{s}_u\ev{X_2} }  C_{X_1}  C_{X_2}  $$
where $\mathfrak{s}_u\ev{X} = \frac{ \deg_{X}\ev{u} }{ k }$ is the number of times $u$ is chosen as the root, in a rooting $R \in \mathfrak{R}\ev{X}$ of $X$. 
\end{lemma}	
See \hyperlink{proof:nonmultiplicativepf}{proof} on p.~\pageref*{proof:non_multiplicative_pf_page}. 

If the intersection of two Veblen hypergraps $X_1$ and $X_2$ of rank $k\geq 2$ is a singleton $\{u\}$ with $\deg_{X_1}\ev{u} = \deg_{X_2}\ev{u} = k$, then we can express the total weight of $X_1\cup X_2$ in terms of subweights of $X_1$ and $X_2$:

\begin{proposition}
\label{prop:intersecting_infras}
Let $\mathcal{H}$ be a simple hypergraph. Let $X:= X_1 \cup X_2$ be an infragraph of $\mathcal{H}$ such that $V\ev{ X_1 }\cap V\ev{ X_2 } = \{ u \}$. If $\deg_{X_1}\ev{ u } = \deg_{X_2}\ev{ u } = k$, then, for each $n\geq 0$, we have 
$$ w_n\ev{ X } =  \ew{ \ev{ k-1}^{ 1 - n }  - 1  } \cdot w_n\ev{ X_1 }\cdot w_n\ev{ X_2 } $$
In particular,
$$ w_0\ev{ X } = w_0\ev{ X_1 }\cdot w_0\ev{ X_2 } $$
\end{proposition}
See \hyperlink{proof:intersectinginfraspf}{proof} on p.~\pageref*{proof:intersecting_infras_pf_page}. 

\section{Future Directions}
\begin{enumerate}
\item We have shown that connected infragraphs of a simple hypergraph of rank $k\geq 3$ can be taken as its pieces. Is there a smaller/alternative set of pieces, from which we can obtain the weight of each infragraph as a product of the weights of some pieces from that set? 
\item The rational function $\frac{\phi_{G-u} \ev{ t } }{\phi_{G} \ev{ t }}$ can be used to show the Eigenvalue Interlacing Theorem (\cite[Theorem 5.3, p.~71]{godsil}). Is there a way to generalize this method to hypergraphs?
\item In \cite[p.~71]{godsil}, the following identity for the characteristic polynomial of a graph $G$ is given:
$$
\phi_{G}'\ev{t}^{2} -  \phi_{G}''\ev{t}  \phi_{G}\ev{t} = \sum_{i,j} \phi^{i,j}_{G}\ev{t}^{2}
$$ 
where $\phi^{i,j}_{G}\ev{t}$ is the $ij$th entry of the adjugate matrix of $t\cdot I - \mathbb{A}_G$. Is there a version of this identity for hypergraphs? 
\end{enumerate}

\section*{Acknowledgements}

This project started at and was made possible by \textit{Spectral Graph and Hypergraph Theory: Connections \& Applications}, December 6-10, 2021, a workshop at the American Institute of Mathematics with support from the US National Science Foundation. 


\bibliographystyle{plain}
\bibliography{bibliography}

\begin{thebibliography}{10}

\bibitem{best}
T.~{Aardenne-Ehrenfest, van} and N.G. {Bruijn, de}.
\newblock Circuits and trees in oriented linear graphs.
\newblock {\em Simon Stevin : Wis- en Natuurkundig Tijdschrift}, 28:203--217,
  1951.

\bibitem{samy_juge}
Samy Abbes and Vincent Jugé.
\newblock Uniform generation of large traces, 10 2024.

\bibitem{abdesselam_brydges}
A.~Abdesselam and D.~C. Brydges.
\newblock Cramer’s rule and loop ensembles.
\newblock \url{ https://personal.math.ubc.ca/~db5d/Research/Preprints/LERW.pdf
  }.
\newblock {A}ccessed 2024-9-24.

\bibitem{bondy}
J.A. Bondy and U.S.R Murty.
\newblock {\em Graph Theory}.
\newblock Springer Publishing Company, Incorporated, 1st edition, 2008.

\bibitem{melou}
Mireille Bousquet-Mélou.
\newblock Rational and algebraic series in combinatorial enumeration.
\newblock {\em International Congress of Mathematicians, ICM 2006}, 3, 06 2008.

\bibitem{melou_polyomino}
Mireille Bousquet-Mélou and Xavier~Gérard Viennot.
\newblock Empilements de segments et q-énumération de polyominos convexes
  dirigés.
\newblock {\em Journal of Combinatorial Theory, Series A}, 60(2):196--224,
  1992.

\bibitem{cc1}
Gregory~J. Clark and Joshua~N. Cooper.
\newblock A {H}arary-{S}achs theorem for hypergraphs.
\newblock {\em Journal of Combinatorial Theory, Series B}, 149:1--15, 2021.

\bibitem{cc2}
Gregory~J. Clark and Joshua~N. Cooper.
\newblock Applications of the harary-sachs theorem for hypergraphs.
\newblock {\em Linear Algebra and its Applications}, 649:354--374, 2022.

\bibitem{usingalg}
David~A. Cox, John Little, and Donal O'Shea.
\newblock {\em Using Algebraic Geometry}.
\newblock Springer, 2005.

\bibitem{flajolet}
Philippe Flajolet and Robert Sedgewick.
\newblock {\em Analytic Combinatorics}.
\newblock Cambridge University Press, USA, 1 edition, 2009.

\bibitem{Godsil_1992}
C.~D. Godsil.
\newblock Walk generating functions, christoffel-darboux identities and the
  adjacency matrix of a graph.
\newblock {\em Combinatorics, Probability and Computing}, 1(1):13–25, 1992.

\bibitem{godsil}
Chris Godsil.
\newblock {\em Algebraic Combinatorics}.
\newblock Routledge, 1st edition, 1993.

\bibitem{goodman}
Jacob~E. Goodman and Joseph O'Rourke, editors.
\newblock {\em Handbook of discrete and computational geometry}.
\newblock CRC Press, Inc., USA, 1997.

\bibitem{viennot_beauchamps}
D~Gouyou-Beauchamps and G~Viennot.
\newblock Equivalence of the two-dimensional directed animal problem to a
  one-dimensional path problem.
\newblock {\em Adv. Appl. Math.}, 9(3):334–357, September 1988.

\bibitem{harary}
Frank Harary.
\newblock The determinant of the adjacency matrix of a graph.
\newblock {\em SIAM Review}, 4(3):202--210, 1962.

\bibitem{Jacobi1834}
C.G.J. Jacobi.
\newblock De binis quibuslibet functionibus homogeneis secundi ordinis per
  substitutiones lineares in alias binas tranformandis, quae solis quadratis
  variabilium constant; una cum variis theorematis de tranformatione
  etdeterminatione integralium multiplicium.
\newblock {\em Journal für die reine und angewandte Mathematik},
  1834(12):1--69, 1834.

\bibitem{kocay_recons_spanning}
William~L Kocay.
\newblock On reconstructing spanning subgraphs.
\newblock {\em Ars Combin}, 11:301--313, 1981.

\bibitem{krattenthaler}
C.~Krattenthaler.
\newblock The theory of heaps and the {C}artier–{F}oata monoid.
\newblock In {\em Problèmes combinatoires de commutation et réarrangements}.
  Springer, 2006.
\newblock Appendix of the electronic edition.

\bibitem{morozov}
A.~Morozov and Shamil Shakirov.
\newblock Analogue of the identity log det = trace log for resultants.
\newblock {\em arXiv: Mathematical Physics}, 2008.

\bibitem{oeis}
{OEIS Foundation Inc.}
\newblock The {O}n-{L}ine {E}ncyclopedia of {I}nteger {S}equences, 2024.
\newblock Published electronically at \url{http://oeis.org}.

\bibitem{thatte}
Igor~C. Oliveira and Bhalchandra~D. Thatte.
\newblock An algebraic formulation of the graph reconstruction conjecture.
\newblock {\em Journal of Graph Theory}, 81(4):351--363, 2016.

\bibitem{qi}
Liqun Qi.
\newblock Eigenvalues of a real supersymmetric tensor.
\newblock {\em Journal of Symbolic Computation}, 40(6):1302--1324, 2005.

\bibitem{sachs}
Horst Sachs.
\newblock \"{U}ber teiler, faktoren und charakteristiche polynome von graphen.
\newblock {\em Wiss. Z. TH Ilmenau}, 12:7--12, 1966.

\bibitem{sagemath}
{The Sage Developers}.
\newblock {S}agemath, the {S}age {M}athematics {S}oftware {S}ystem ({V}ersion
  10.1).
\newblock \url{ https://www.sagemath.org }, 2023.
\newblock {A}ccessed 2024-10-29.

\bibitem{veblen}
Oswald Veblen.
\newblock An application of modular equations in analysis situs.
\newblock {\em Annals of Mathematics}, 14(1/4):86--94, 1912.

\bibitem{viennot_talk}
G{\'e}rard~Xavier Viennot.
\newblock Courses {T}alca, {C}hile 2013/2014.
\newblock \url{ https://cours.xavierviennot.org/Talca_2013_14.html }.
\newblock {A}ccessed 2024-9-24.

\bibitem{viennot}
G{\'e}rard~Xavier Viennot.
\newblock Heaps of pieces, i : Basic definitions and combinatorial lemmas.
\newblock In Gilbert Labelle and Pierre Leroux, editors, {\em Combinatoire
  {\'e}num{\'e}rative}, pages 321--350, Berlin, Heidelberg, 1986. Springer
  Berlin Heidelberg.

\bibitem{viennot_french_paper}
G{\'e}rard~Xavier Viennot.
\newblock The art of bijective combinatorics.
\newblock \url{
  https://cours.xavierviennot.org/Talca_2013_14_files/LMA3.Viennot.v4.pdf },
  2017.

\end{thebibliography}

\section{Appendix}
\vspace{-0.5cm}
\hypertarget{proof:downsetrecomposed}{\phantom{a}}
\begin{proof}[Proof of \Cref{lem:push_down}:]
\label{proof:downset_recomposed}
First, we show that $\mathcal{M}\ev{ H_2 } = \mathcal{M} \ev{ H } \setminus \{\omega\}$. 
\begin{enumerate}
\item[\textit{i)}] ($\subseteq $): Take a maximal element $x \in \mathcal{M}\ev{ H_2 }$. Suppose, for a contradiction, that $x \notin \mathcal{M} \ev{ H } $, i,e, $x \lessdot y$, for some $y$. Since $x$ is a maximal element of $H_2$, it follows that $y\notin H_2$, i.e., $y \in H_1 = \downset[]{\omega}$. By transitivity, we get $x\leq \omega$, contradicting $x\in H_2 = H\setminus H_1$.
\item[\textit{ii)}] ($\supseteq $): Let $x\in \mathcal{M} \ev{ H } \setminus \{\omega\}$ be chosen. For any $y\in H_2$, if $x\leq y$, then $x=y$, since $y\in H$ and $x$ is maximal in $H$. 
\end{enumerate}
Second, we show that $ H = T$, where $T := H_1\circ H_2 $. The ground sets are clearly equal: $\Omega = \downset[ ]{\omega} \sqcup ( \Omega \setminus \downset[ ]{\omega} )$. Let $x,y\in \Omega$ be fixed. We claim that $x\leq_T y $ if and only if $x\leq_H y$. 
\begin{enumerate}
\item[\textit{i)}] $\ev{\Longrightarrow}$: Assume $x\leq_T y$. Then, there is a sequence $\{z_i\}_{i=0}^{k}$ with $k\geq 0$ such that $x = z_k \lessdot_T  \ldots  \lessdot_T z_0 = y$. 
\begin{enumerate}
\item Assume that $y\in H_1$. We claim that $\{ z_j\}_{j=0}^{k} \subseteq H_1$. We proceed by induction on $j$. For the base case, we have $ j = 0 $, and $z_{0} = y \in H_1 $ holds. For the inductive hypothesis, assume that $z_{j} \in H_1$, for $j\geq 0$. For the inductive step, note that $z_{j+1} \lessdot_T z_{j} $ and $z_j \in H_1$ together imply, by the definition of composition of heaps, that $z_{j+1} \in H_1$, and the claim follows. Now, for any $j= 0,\ldots,k-1$, we have $z_j,z_{j+1} \in H_1$ and $z_{j+1} \lessdot_{T} z_{j}$. Then, by the definition of composition of heaps, we get $z_{j+1} \lessdot_{H_1} z_{j}$, which implies, by the definition of $\leq_{H_1}$, that $z_{j+1} \lessdot_{H} z_{j}$. 
\item Assume that $x\in H_2$. Using a similar argument, we apply reverse induction on $j$ and show that $\{ z_j\}_{j=0}^{k} \subseteq H_2$ and $z_{j} \lessdot_{H} z_{j-1}$, for each $j=k,k-1,\ldots,1$.
\item The only remaining case to consider is $x\in H_1$ and $y\in H_2$. Let $i$ be the maximum index, such that $z_i \in H_2$. In particular, $z_{i+1} \in H_1 $. By Part (a) applied to $x$ and $z_{i+1}$, we obtain $x = z_k \lessdot_H \ldots \lessdot_H z_{i+1}$. By Part (b) applied to $z_i$ and $y$, we obtain  $z_i \lessdot_H \ldots \lessdot_H z_{0} = y $. Finally, we claim that $z_{i+1} \lessdot_H z_i$. Note that $z_{i}$ and $z_{i+1}$ are concurrent, by the definition of $\leq_T$. Since $H$ is a heap, it follows that $z_{i+1} \leq_H z_i$ or $z_{i} \leq_H z_{i+1} $. On the other hand, since $H_1$ is down-closed, $ z_{i+1} \in H_1$ and $z_i \not\in H_1$ together imply $z_{i} \not\leq_H z_{i+1} $. Hence, we obtain $z_{i+1} \leq_H z_i$. 
\end{enumerate}
Therefore, we obtain $x = z_k \lessdot_H \ldots \lessdot_H z_0 = y$, which implies $x\leq_H y$. 
\item[\textit{ii)}] $\ev{\Longleftarrow}$: Assume $x \leq_H y$. Then, there is a sequence $\{z_i\}_{i=0}^{k}$ with $k\geq 0$ such that $x = z_k \lessdot_H  \ldots \lessdot_H z_0 = y$. 
\begin{enumerate}
\item Assume that $y\in H_1$. For any $j\geq 0$, we have $z_j\leq_H y \leq_H \omega$, which implies $z_j \in H_1$. Then, for each $j$, we have $z_{j+1} \lessdot_H z_j$, and so, $z_{j+1} \lessdot_{H_1} z_j$, which implies $z_{j+1} \lessdot_{T} z_j$.
\item Assume that $x\in H_2$. As in the previous part, for any $j$, we have $x \not\leq_H \omega$ and $x \leq_H z_j$, so we infer $z_j \in H_2$. Then, for each $j$, we have $z_{j+1} \lessdot_H z_j$, and so, $z_{j+1} \lessdot_{H_2} z_j$, which implies $z_{j+1} \lessdot_{T} z_j$.
\item The only remaining case to consider is $x\in H_1$ and $y\in H_2$. Let $i$ be the maximum index such that $z_i \in H_2$. In particular, $z_{i+1} \in H_1 $. By Parts (a) and (b), we obtain $x = z_k \lessdot_T \ldots \lessdot_T z_{i+1}$ and $z_i \lessdot_T \ldots \lessdot_T z_{0} = y $. Finally, we claim that $z_{i+1} \lessdot_T z_i$. Since $H$ is a heap, it follows that $z_{i}$ and $z_{i+1}$ are concurrent. So, by the definition of composition of heaps, we obtain $z_{i+1} \lessdot_T z_{i}$. 
\end{enumerate}
Therefore, we obtain $x = z_k \lessdot_T \ldots \lessdot_T z_0 = y$, which implies $x\leq_T y$. 
\end{enumerate}
\end{proof}

\vspace{-0.5cm}
\hypertarget{proof:bijtrailspyramids}{\phantom{a}}
\begin{proof}[Proof of \Cref{cor:bij_trails_pyramids}: ]
\label{proof:bij_trails_pyramids_pf}
Let $X$ be a multi-graph, with a fixed edge $e \in E\ev{X}$. 
There is a natural correspondence between Eulerian trails of $X$ and the Eulerian trails of orientations of $X$. We define inverse functions:  
\[\begin{tikzcd}
\mathcal{W}\ev{X} \arrow[r,shift left=2pt,"\mathcal{J}"] & \arrow[l,shift left=2pt,"\mathcal{L} "]  \bigsqcup_{ D\in \mathcal{O}\ev{ X } } \mathcal{W}\ev{D}
\end{tikzcd}\]
where for each Eulerian trail $\mathbf{w} = \ev{ v_0, e_1, v_1, \ldots, v_{d-1},e_d,v_d} \in \mathcal{W}\ev{X}$, we define $\mathcal{J}\ev{ \mathbf{w} } := \mathbf{w} \in \mathcal{W}\ev{D} \text{ where } D\in \mathcal{O}\ev{X}$ is the unique orientation of $X$ with $V\ev{D} = V\ev{X}$, $E\ev{D} = E\ev{\mathbf{w}}$ and $\psi\ev{e_i} = \ev{ v_{i-1}, v_i }$, for each $i=1,\ldots,d$. Conversely, for each $D\in \mathcal{O}\ev{X}$ and $\mathbf{w}\in \mathcal{W}\ev{ D }$, we have $\mathcal{L}\ev{ \mathbf{w} } = \mathbf{w} \in \mathcal{W}\ev{X}$, i.e., $\mathcal{L}$ forgets that the edges of $\mathbf{w}$ are oriented. 
Furthermore, we have a mapping $\bigsqcup_{ D\in \mathcal{O}\ev{ X } } \mathcal{B}\ev{D} \rightarrow \mathcal{B}\ev{X}$ given by $[\mathbf{w}]_{\rho} \mapsto [\mathbf{w}]_{\rho}$, which further extends to a function $\mathcal{K}: \bigsqcup_{ D\in \mathcal{O}\ev{ X } } \mathbbm{dp}\ev{ D } \rightarrow  \mathbbm{dp}\ev{ X } $. We get a commutative diagram:
$$
\begin{tikzcd}
\bigsqcup_{ D\in \mathcal{O}\ev{X} } \mathcal{W}^{e}\ev{D} \arrow{r}{ g }  & \bigsqcup_{ D\in \mathcal{O}\ev{X} }\mathbbm{dp}^{ e }\ev{ D } \arrow{d}{ \mathcal{K} } \\
\mathcal{W}^{e}\ev{X} \arrow{r}{ h } \arrow{u}{ \mathcal{J} } & \mathbbm{dp}^{ e }\ev{ X } 
\end{tikzcd}
$$
By \Cref{thm:pyramids_containing_edge} and a routine check, the maps $\mathcal{J}$, $\mathcal{K}$ and $g$ are all bijections and $h = \mathcal{K} \circ g \circ \mathcal{J}$ (the diagram commutes) and so, $h$ is a bijection as well. 
\end{proof}

\vspace{-0.5cm}
\hypertarget{proof:bijwalkspyramidspf}{\phantom{a}}
\begin{proof}[Proof of \Cref{cor:bij_walks_pyramids}: ]
\label{proof:bij_walks_pyramids_pf}

Define a mapping $\mathcal{J}$ and its left-inverse $\mathcal{L}$:
\[
\begin{tikzcd}
\mathcal{U}\ev{ G } \arrow[r,shift left=2pt,"\mathcal{J}"] & \arrow[l,shift left=2pt,"\mathcal{L} "]  \bigsqcup_{X \in \textup{Inf}^{1}\ev{G}} \mathcal{W}\ev{X} 
\end{tikzcd}
\]
where for each $\mathbf{w} = \ev{ v_0, e_1, v_1, e_2,\ldots, e_{d-1}, v_{d-1}, e_{d}, v_0 }$, we define $\mathcal{J}\ev{ \mathbf{w} } = \ev{ v_0, 1, v_1, 2,\ldots, d-1, v_{d-1}, d, v_0 } \in \mathcal{W}\ev{X}$, such that $X = \ev{ V, E,\varphi}$ is the infragraph with $V = V\ev{G}$, $E=[d]$ and $\varphi\ev{i} = e_i$, for each $i=1,\ldots,d$. Conversely, given $X \in \textup{Inf}^{1}\ev{G}$ and $\mathbf{w} = \ev{ v_0, e_1, v_1, e_2,\ldots, v_{d-1}, e_{d}, v_0 } \in \mathcal{W}\ev{X}$, then $\mathcal{L}\ev{\mathbf{w}} =  \ev{ v_0, \varphi\ev{ e_1 }, v_1, \varphi\ev{ e_2 },\ldots, v_{d-1}, \varphi\ev{ e_d }, v_0 } $, flattening the edges. It follows that $\text{id} = \mathcal{L}\circ \mathcal{J}$. 
Furthermore, we have a mapping $\bigsqcup_{X \in \textup{Inf}^{1}\ev{G}} \mathcal{B}\ev{X}\rightarrow \widetilde{\mathcal{B}}\ev{G}$ given by $$[\ev{ v_0, e_1, v_1, e_2,\ldots, e_{d-1}, v_{d-1}, e_{d}, v_0 }]_{\rho} \mapsto [\ev{ v_0, \varphi\ev{ e_1 }, v_1, \varphi\ev{ e_2 },\ldots, v_{d-1}, \varphi\ev{ e_d }, v_0 }]$$
which further extends to a function $\mathcal{K}: \bigsqcup_{X \in \textup{Inf}^{1}\ev{G}} \mathbbm{dp}\ev{ X } \rightarrow  \mathbbm{p}\ev{ G } $. 
For each $X\in \textup{Inf}^{1}\ev{G}$, define an equivalence relation $\unsim_{\mathcal{W}\ev{X}}$ of parallelism (cf. Section \ref{sec:preliminaries_for_graphs}) on $ \mathcal{W}\ev{X} $ by:
$$ 
\ev{ v_0, e_1, v_1,\ldots, e_{d-1}, v_{d-1}, e_{d}, v_0 } \unsim_{\mathcal{W}\ev{X}} \ev{ u_0, f_1, u_1,\ldots, f_{d-1}, u_{d-1}, f_{d}, u_0 } \leftrightarrow v_i = u_i \text{ and } e_j \unsim f_j \text{ for each $i,j$.}
$$
Note that $[w]_{\unsim_{\mathcal{W}\ev{X}}} = M_X$, for each Eulerian trail $\mathbf{w}$ of $X$.  
Furthermore, we define an equivalence relation $\unsim_{ \mathcal{B}\ev{X} }$ of parallelism on $\mathcal{B}\ev{X}$ by:
$$ 
\beta_1 \unsim_{\mathcal{B}\ev{X}} \beta_2 \text{ if and only if } \beta_1 = [\mathbf{w}_1],\ \beta_2 = [\mathbf{w}_2], \ \text{ and } \mathbf{w}_1 \unsim_{\mathcal{W}\ev{X}} \mathbf{w}_2 \text{ for some } \mathbf{w}_1, \mathbf{w}_2 \in \mathcal{W}\ev{X}
$$
which further extends to an equivalence relation on $\mathbbm{dp}\ev{ X }$, such that $|[P]_{\unsim_{\mathcal{B}\ev{X}}}| = M_X$, for each decomposition pyramid $P\in \mathbbm{dp}\ev{ X }$. 
We obtain mappings $\widetilde{h},\widetilde{\mathcal{J}},\widetilde{\mathcal{L}}$ and $\widetilde{\mathcal{K}}$, that are well-defined, independent of the choice of representative of an equivalence class in their respective domains, where $ \begin{tikzcd}
\mathcal{U}\ev{G}\arrow[r,shift left=2pt,"\widetilde{\mathcal{J}}"] & \arrow[l,shift left=2pt,"\widetilde{\mathcal{L}} "]  \bigsqcup_{X \in \textup{Inf}^{1}\ev{G}} \mathcal{W}\ev{X} / \unsim_{\mathcal{W}\ev{X}}
\end{tikzcd}$ is a one-to-one correspondence. We obtain a commutative diagram, 
$$
\begin{tikzcd}
\bigsqcup_{X \in \textup{Inf}^{1}\ev{G}} \mathcal{W}^{ [e]_{\unsim} }\ev{X} / \unsim_{\mathcal{W}\ev{X}} \arrow[xshift=5pt]{d}{ \widetilde{\mathcal{L}} }   \arrow{r}{ \widetilde{h} }  & \bigsqcup_{X \in \textup{Inf}^{1}\ev{G}} \mathbbm{dp}^{ [e]_{\unsim} }\ev{ X } / \unsim_{\mathcal{B}\ev{X} }  \arrow{d}{ \widetilde{\mathcal{K}} } \\
\mathcal{U}^{e}\ev{G} \arrow{u}{ \widetilde{\mathcal{J}} } \arrow{r}{ q } & \mathbbm{p}^{ e }\ev{ G } 
\end{tikzcd}
$$
where $\mathcal{W}^{ [e]_{\unsim} }\ev{X}$ is the set of equivalence classes of Eulerian trails such that the last edge is parallel to $e$, and $\mathbbm{dp}^{ [e]_{\unsim} }\ev{ X }$ is the set of equivalence classes of decomposition pyramids such that the maximal piece contains $e$. 
By \Cref{cor:bij_trails_pyramids} and a routine check, the maps $\widetilde{h}, \widetilde{\mathcal{J}}, \widetilde{\mathcal{K}}$ are bijective and $q = \widetilde{\mathcal{K}} \circ \widetilde{h} \circ \widetilde{\mathcal{J}}$, which implies that $q$ is bijective as well.
\end{proof}

\vspace{-0.5cm}
\hypertarget{proof:interestingidentitypf}{\phantom{a}}
\begin{proof}[Proof of \Cref{cor:interesting_identity}: ]
\label{proof:interesting_identity_pf_page}
First, we have,
\begin{align*}
-\log\ev{ \widetilde{\phi}_G\ev{t} } & = \sum_{P\in \mathbbm{p}\ev{G }} \frac{1}{|P|} t^{-|E\ev{P}|} & \text{ by \Cref{cor:all_pyramids}}
\end{align*}
On the other hand, 
\begin{align*}
-\log\ev{ \widetilde{\phi}_G\ev{t} } & = \sum_{\mathbf{w}\in \mathcal{U}\ev{G}} \frac{1}{|\mathbf{w}|} t^{-|\mathbf{w}|} & \text{ by \Cref{thm:log_and_all_walks}}\\
& =  \sum_{u\in V\ev{G}} \sum_{\mathbf{w}\in \mathcal{U}^{u}\ev{G}} \frac{1}{|\mathbf{w}|} t^{-|\mathbf{w}|} = \sum_{u\in V\ev{G}} \sum_{ \substack{P \in \mymathbb{p} \ev{ G } \\ u\in \mathcal{M}\ev{P} }  } \frac{1}{|E\ev{P}|} t^{-|E\ev{P}|} & \text{ by \Cref{cor:bij_walks_pyramids}} \\ 
& = \sum_{ P \in \mymathbb{p} \ev{ G }  } \frac{ |V\ev{\mathcal{M}\ev{P}}| }{|E\ev{P}|} t^{-|E\ev{P}|} 
\end{align*}
\end{proof}

\hypertarget{proof:nonmultiplicativepf}{\phantom{a}}
\begin{proof}[Proof of \Cref{lem:non_multiplicative}: ]
\label{proof:non_multiplicative_pf_page}
From now on, we write $\deg_{D}\ev{ u }$ to mean the in-degree of $u\in V\ev{D}$ for a digraph $D$. Let $X_1$ and $X_2$ be connected Veblen hypergraphs with $V\ev{X_1} = [n]$ and $V\ev{X_2} = \{ n, \ldots, a\}$, where $ 1\leq n \leq a $, $u=n$ and $V\ev{X_1} \cap V\ev{X_2} = \{u\}$. Note that $\mathfrak{s}_u\ev{X} = \mathfrak{s}_u\ev{X_1} + \mathfrak{s}_u\ev{X_2}$, as the intersection of $X_1$ and $X_2$ is a singleton. For each pair of rootings $\ev{ \mathbf{R}_1, \mathbf{R}_1} \in \overline{\mathfrak{R}}\ev{X_1} \times \overline{\mathfrak{R}}\ev{X_2}$, we can obtain a rooting of $X = X_1\cup X_2$, as follows: Let 
$$ \ev{S_{1}\ev{e_{11}},\ldots,S_{1}\ev{e_{1r_1}} , \ldots ,S_{n}\ev{e_{n1}},\ldots,S_{n}\ev{e_{nr_n}}} := R_1 \in \mathbf{R}_1 $$
be a representative, where $i$ is the root of $\mathfrak{s}_i\ev{ X_1 } := r_i $ many stars for $i = 1,\ldots, n$ and let
$$ \ev{S_{n}\ev{f_{n 1}},\ldots,S_{n}\ev{f_{n s_n}}  , \ldots , S_{a}\ev{f_{a 1}},\ldots,S_{a}\ev{f_{a s_a}}} := R_2 \in \mathbf{R}_2 $$
be a representative, where $j$ is the root of $\mathfrak{s}_j\ev{ X_2 } := s_j$ many stars for $j = n,\ldots, a$. Then, we can obtain a rooting:
$$ 
R_1 \oplus R_2:= \ev{S_{1}\ev{e_{11}},\ldots,S_{1}\ev{e_{1r_1}}, \ldots, S_{ n }\ev{e_{n 1}},\ldots, S_{ n }\ev{e_{n r_n}}, S_{ n }\ev{f_{n 1}},\ldots, S_{ n }\ev{f_{n s_n}}, \ldots , S_{a}\ev{f_{a 1}}, \ldots, S_{a}\ev{f_{a s_a}} } 
$$ 
which is a representative of the equivalence class $ [ R_1 \oplus R_2 ]_{\unsim} \in \overline{\mathfrak{R}}\ev{X_1\cup X_2} $. Hence, we obtain a well-defined bijection:
\begin{align*}
\overline{\mathfrak{R}}\ev{X_1} \times \overline{\mathfrak{R}}\ev{X_2} & \leftrightarrow \overline{\mathfrak{R}}\ev{X_1\cup X_2} \\ 
\ev{  \mathbf{R}_1 ,  \mathbf{R}_2 } & \mapsto \mathbf{R}_1 \oplus \mathbf{R}_2 := [ R_1 \oplus R_2 ]_{\unsim} \text{ for any } R_1 \in  \mathbf{R}_1,\ R_2 \in  \mathbf{R}_2
\end{align*}
such that:
\begin{align*}
& | \mathbf{R}_1 \oplus \mathbf{R}_2 | = \dfrac{  \mathfrak{s}_u\ev{X} ! }{ \Gamma_{\mathbf{R}_1} \cdot \Gamma_{\mathbf{R}_2} }  = \dfrac{  (\mathfrak{s}_u\ev{X_1} + \mathfrak{s}_u\ev{X_2})! }{ \mathfrak{s}_u\ev{X_1}! \mathfrak{s}_u\ev{X_2}! }  \cdot \dfrac{ \mathfrak{s}_u\ev{X_1}! }{ \Gamma_{\mathbf{R}_1} } \cdot \dfrac{ \mathfrak{s}_u\ev{X_2}! }{ \Gamma_{\mathbf{R}_2} } = \binom{ \mathfrak{s}_u\ev{X} }{\mathfrak{s}_u\ev{X_1},\mathfrak{s}_u\ev{X_2} }\cdot | \mathbf{R}_1 | \cdot | \mathbf{R}_2 |
\end{align*}
Therefore, we obtain,
\begin{align*}
& C_X = \sum_{ \mathbf{R}  \in \overline{\mathfrak{R}}\ev{X} } |\mathbf{R}|\cdot \dfrac{ \tau_{D_\mathbf{R} } }{ \prod_{v\in D_\mathbf{R} } \deg_{D_\mathbf{R}}\ev{v}} = \sum_{ \substack{ \mathbf{R}_1\in \overline{\mathfrak{R}}\ev{X_1} \\ \mathbf{R}_2\in \overline{\mathfrak{R}}\ev{X_2} } }  | \mathbf{R}_1 \oplus \mathbf{R}_2 |  \cdot \dfrac{ \tau_{D_{\mathbf{R}_1} \cup D_{\mathbf{R}_2} } }{ \displaystyle\prod_{v\in D_{\mathbf{R}_1} \cup D_{\mathbf{R}_2} } \deg_{D_{\mathbf{R}_1} \cup D_{\mathbf{R}_2}}\ev{v} } \\  
& =  \sum_{ \substack{ \mathbf{R}_1\in \overline{\mathfrak{R}}\ev{X_1} \\ \mathbf{R}_2\in \overline{\mathfrak{R}}\ev{X_2} } } | \mathbf{R}_1 \oplus \mathbf{R}_2 | \cdot \dfrac{  \tau_{D_{\mathbf{R}_1}} \tau_{D_{\mathbf{R}_2} }}{ \deg_{D_{\mathbf{R}_1} \cup D_{\mathbf{R}_2}}\ev{u} \displaystyle\prod_{\substack{v\in D_{\mathbf{R}_1} \\ v\neq u} } \deg_{D_{\mathbf{R}_1}}\ev{v} \cdot \displaystyle\prod_{\substack{u\in D_{\mathbf{R}_2} \\ v\neq u } } \deg_{D_{\mathbf{R}_2}}\ev{u} } \\
& =  \sum_{ \substack{ \mathbf{R}_1\in \overline{\mathfrak{R}}\ev{X_1} \\ \mathbf{R}_2\in \overline{\mathfrak{R}}\ev{X_2} } } | \mathbf{R}_1 \oplus \mathbf{R}_2 | \dfrac{ \deg_{D_{\mathbf{R}_1}}\ev{u} \deg_{D_{\mathbf{R}_2}}\ev{u}  }{ \deg_{D_{\mathbf{R}_1} \cup D_{\mathbf{R}_2}}\ev{u} } \dfrac{  \tau_{D_{\mathbf{R}_1}} }{  \displaystyle\prod_{ u\in D_{\mathbf{R}_1}  } \deg_{D_{\mathbf{R}_1}}\ev{v} } \dfrac{  \tau_{D_{\mathbf{R}_2} }}{ \displaystyle\prod_{ u\in D_{\mathbf{R}_2} } \deg_{D_{\mathbf{R}_2}}\ev{u} } \\ 
& =  \sum_{ \substack{ \mathbf{R}_1\in \overline{\mathfrak{R}}\ev{X_1} \\ \mathbf{R}_2\in \overline{\mathfrak{R}}\ev{X_2} } } \binom{  \mathfrak{s}_u\ev{X} }{\mathfrak{s}_u\ev{X_1}, \mathfrak{s}_u\ev{X_2}}  \cdot | \mathbf{R}_1 | \cdot | \mathbf{R}_2 |  \dfrac{ \deg_{D_{\mathbf{R}_1}}\ev{u} \deg_{D_{\mathbf{R}_2}}\ev{u}  }{ \deg_{D_{\mathbf{R}_1} \cup D_{\mathbf{R}_2}}\ev{u} } \dfrac{  \tau_{D_{\mathbf{R}_1}} }{  \displaystyle\prod_{ u\in D_{\mathbf{R}_1}  } \deg_{D_{\mathbf{R}_1}}\ev{v} } \dfrac{  \tau_{D_{\mathbf{R}_2} }}{ \displaystyle\prod_{ u\in D_{\mathbf{R}_2} } \deg_{D_{\mathbf{R}_2}}\ev{u} }\\ 
& =  \sum_{ \substack{ \mathbf{R}_1\in \overline{\mathfrak{R}}\ev{X_1} \\ \mathbf{R}_2\in \overline{\mathfrak{R}}\ev{X_2} } } \binom{  \mathfrak{s}_u\ev{X} }{\mathfrak{s}_u\ev{X_1}, \mathfrak{s}_u\ev{X_2}}  \cdot | \mathbf{R}_1 | \cdot | \mathbf{R}_2 | \dfrac{ \ev{k-1} \mathfrak{s}_u\ev{X_1} \ev{k-1} \mathfrak{s}_u\ev{X_1}  }{ \ev{k-1} \mathfrak{s}_u\ev{X} }  \dfrac{  \tau_{D_{\mathbf{R}_1}} }{  \displaystyle\prod_{ u\in D_{\mathbf{R}_1}  } \deg_{D_{\mathbf{R}_1}}\ev{v} } \dfrac{  \tau_{D_{\mathbf{R}_2} }}{ \displaystyle\prod_{ u\in D_{\mathbf{R}_2} } \deg_{D_{\mathbf{R}_2}}\ev{u} }\\ 
& \hspace{2cm}\text{ by \Cref{rmk:previous_work} } \\ 
& = \ev{k-1} \dfrac{ \mathfrak{s}_u\ev{X_1}\mathfrak{s}_u\ev{X_1}  }{ \mathfrak{s}_u\ev{X} } \binom{  \mathfrak{s}_u\ev{X} }{\mathfrak{s}_u\ev{X_1}, \mathfrak{s}_u\ev{X_2}} \cdot \sum_{\mathbf{R}_1 \in \overline{\mathfrak{R}}\ev{X_1} } | \mathbf{R}_1 |  \dfrac{  \tau_{D_{\mathbf{R}_1} }}{ \displaystyle\prod_{ u\in D_{\mathbf{R}_1} } \deg_{D_{\mathbf{R}_1}}\ev{u} }   \sum_{\mathbf{R}_2 \in \overline{\mathfrak{R}}\ev{X_2}} | \mathbf{R}_2 | \dfrac{  \tau_{D_{\mathbf{R}_2} }}{ \displaystyle\prod_{ u\in D_{\mathbf{R}_2} } \deg_{D_{\mathbf{R}_2}}\ev{u} }\\ 
& = \ev{k-1} \dfrac{ \mathfrak{s}_u\ev{X_1}\mathfrak{s}_u\ev{X_1}  }{ \mathfrak{s}_u\ev{X} } \binom{  \mathfrak{s}_u\ev{X} }{\mathfrak{s}_u\ev{X_1}, \mathfrak{s}_u\ev{X_2}} \cdot C_{X_1} \cdot  C_{X_2} 
\end{align*}
\end{proof}

\vspace{-0.5cm}
\hypertarget{proof:intersectinginfraspf}{\phantom{a}}
\begin{proof}[Proof of \Cref{prop:intersecting_infras}: ]
\label{proof:intersecting_infras_pf_page}
Since $\deg_{X_1}\ev{ u } = \deg_{X_2}\ev{ u } = k$, it follows that $ \widetilde{\mathcal{S}} \ev{ X } = \mathcal{D}_1 \sqcup \mathcal{D}_2 $, where $\mathcal{D}_j = \{ \mathbf{S} = [\{ A_1, \ldots, A_m\}]_{\unsim} \in  \widetilde{\mathcal{S}} \ev{ X } : | \{ i : u\in A_i\} | = j \} $, for $j=1,2$. Therefore,
\begin{align*}
& w_n\ev{ X } t^{|E\ev{X}|} = - \Delta\ev{ n,X }  = \sum_{ \mathbf{S} \in \widetilde{\mathcal{S}} \ev{ X }} \ev{ -\ev{k-1}^{n} }^{ c\ev{ \mathbf{S} } + 1 } \dfrac{1}{\alpha_{\mathbf{S}}} C_{\mathbf{S} } \\ 
& = \sum_{ \mathbf{S} \in \mathcal{D}_1 } \ev{ -\ev{k-1}^{n} }^{ c\ev{ \mathbf{S} } + 1 } \dfrac{1}{\alpha_{\mathbf{S}}} C_{\mathbf{S} } +  \sum_{ \mathbf{S} \in \mathcal{D}_2 } \ev{ -\ev{k-1}^{n} }^{ c\ev{ \mathbf{S} } + 1 } \dfrac{1}{\alpha_{\mathbf{S}}} C_{\mathbf{S} } 
\end{align*}
We have a bijection,
\begin{align*}
\widetilde{\mathcal{S}} \ev{ X_1 } \times \widetilde{\mathcal{S}} \ev{ X_2 } &\rightarrow \mathcal{D}_1 \\
\ev{ \mathbf{S}_1 ,  \mathbf{S}_2}  & \mapsto \mathbf{S}_1\odot \mathbf{S}_2:= [ \{ A_1,\ldots, A_m\cup B_m, \ldots, B_{m+r-1}\} ]_{\unsim}
\end{align*}
where $\mathbf{S}_1 = [  \{ A_1,\ldots, A_m \} ]_{\unsim} $, $\mathbf{S}_2 = [ \{ B_m,\ldots, B_{m+r-1} \} ]_{\unsim} $ and $u\in A_m \cap B_m$. As in the proof of \Cref{prop:cover_mult}, the following is also a bijection:
\begin{align*}
\widetilde{\mathcal{S}} \ev{ X_1 } \times \widetilde{\mathcal{S}} \ev{ X_2 } &\rightarrow \mathcal{D}_2 \\
\ev{ \mathbf{S}_1 ,  \mathbf{S}_2}  & \mapsto \mathbf{S}_1\sqcup \mathbf{S}_2:= [ \{ A_1,\ldots, A_m, B_m, \ldots, B_{m+r-1}\} ]_{\unsim}
\end{align*}
where $\mathbf{S}_1 = [\{ A_1,\ldots, A_m \} ]_{\unsim}$, $\mathbf{S}_2 =[  \{ B_m,\ldots, B_{m+r-1} \} ]_{\unsim} $ and $u\in A_m \cap B_m$. Therefore, 
\begin{align*}
& w_n\ev{ X } = - \sum_{ \substack{\mathbf{S}_1 \in \widetilde{\mathcal{S}} \ev{ X_1 } \\ \mathbf{S}_2 \in \widetilde{\mathcal{S}} \ev{ X_2 }} } \ev{-\ev{k-1}^{n}}^{c\ev{ \mathbf{S}_1 \odot \mathbf{S}_2 }} \dfrac{1}{ \alpha_{  \mathbf{S}_1 \odot \mathbf{S}_2  } } C_{ \mathbf{S}_1 \odot \mathbf{S}_2 } - \sum_{ \substack{\mathbf{S}_1 \in \widetilde{\mathcal{S}} \ev{ X_1 } \\ \mathbf{S}_2 \in \widetilde{\mathcal{S}} \ev{ X_2 }} } \ev{-\ev{k-1}^{n}}^{c\ev{ \mathbf{S}_1 \sqcup \mathbf{S}_2 }}  \dfrac{1}{ \alpha_{  \mathbf{S}_1 \sqcup \mathbf{S}_2  } } C_{ \mathbf{S}_1 \sqcup \mathbf{S}_2 } \\ 
& = - \sum_{ \substack{\mathbf{S}_1 \in \widetilde{\mathcal{S}} \ev{ X_1 } \\ \mathbf{S}_2 \in \widetilde{\mathcal{S}} \ev{ X_2 }} } \ev{-\ev{k-1}^{n}}^{c\ev{ \mathbf{S}_1 } + c\ev{ \mathbf{S}_2 } - 1 } \dfrac{\ev{k-1}}{ \alpha_{  \mathbf{S}_1 }  \alpha_{  \mathbf{S}_2 }  }   C_{ \mathbf{S}_1 }  C_{ \mathbf{S}_2 } - \sum_{ \substack{\mathbf{S}_1 \in \widetilde{\mathcal{S}} \ev{ X_1 } \\ \mathbf{S}_2 \in \widetilde{\mathcal{S}} \ev{ X_2 }} } \ev{-\ev{k-1}^{n}}^{c\ev{ \mathbf{S}_1 } + c\ev{ \mathbf{S}_2 }  } \dfrac{1}{ \alpha_{  \mathbf{S}_1 }  \alpha_{  \mathbf{S}_2 }  }  C_{ \mathbf{S}_1 }  C_{ \mathbf{S}_2 } \\ 
& \text{ by \Cref{lem:non_multiplicative}} \\ 
& =  \ew{ \dfrac{k-1}{\ev{ k-1 }^{n}}  - 1  } \sum_{ \substack{\mathbf{S}_1 \in \widetilde{\mathcal{S}} \ev{ X_1 } \\ \mathbf{S}_2 \in \widetilde{\mathcal{S}} \ev{ X_2 }} } \ev{-\ev{k-1}^{n}}^{c\ev{ \mathbf{S}_1 } + c\ev{ \mathbf{S}_2 }  } \dfrac{1}{ \alpha_{  \mathbf{S}_1 } \alpha_{  \mathbf{S}_2 }  } C_{ \mathbf{S}_1 }  C_{ \mathbf{S}_2 }  \\
& = \ew{ \dfrac{ k-1 }{\ev{ k-1}^{n}}  - 1  }  \cdot \ew{ - \sum_{ \mathbf{S}_1 \in \widetilde{\mathcal{S}} \ev{ X_1 } } \ev{-\ev{k-1}^{n}}^{c\ev{ \mathbf{S}_1 } } \dfrac{1}{ \alpha_{  \mathbf{S}_1  } } C_{ \mathbf{S}_1 } } \cdot \ew{ - \sum_{ \mathbf{S}_2 \in \widetilde{\mathcal{S}} \ev{ X_2 } } \ev{-\ev{k-1}^{n}}^{c\ev{ \mathbf{S}_2 }  } \dfrac{1}{ \alpha_{  \mathbf{S}_2  } } C_{ \mathbf{S}_2 } } \\
& = \ew{ \ev{ k-1}^{ 1 - n }  - 1  } \cdot w_n\ev{ X_1 } \cdot w_n\ev{ X_2 }
\end{align*}
\end{proof}

\end{document}